\documentclass{amsart}

\usepackage{hyperref}
\usepackage{epsfig}
\usepackage{flafter}
\usepackage{comment}

\hypersetup{
    colorlinks=true,    
    linkcolor=blue,          
    citecolor=blue,      
    filecolor=blue,      
    urlcolor=blue           
}

\usepackage{tikz}
\usetikzlibrary{graphs,positioning,arrows,shapes.misc,decorations.pathmorphing}
\tikzstyle{vertex}=[
    draw,
    circle,
    fill=black,
    inner sep=1pt,
    minimum width=5pt,
]
\usetikzlibrary{matrix,arrows,decorations.pathmorphing}
\usepackage{tikz-cd}
\usepackage[position=top]{subfig}
\usepackage{color}

\usepackage{amsmath}
\usepackage{amsfonts,graphics,amsthm,amsfonts,amscd,latexsym}
\usepackage[abbrev,alphabetic,backrefs,initials,nobysame]{amsrefs}
\usepackage{amssymb}
\usepackage[all]{xy}
\usepackage{mathtools}

\calclayout

\usetikzlibrary{decorations.pathmorphing}
\tikzstyle{printersafe}=[decoration={snake,amplitude=0pt}]

\DeclareFontFamily{U}{wncy}{}
\DeclareFontShape{U}{wncy}{m}{n}{<->wncyr10}{}
\DeclareSymbolFont{mcy}{U}{wncy}{m}{n}
\DeclareMathSymbol{\Sh}{\mathord}{mcy}{"58} 

\newcommand{\rank}{\operatorname{rank}}

\newcommand{\vol}{\operatorname{vol}}

\newcommand{\rar}{\rightarrow}	
\newcommand{\drar}{\dashrightarrow}

\renewcommand{\qq}{\mathbb{Q}}
\newcommand{\zz}{\mathbb{Z}}
\newcommand{\nn}{\mathbb{N}}
\newcommand{\rr}{\mathbb{R}}

\newcommand{\setbaseslf}{\mathfrak B^{d-1}_{\text{log Fano}}}
\newcommand{\logfanoeps}{\mathfrak D^{d-1, \epsilon}_{\text{log Fano}}}
\newcommand{\logfanoepszero}{\mathfrak{D}^{d-1, , \epsilon_0}_{\text{log Fano}}}

\newcommand{\setbases}{\mathfrak B^{2, \celliptic}_{\text{LCY, RC}}}
\newcommand{\setenr}{\mathfrak B^2_{Enr, DV}}

\newcommand{\familykodtwo}{\mathfrak F^{3,v,C}_{\kappa=2}}
\newcommand{\familyellcy}{\mathfrak F^3_{\text{CY, ell}}}
\newcommand{\familyellcyrat}{\mathfrak F^3_{\text{CY, ell, rat}}}
\newcommand{\familyellcyenr}{\mathfrak F^3_{\text{CY, ell, Enr}}}
\newcommand{\familyellfano}{\mathfrak F^{d,  C}_{\text{CY, ell, LF}}}
\newcommand{\specialpoint}{p_{S_{D_{\deg=d}, \eta} - D \subs 0,\eta.}}

\newcommand{\celliptic}{\mathcal C_{ell}}

\newcommand{\nsr}{{\rm N}^1_{\mathbb R}}

\newcommand{\nsdr}{{\rm N}_{1, \mathbb R}}

\DeclareMathOperator{\Supp}{Supp}

\def\O#1.{\mathcal {O}_{#1}}			
\def\pr #1.{\mathbb P^{#1}}				
\def\af #1.{\mathbb A^{#1}}			
\def\ses#1.#2.#3.{0\to #1\to #2\to #3 \to 0}	
\def\xrar#1.{\xrightarrow{#1}}			
\def\K#1.{K_{#1}}						
\def\bA#1.{\mathbf{A}_{#1}}			
\def\bM#1.{\mathbf{M}_{#1}}
\def\bN#1.{\mathbf{N}_{#1}}
\def\bL#1.{\mathbf{L}_{#1}}				
\def\bB#1.{\mathbf{B}_{#1}}				
\def\bK#1.{\mathbf{K}_{#1}}			
\def\subs#1.{_{#1}}					
\def\sups#1.{^{#1}}						
\def \neone#1.{\overline{{\rm NE}(#1)}}

  		\newtheorem{theorem}{Theorem}[section]
  		\newtheorem{lemma}[theorem]{Lemma}
  		\newtheorem{proposition}[theorem]{Proposition}
  		\newtheorem*{conconj}{Cone conjecture}

  		\newtheorem{corollary}[theorem]{Corollary}

\theoremstyle{definition}
  		
  		\newtheorem{notation}[theorem]{Notation}
  		\newtheorem{definition}[theorem]{Definition}
  		\newtheorem{example}[theorem]{Example}

        \newtheorem{remark}[theorem]{Remark}

\theoremstyle{remark}

\numberwithin{equation}{section}

\makeatletter
\@namedef{subjclassname@2020}{
  \textup{2020} Mathematics Subject Classification}
\makeatother

\begin{document}

\title{Boundedness of elliptic Calabi--Yau threefolds}


\author[S.~Filipazzi]{Stefano Filipazzi}
\address{
EPFL, SB MATH CAG, MA C3 625 (B\^{a}timent MA), Station 8, CH-1015 Lausanne, Switzerland
}
\email{stefano.filipazzi@epfl.ch}

\author[C.D.~Hacon]{Christopher D. Hacon}
\address{
Department of Mathematics, University of Utah, Salt Lake City, UT 84112, USA
}
\email{hacon@math.utah.edu}

\author[R.~Svaldi]{Roberto Svaldi}
\address{Dipartimento di Matematica ``F. Enriques'', Universit\`a degli Studi di Milano, Via Saldini 50, 20133 Milano (MI), Italy}
\email{roberto.svaldi@unimi.it}


\begin{abstract}
We show that elliptic Calabi--Yau threefolds form a bounded family.
We also show that the same result holds for minimal terminal threefolds of Kodaira dimension $2$, upon fixing the rate of growth of pluricanonical forms and the degree of a multisection of the Iitaka fibration.
Both of these hypotheses are necessary to prove the boundedness of such a family.
\end{abstract}

\thanks{
SF was partially supported by ERC starting grant \#804334.
CH was partially supported by NSF research grants no:  DMS-1952522, DMS-1801851 and by a grant from the Simons Foundation; Award Number: 256202.
RS gratefully acknowledges support from the European Union's Horizon 2020 research and innovation programme under the Marie Sk\l{}odowska-Curie grant agreement No. 842071 and from the ``Programma per giovani ricercatori Rita Levi Montalcini'' of the Italian Ministry of University and Research (MUR).
RS is a member of GNSAGA of INDAM}

\maketitle

\setcounter{tocdepth}{1}

\tableofcontents

\section{Introduction}
Throughout this paper, we work over the field of complex numbers $\mathbb C$.

Normal projective varieties with numerically trivial canonical bundle (in short, $K$-trivial varieties) and mild singularities are one of the fundamental building blocks in the birational classification of projective varieties and they play a prominent role in many areas of research. 
It is well known that their birational geometry is rather rich and subtle, and many phenomena in this context are yet to be fully understood.
Among $K$-trivial varieties, an important and rich class that still defies our understanding is given by Calabi--Yau varieties, i.e., projective varieties $X$ with $\qq$-factorial terminal singularities, $K_X \sim 0$ and $h^i(\mathcal O _X)=0$ for $0< i<\dim X$.

A fundamental and long-standing question, originally due to M. Reid and S.-T. Yau, see, for example,~\cites{reid, Yau}, is whether Calabi--Yau threefolds have finitely many topological types.
From the point of view of birational geometry, one could try to answer affirmatively the above question by showing that Calabi--Yau threefolds are parametrized by finitely many algebraic families of deformations.
In general, $K$-trivial varieties certainly do not form finitely many algebraic families: in dimension 3, it suffices to consider, for example, the case of products of K3 surfaces and elliptic curves -- although this example does not contradict the conjecture on the finiteness of topological types in a fixed dimension.
Hence, we cannot drop the condition on the vanishing of the middle cohomology of the structure sheaf.
An interesting and important result towards a definitive answer to the question posed by Reid and Yau is due to Gross~\cite{Gro94}:
he showed that there exist finitely many projective families $\mathcal X_i \rar \mathcal S_i \rar T_i$ over finite type schemes $T_i$ such that for any elliptic Calabi--Yau fibration $f \colon X \rar S$, whose base $S$ is a rational surface, there is a closed point $t$ in some $T_i$ such that $X$ (resp. $S$) is birationally isomorphic to the fiber $\mathcal X_t$ (resp. $\mathcal S_t$) over $t$ and these birational isomorphism can be chosen so that they identify $f$ with the induced fibration $\mathcal X_t \rar \mathcal S_t$.
Even better, it is not hard to show that the birational map $X \dashrightarrow \mathcal X_t$ is an isomorphism in codimension 1, and hence can be decomposed into a finite sequence of flops.
We summarize this property by saying that elliptic Calabi--Yau threefolds form a bounded family modulo flops.
Recently, Wilson~\cites{Wil17, Wil20} has proven some new results in the context of boundedness of Calabi--Yau threefolds at large.

The class of elliptic Calabi--Yau threefolds is of central importance in the study of Calabi--Yau threefolds in general:
indeed, it is expected, based on known examples, that Calabi--Yau threefolds of large Picard rank are always elliptically fibered, perhaps after flopping a finite number of curves.
Thus, this approach may eventually show that Calabi--Yau threefolds of large Picard rank have finitely many topological types.
When the base $S$ of an elliptic Calabi--Yau $f \colon X\to S$ is not rational, then it is birational to a (possibly singular) Enriques surface and $f$ is isotrivial, see~\cite{Gr91}*{Theorem~3.1}. 
Thus, the birational geometry of these fibrations is well understood.
Even better, by work of Koll\'ar and Larsen \cite{KL09}*{Theorem 14}, it is known that $X$ becomes a product of a K3 or Abelian surface with an elliptic curve, after a quasi-\'etale cover, see also~\cite{Nak88}*{Appendix}.

Since Calabi--Yau threefolds may have infinitely many models that are isomorphic in codimension 1, it is not clear whether the result of Gross implies the boundedness of topological types for these Calabi--Yau varieties.
However, another celebrated conjecture, the Kawamata--Morrison Conjecture, predicts that the isomorphism types of such models are just finitely many distinct ones.
Kawamata \cite{Kaw97} proved a weaker version of this conjecture in the elliptically fibered case: 
given an elliptic threefold $f \colon X\to S$, he showed that, up to isomorphism over $S$, there are only finitely many models of $f$ over $S$ isomorphic in codimension 1, cf. Theorem~\ref{thm KM conj general}.
Hence, Kawamata's result offers a first hint towards proving boundedness of the topological types elliptic Calabi--Yau varieties starting from Gross's theorem.

We give a complete and affirmative answer to Yau's question for elliptic Calabi--Yau threefolds.
In this paper, an elliptic Calabi--Yau threefold will be a terminal 
$\mathbb Q$-factorial 
projective variety 
$X$ 
with 
$K_X \sim 0$ 
and
$H^i(X, \mathcal O_X)=0$
for 
$i=1,2$,
which is moreover endowed with a morphism with connected fibers
$f \colon X \to S$
of relative dimension 1 -- which immediately implies, as 
$K_X\sim 0$
that the generic fiber is a curve of genus $1$.
We do not require any further assumptions on the morphism $f$, besides that on its relative dimension.
The surface $S$ is only assumed to be normal;
then, the canonical bundle formula and the assumptions on the singularities of $X$ immediately imply that $S$ has klt singularities.

\begin{theorem} 
\label{thm intro}
The set $\familyellcy$ of elliptic Calabi--Yau threefolds forms a bounded family.
\end{theorem}

Our result includes also the case of those elliptic fibrations whose base is non rational; in such case, the bases of the fibration is a surface with at worst Du Val singularities whose minimal resolution is an Enriques surface.
Following the philosophy introduced above, we show that there exist a finite number of algebraic families such that any elliptic Calabi--Yau threefold appears as the fiber of one of the families.

Theorem~\ref{thm intro} immediately yields the following important corollary proving the finiteness of the topological types of elliptic Calabi--Yau threefolds, answering a classical question in string theory, see \cite{g23} for a detailed account of the consequences of boundedness in string theory.

\begin{corollary}
\label{cor.ell.top}
There are only finitely many topological types of elliptic Calabi--Yau threefolds.
\end{corollary}

Most of the methods developed to tackle Theorem \ref{thm intro} can be used to study the boundedness of elliptically fibered varieties in general.
For instance, they naturally apply to minimal $n$-folds $X$ with $\kappa(X)=n-1$, as their Iitaka fibration $f \colon X \rar Y$ is an elliptic fibration.
This circle of ideas has been explored by the first-named author in \cite{Fil20}, where necessary and sufficient conditions for boundedness modulo flops are settled.
By further exploring the methods of the proof of Theorem \ref{thm intro}, we are also able to improve the criteria in \cite{Fil20} to criteria for honest boundedness in the case of threefolds of Kodaira dimension 2.
Let us recall that for a divisor $D$ on a normal algebraic variety, we define ${\rm vol}_{k}(D)= \lim_{m \to \infty} \frac{h^0(X, mD)}{\frac{m^k}{k!}}$ and this value is strictly positive (and finite) exactly when the Iitaka dimension of $D$ is $k$.

\begin{theorem} 
\label{thm intro2}
Fix a positive integer $d$ and a positive real number $v$. 
The set 
\[
\familykodtwo
\coloneqq 
\left\{
X \
\middle \vert \
\begin{array}{l}
\text{$X$ is a projective $\qq$-factorial terminal threefold}, \ \text{$K_X$ is nef}, \kappa (K_X)=2,
\\
{\rm vol}_{2}(K_X)=v, 
\text{and the Iitaka fibration of $X$ admits a degree $C$ rational multisection}
\end{array}
\right\}
\] 
forms a bounded family.
 \end{theorem}

We stress that the conditions in Theorem \ref{thm intro2} are all necessary.
Indeed, as discussed in \cite{Fil20}*{\S~3}, given a family $\pi \colon \mathcal{X} \rar T$ of minimal $n$-folds of Kodaira dimension $n-1$, up to a stratification of $T$, their Iitaka fibrations deforms along the family, that is, $\pi$ factors as $\mathcal{X} \rar \mathcal{Y} \rar T$ and for any $t \in T$, $\mathcal X_t \rar \mathcal Y_t$ is the Iitaka fibration of $\mathcal X_t$.
Consequently, up to a further stratification, Kodaira's formula for the canonical bundle of $\mathcal{X}_t \rar \mathcal{Y}_t$ is obtained by restriction of the formula for $\mathcal{X} \rar \mathcal{Y}$.
Furthermore, up to an additional stratification, a rational multisection of $\mathcal{X} \rar \mathcal{Y}$ induces a rational multisection of $\mathcal{X}_t \rar \mathcal{Y}_t$.

The following examples will show that these conditions are not vacuously satisfied, but they need to be imposed and they are independent of each other.

\begin{example}
In this example, we produce an unbounded class of smooth minimal threefolds of Kodaira dimension 2 with bases belonging to a bounded family.
The unboundedness will follow from the fact that the elliptic threefolds in our construction do not admit a multisection of bounded degree, and $\mathrm{vol}_2(X)$ does not belong to a finite set.

Fix an elliptic curve $E$.
Let us consider the diagonal action of $\zz / n \zz$ on $E \times \pr 1.$ given as the translation by an element of order $n$ on $E$ and the action of a primitive $n$-th root of unity on $\pr 1.$.
Then, the induced action on $E \times \pr 1.$ has no fixed points and we obtain the following commutative diagram
\[
\xymatrix{
E \times \mathbb{P}^1
\ar[rr]^{\phi} \ar[d]_f 
& &
S_n\coloneqq (E \times \mathbb{P}^1) / (\mathbb Z / n \mathbb Z) \ar[d]^{g_n}\\
\pr 1. \ar[rr]^{\psi}
& &
\mathbb{P}^1= \mathbb P ^1 / (\mathbb Z / n \mathbb Z).
}
\]
By construction, $g_n$ has fibers of multiplicity $n$ over $\{ 0 \}$ and $\{ \infty \}$.
Thus, Kodaira's canonical bundle formula for surfaces implies that $K_{S_n} \sim_\mathbb{Q} g_n^\ast (K_{\pr 1.}+ (1-\frac{1}{n})(\{ 0 \} + \{\infty\}))$.
Let $C$ be a genus $2$ curve obtained by as a degree 2 cover $h \colon C \to \pr 1.$ branched away from $\{ 0 \}$ and $\{ \infty \}$.
Taking the base change of $g_n$ by $h$, we obtain a surface $T_n \coloneqq S_n \times_{\pr 1.} C$ with a morphism $l_n \colon T_n \rar C$ such that $\kappa(T_n)=1$ and $K_T \sim_\mathbb{Q}h^\ast(K_C+(1-\frac{1}n)(p+q+r+s))$, where $p,q,r,s$ are the preimages of $\{0\}$ and $\{ \infty \}$ on $C$.
The divisor $K_C+(1-\frac{1}n)(p+q+r+s)$ has degree $2+4(1-\frac{1}{n})$.
As mentioned above, when considering the Iitaka fibration in a bounded family of minimal models, Kodaira's formula for the canonical bundle is obtained by restriction, up to a suitable stratification.
Thus, the fact that the log pairs $(C,(1-\frac{1}n)(p+q+r+s))$ have coefficients varying in an infinite set implies that corresponding surfaces $T_n$ do not belong to a bounded family.
As the fibration $T_n \rar C$ has fibers of multiplicity $n$, that cannot admit a multisection of degree less than $n$.
Thus, as $n$ varies, the fibrations $g_n$ do not admit a common upper bound for the minimal degree of a multisection.
Then, taking the product of $T_n \rar C$ with $C$, we get an example of a smooth minimal threefold $X_n$ with $\kappa (X_n)=2$ and with the same properties as above.
Indeed, if $W_n$ is a multisection of $X_n \rar C \times C$, for a general choice of $c \in C$, $W_n$ induces a multisection of $X_n \times_{C\times C} \{c\} \times C \rar \{c\} \times C$ which is isomorphic to $T_n \rar C$.
Hence, although the bases of the Iitaka fibrations of the $X_n$ are all isomorphic and thus trivially belong to a bounded family, on the other hand, $\mathrm{vol}_2(X_n)$ does attain infinitely many distinct values, and there is no lower bound for the degree of a rational multisection of $X_n \rar C \times C$.
\end{example}

\begin{example}
In this example, we show that the conditions of Theorem \ref{thm intro2} 
are independent of each other.

Let $C_n$ be a smooth curve of genus $n$, and let $E$ be an elliptic curve.
Then, $C_n \times C_n \times E \rar C_n \times C_n$ is a smooth minimal elliptic threefold of Kodaira dimension 2 admitting a section.
Furthermore, $\mathrm{vol}_2(C_n \times C_n \times E)=(2n-2)^2$ depends on $n$.
Thus, the set of the varieties $C_n$ is not bounded, showing that it is not sufficient to just assume the existence of an upper bound on the degree of a multisection of the Iitaka fibration to prove the boundedness of the minimal terminal $n$-folds of Kodaira dimension $n-1$.

Fix a curve $C$ with $g(C) \geq 2$, and let $E$ be an elliptic curve.
Then, by \cite{Fil20}*{Example 3.1}, there exists a set of smooth surfaces $f_n \colon S_n \rar C$ with the following properties: $f_n$ is smooth, isotrivial, and $f_n$ does not admit a multisection of degree less than $n$.
Setting $X_n \coloneqq S_n \times C$ and $g_n \colon X_n \rar C \times C$ the induced map, then $X_n$ is a smooth minimal threefold with $\kappa(X_n)=2$, $\mathrm{vol}_2(K_{X_n})=(2g(C)-2)^2$ fixed, whereas $X_n$ does not admit a rational multisection of degree less than $n$.
Hence, this example in turn shows that it is not sufficient to just assume the existence of an upper bound on $\mathrm{vol}_{n-1}(Y)$ to prove the boundedness of the minimal terminal $n$-folds $Y$ of Kodaira dimension $n-1$.
\end{example}

\subsection*{Strategy of proof} 
In the context of Theorem~\ref{thm intro} and Theorem~\ref{thm intro2}, we shall consider a set of elliptically fibered varieties $\mathfrak{F}$ that is known to be bounded modulo flops.
Furthermore, we can assume that these flops preserve the elliptic fibration.
More precisely, we shall assume that there exists a family $\xymatrix{\pi \colon \mathcal{X} \ar[r]^{\tilde{f}} & \mathcal{S} \ar[r]^{\tilde{g}} & T}$ of projective morphisms of quasi-projective varieties such that $\pi$ is a flat family of threefolds, $\tilde g$ is a flat family of surfaces, and for every fibration $f \colon X \rar S \in \mathfrak{F}$, there exists $t \in T$ such that the following diagram holds
\[
\xymatrix{
X \ar@{-->}[rrrrrr]_{\text{sequence of $K_X$-flops}}  \ar[d]^{f}
& & & & & &
\mathcal X_t \ar[d]^{\tilde{f}\vert_{\mathcal X_t}}
\\
S \ar[rrrrrr]_{\text{isomorphism}} 
& & & & & &
\mathcal S_t.
}
\]

In this setup, Kawamata \cite{Kaw97} showed that each $X \rar S$ admits only finitely many relatively minimal models over $S$, up to isomorphism.
That is, while there may be infinite sequences of flops over $S$ and thus infinitely many marked minimal models, these models belong to finitely many isomorphism classes of $S$-schemes.
In view of this, our strategy will be to show that also $\mathcal{X} \rar \mathcal{S}$ admits only finitely many relatively minimal models, up to isomorphism over $\mathcal{S}$, and that every $X \rar S$ in $\mathfrak{F}$ appears as a fiber of one of those finitely many models of $\mathcal{X} \rar \mathcal{S}$.
More precisely, we shall prove the following two steps:
\begin{itemize}
    \item[(i)] generalize the results of \cite{Kaw97} to relatively minimal elliptic fibrations of arbitrary dimension; and
    \item[(ii)] argue that, under suitable geometric assumptions, every sequence of $\K \mathcal{X}_t.$-flops $\mathcal{X}_t \drar \mathcal{X}'_t$ relative to $\mathcal{S}_t$ can be lifted to a sequence of $\K \mathcal{X}.$-flops $\mathcal{X} \drar \mathcal{X}'$ relative to $\mathcal{S}$.
\end{itemize}

Step (i) guarantees that $\mathcal{X} \rar \mathcal{S}$ admits only finitely many relatively minimal models, $\mathcal{X}_1,\ldots , \mathcal{X}_k$, up to isomorphism over $\mathcal{S}$, see \S~\ref{section cone conj}.
Then, step (ii) guarantees that each fibration in $\mathfrak{F}$ appears as the fiber over a closed point of $\mathcal{X}_i \rar T$ for some $1 \leq i \leq k$.
Indeed, let $X \rar S$ be an element of $\mathfrak{F}$.
Then, by assumption, there is a closed fiber $\mathcal{X}_t \rar \mathcal{S}_t$ such that $\mathcal{S}_t = S$ and $\mathcal{X}_t \drar X$ decomposes as a sequence of $\K \mathcal{X}_t.$-flops over $S=\mathcal{S}_t$.
Then, by (ii), we can lift this sequence as a sequence of $\K \mathcal{X}.$-flops $\mathcal{X} \drar \mathcal{X}'$ over $\mathcal{S}$.
Then, by construction, we have that $X$ is isomorphic to the fiber $\mathcal{X}'_t$.
Since $\mathcal{X}'$ is isomorphic over $\mathcal{S}$, and hence over $T$, to $\mathcal{X}_i$ for some $1 \leq i \leq k$, it follows that $X$ is isomorphic to $\mathcal{X}_{i,t}$, showing boundedness as desired.

In general, it is hard to show that a flop can be lifted from a special fiber of a family, as the Picard rank of the fibers can jump countably many times.
On the other hand, the Calabi--Yau condition guarantees that, under a suitable base change, the Picard rank remains constant in a family, allowing for identification between the relative N\'eron--Severi group and the N\'eron--Severi group of each fiber.
This is worked out in \S~\ref{section deform}.
The results of \S~\ref{section cone conj} and \S~\ref{section deform} are then combined to prove Theorem \ref{metatheorem}, which represents a general criterion to prove boundedness for Calabi--Yau varieties that form a bounded family up to flops.
Then, Theorem \ref{thm intro} immediately follows from Theorem \ref{metatheorem}.

The case of Theorem \ref{thm intro2} is different, as it is not true in general that threefolds of Kodaira dimension 2 have locally constant Picard rank.
To circumvent this issue, we shall use results of Koll\'ar and Mori showing that flops of terminal threefolds are locally unobstructed, see \cite{KM92}*{Theorem 11.10}.
Thus, while Theorem \ref{thm intro} relies on arguments that are valid in higher dimension, the proof of Theorem \ref{thm intro2} is special to the case of threefolds.

\subsection*{Acknowledgements} The authors would like to thank Burt Totaro for kindly suggesting the proof of Theorem \ref{t-BT}, which is a generalization of \cite{Tot12}*{Theorem 4.1}.
The authors wish to thank Joaqu\'in Moraga for reading a preliminary draft of this work.
They also thank Antonella Grassi, Talon Stark, and Isabel Stenger for useful comments on the first version of this work.
Lastly, the authors would like to thank the anonymous referee for useful comments and suggestions that helped the authors improve the clarity of this work.

\section{Preliminaries}

\subsection{Terminology and conventions}
\label{term.subs}
Throughout this paper, we will work over $\mathbb C$.
For anything not explicitly addressed in this subsection, we refer the reader to~\cites{KM98, Kol13}.

\subsection{Notation on morphisms and maps}
\label{sect:not.maps}
A \emph{contraction} is a projective morphism $f\colon X \rar Y$ of quasi-projective varieties with $f_\ast  \O X. = \O Y.$. 
If $X$ is normal, then so is $Y$ and the fibers of $f$ are connected.

A \emph{fiber space} is a contraction $f\colon X \rar Y$ of normal quasi-projective varieties with $\dim X > \dim Y$.
Given a fiber space $f \colon X \rar Y$, we define
\begin{align*}
\mathrm{Bir}(X/Y) \coloneqq 
\left\{
\phi \in \mathrm{Bir}(X) 
\ \vert \ 
f \circ \phi = f
\right \}
\quad \text{and} \quad
\mathrm{Aut}(X/Y) \coloneqq 
\left\{
\psi \in \mathrm{Aut}(X) 
\ \vert \ 
f \circ \psi = f
\right \}.
\end{align*}
There exists a natural identification $\mathrm{Bir}(X/Y)=\mathrm{Bir}(X_\eta)$, where $\eta$ is the generic point of $Y$, see \cite{Han91} -- $\mathrm{Bir}(X/Y)$ is denoted by $\mathrm{Bir}_Y(X)$ in \cite{Han91}.
More precisely, the $k$-points of $\mathrm{Bir}(X/Y)$ are identified with the $k(Y)$-points of $\mathrm{Bir}(X_\eta)$.
Hence, if $f \colon X \rar Y$ is an elliptic fibration, then $\mathrm{Bir}(X/Y)=\mathrm{Aut}_{k(Y)}(X_\eta)$.

We will need the following simple result.
\begin{lemma} \label{lemma_rel_bir}
Let $f \colon X \rar Y$ be a contraction of normal varieties.
Assume that that $f$ admits a factorization
\begin{align*}
\xymatrix{X \ar[r]^g \ar@/^14pt/[rr]^{f} & Y' \ar[r]^h & Y}
\end{align*}
where $h \colon Y' \rar Y$ be a birational contraction.
Then, 
$\mathrm{Bir}(X/Y)=\mathrm{Bir}(X/Y')$ 
and
$\mathrm{Aut}(X/Y)=\mathrm{Aut}(X/Y')$.
\end{lemma}

\begin{proof}
Fix $\phi  \in \mathrm{Bir}(X/Y)$.
Since $f = h \circ g$, then $\mathrm{Bir}(X/Y') \subset \mathrm{Bir}(X/Y)$ -- this inclusion holds even when $Y' \rar Y$ is not a birational morphism.
Let $U \subset Y$ be an open subset over which $h$ is an isomorphism and let $X_U \coloneqq X \times_Y U$.
Then, by construction, $g= g \circ \phi$ on $X_U$.
Thus, $\phi \in \mathrm{Bir}(X/Y')$.
This proves $\mathrm{Bir}(X/Y)=\mathrm{Bir}(X/Y')$.
Finally, $\mathrm{Aut}(X/Y)=\mathrm{Aut}(X/Y')$ follows immediately.
\end{proof}

Let 
$r \colon X \dashrightarrow X'$ 
be a birational map of quasi-projective varieties.
We say that 
$r$ 
is an isomorphism in codimension 1 if there exists closed Zariski subvarieties 
$Z \subset X$, 
$Z' \subset X'$
of codimension at least 
$2$
such that 
$r$ 
induces an isomorphism 
between 
$X \setminus Z$ 
and 
$X' \setminus Z'$.

Let 
$f \colon X \to Y$, 
$f' \colon X' \to Y$
be morphisms (with the same target variety $Y$).
Then an isomorphism in codimension 1 
$r \colon X \dashrightarrow X'$
is said to be an isomorphism in codimension 1 over $Y$, if the naturally induced diagram 
\begin{align*}
    \xymatrix{
X 
\ar@{-->}[rr]^{r}
\ar[dr]^{f}
& &
X'
\ar[dl]^{f'}
\\
& Y &
}
\end{align*}
commutes.

\subsection{Divisors}
Let $\mathbb{K}$ denote $\zz$, $\qq$, or $\rr$. We say that $D$ is a \emph{$\mathbb{K}$-divisor} on a variety $X$ if we can write $D = \sum \subs i=1. ^n d_i P_i$ where $d_i \in \mathbb{K}$, $n \in \nn$ and the $P_i$ are prime Weil divisors on $X$ for all $i=1, \ldots, n$. 
We say that $D$ is $\mathbb{K}$-Cartier if it can be written as a $\mathbb{K}$-linear combination of $\zz$-divisors that are Cartier.
The \textit{support} of a $\mathbb{K}$-divisor $D=\sum_{i=1}^n d_iP_i$ is the union of the prime divisors appearing in the formal sum $\mathrm{Supp}(D)= \sum_{i=1}^n P_i$.
In all of the above, if $\mathbb{K}= \zz$, we will systematically drop it from the notation.

Given a $\mathbb K$-divisor $D$ and a prime divisor $P$ in the support of $D$, we denote by $\mu_P (D)$ the coefficient of $P$ in $D$.
Given $D = \sum_{P \text{ prime}} \mu_{P_i}(D) P_i$ on a normal variety $X$, and a morphism $\pi \colon X \to Z$, we define the vertical (resp. horizontal) part $D^v$ (resp. $D^h$) of $D$ by
\begin{align*}
D^v \coloneqq \sum_{\pi(P_i) \subsetneqq Z} \mu_{P_i}(D) P_i, 
\qquad
D^h \coloneqq \sum_{\pi(P_i) = Z} \mu_{P_i}(D) P_i.
\end{align*}
Let $D_1$ and $D_2$ be $\mathbb K$-divisors on $X$ and let $\pi \colon X \rar Z$ be a projective morphism of normal varieties.
We write $D_1 \sim_{\mathbb K , \pi} D_2$ if there is a $\mathbb{K}$-Cartier divisor $L$ on $Z$ such that $D_1 - D_2 \sim _{\mathbb K}f^\ast L$.
Equivalently, we may also write $D_1 \sim_{\mathbb{K}, Z} D_2$, or $D_1 \sim_{\mathbb{K}} D_2$ over $Z$.
Similarly, if $Z= \mathrm{Spec}(k)$, where $k$ is the ground field, we omit $Z$ from the notation.
In case $D_1$ and $D_2$ are $\mathbb K$-Cartier, we say that $D_1$ and $D_2$ are numerically equivalent over $Z$ if $D_1 \cdot C = D_2 \cdot C$ for every curve $C \subset X$ such that $\pi(C)$ is a point, and we write $D_1 \equiv_\pi D_2$, or, alternatively, $D_1 \equiv_Z D_2$.
If $\mathbb{K}=\zz$, we omit it from the notation.

\subsection{Cones of divisors} \label{cones.of.divs}
Let $f \colon X \rar Y$ be a projective morphism of varieties.
We denote by $\nsr(X/Y)$ the real vector space generated by Cartier divisors on $X$ modulo numerical equivalence on curves in $X$ that are contracted by $f$.
It is a finite-dimensional vector space, and its dimension is denoted by $\rho(X/Y)$.
\\
We denote by $V(X/Y)$ the $\mathbb R$-subspace of $\nsr(X/Y)$ generated by the classes of vertical divisors and by $v(X/Y)$ its dimension.
\\
We denote by $A(X/Y)$ the cone of $f$-ample divisors and by $\overline{A}(X/Y)$ its closure, that is, the cone of $f$-nef divisors.
Similarly, we denote by $B(X/Y)$ the cone of $f$-big divisors and by $\overline{B}(X/Y)$ its closure, that is, the cone of $f$-pseudo-effective divisors.
\\
A Cartier divisor $D$ on $X$ is $f$-movable if we have $f_\ast  \O X.(D) \neq 0$ and the codimension of the support of $\mathrm{coker}(f^\ast f_\ast (\O X. (D))\rar \O X. (D))$ is at least 2.
We denote by $\overline{M}(X/Y)$ the closed cone of $f$-movable divisors: this cone is the closure of the cone generated by $f$-movable divisors.
\\
We denote by $B^e(X/Y)$ the cone of $f$-effective divisors, and we set $A^e(X/Y) \coloneqq \overline{A}(X/Y) \cap B^e(X/Y)$ and $M^e(X/Y) \coloneqq \overline{M}(X/Y) \cap B^e(X/Y)$.

\begin{lemma} 
\label{tower movable}
Let $f \colon X \rar Y$ and $g \colon Y \rar Z$ be contractions of normal varieties.
Let $L$ be a line bundle on $X$ that is movable over $Z$.
Then, $L$ is movable over $Y$.
\end{lemma}

\begin{proof}
We wish to show that the cokernel of the natural morphism $f^\ast f_\ast L \rar L$ is supported in codimension at least $2$.
By assumption, $f_\ast L$ is a coherent sheaf on $Y$, and there is a natural morphism
\begin{equation} \label{eq coker 1}
g^\ast g_\ast (f_\ast L) \rar f_\ast  L.
\end{equation}
By assumption, the cokernel of the natural morphism
\begin{equation} \label{eq coker 2}
(g \circ f)^\ast (g \circ f)_\ast L \rar L
\end{equation}
has codimension at least 2.
By \eqref{eq coker 1}, the morphism in \eqref{eq coker 2} factors as
\begin{align*}
f^\ast g^\ast g_\ast f_\ast L \rar f^\ast f_\ast L \rar L,
\end{align*}
and the claim follows.
\end{proof}

\begin{lemma} \label{hyperplane movable}
Let $f \colon X \rar Y$ be a contraction of normal varieties.
Let $L$ be a line bundle on $X$ that is movable over $Y$.
Let $H$ be a general member of a basepoint-free linear system on $Y$, and let $X_H \coloneqq X \times_Y H$.
Then, $L\vert_{X_H}$ is movable over $H$.
\end{lemma}

\begin{proof}
Let us consider the Cartesian diagram
\begin{align*}
\xymatrix{
X_H \ar[rr]^v \ar[d]^g
& & 
X \ar[d]^{f}\\
H \ar[rr]^u
& &
Y.  
}
\end{align*}
By assumption, the natural morphism
\begin{align*}
f^\ast f_\ast L \rar L
\end{align*}
is surjective outside a subset $V \subset X$ of codimension at least 2.
Since $H$ is a general element of a basepoint-free linear series, then $X_H$ is a general element of the free linear series $\vert f^\ast H \vert$, by the projection formula.
Thus, we may assume that $V \cap X_H$ has codimension at least 2 in $X_H$.
By construction, we have $L_{X_H} \coloneqq L\vert_{X_H} = v^\ast L$, and we need to show that
\begin{align*}
g^\ast g_\ast L_{X_H} \rar L_{X_H}
\end{align*}
is surjective outside a set of codimension 2.
Since the pull-back is a right exact functor, we have that
\begin{align*}
v^\ast f^\ast f_\ast L \rar v^\ast L
\end{align*}
is surjective outside $V \cap X_H$, which has codimension at least 2 in $X_H$.

Since $f \circ v= u \circ g$, we have that $v^\ast f^\ast f_\ast L = g^\ast u^\ast f_\ast L$.
By cohomology and base change \cite{Har77}*{Remark III.9.3.1}, there is a natural morphism
\begin{align*}
u^\ast f_\ast L \rar g_\ast  v^\ast L.
\end{align*}
Thus, if we consider the pull-back to $X_H$, we have morphisms
\begin{align*}
v^\ast f^\ast f_\ast L =g^\ast u^\ast f_\ast L \rar g^\ast g_\ast v^\ast L \rar v^\ast L.
\end{align*}
Since the composition is surjective outside of $V \cap X_H$, then so is $g^\ast g_\ast v^\ast L \rar v^\ast L$.
This concludes the proof.
\end{proof}

\begin{lemma} \label{lemma divisors}
Let $f \colon X \rar Y$ be a contraction of quasi-projective varieties.
Let $\gamma \in B(X/Y)$, and let $(\gamma_i)_{i \in \nn} \subset \nsr(X/Y)$ be a sequence converging to $\gamma$.
Then, there exist Weil $\mathbb R$-divisors $D$ and $D_i$ on $X$, and $i \in \nn$, such that:
\begin{enumerate}
\item $[D]=\gamma$, $[D_i]=\gamma_i$;
\item there exists a reduced divisor $\Theta$ on $X$ such that for all $i \in \nn$, the support of $D_i$ is contained in $\Theta$;
\item $(D_i) \subs i \in \nn.$ converges to $D$ in the vector space of $\rr$-Weil divisors;
\item $D \geq 0$ and $D_i \geq 0$ for $i \gg 0$.
\end{enumerate}

Moreover, if $X$ is $\mathbb Q$-factorial and $(X, 0)$ is klt, then there exists a positive real number $\epsilon$ such that 
$(X, \epsilon D)$ is klt and for all $i \gg 0$, $(X, \epsilon D_i)$ is klt.
\end{lemma}

\begin{proof}
Since $f$ is projective and $\nsr(X/Y)$ is finite-dimensional, we may choose a basis $[H_1],\ldots,[H_k]$ for $\nsr(X/Y)$ such that each $H_j$ is ample and effective.
Since bigness is an open condition, up to rescaling, we may assume that $\gamma-\sum_{j=1}^k [H_j]$ is big over $Y$.
Thus, since $\gamma-\sum_{j=1}^k [H_j]$ is big over $Y$, there exists a divisor $G \geq 0$ representing $\gamma-\sum_{j=1}^k [H_j]$ that is big on $X$.
We set $D \coloneqq G + \sum \subs j=1.^k H_j$.
Since $\gamma_i \to \gamma$ and $[H_1],\ldots,[H_k]$ constitute a basis for $\nsr(X/Y)$, then, for $i \gg 0$, $\gamma_i$ is contained in
\begin{align*}
\left \{ 
\gamma+\sum \subs j=1.^k a_j[H_j]
\ \middle\vert \ 
\forall j , -1\leq a_j \leq 1 
\right\}.
\end{align*}
Thus, for $i \gg 0$, $\gamma_i$ can be represented by the divisor $D_i \coloneqq G+\sum \subs j=1. ^k \lambda_{i,j} H_j$, where $0 \leq \lambda_{i,j} \leq 2$ and $\lambda_{i,j} \to 1$ for all $j$.

We now prove the last claim. 
Since we can assume that $X$ is $\mathbb Q$-factorial and $(X, 0)$ is klt, then there exists $\epsilon>0$ such that
$(X, \epsilon (G+\sum_{j=1}^k 2H_j))$
is klt.
In particular, for any of choice of real numbers 
$c_j \in [0, 2]$, 
$j=1, \dots, k$, then also 
$(X, \epsilon (G+\sum_{j=1}^k c_jH_j))$
is klt.
Taking $c_j=1$ $\forall j$ and taking $c_j=\lambda_
{i, j}$, for all $i\gg 0$ show that the last claim holds.
\end{proof}

\begin{lemma} \label{SES vertical divisors}
Let $(X,\Delta)$ be a klt pair and $f \colon X \rar Y$ be a contraction of normal varieties. 
Assume that $\K X.+\Delta \sim \subs \qq, f. 0$ and that $f$ factors as $X \xrightarrow{g} Y' \xrightarrow{\pi} Y$, where $\pi$ is a birational morphism of normal varieties.
Then, there is a short exact sequence
\begin{align*}
0 \rar \nsr(Y'/Y) \rar V(X/Y) \rar V(X/Y')  \rar 0.
\end{align*}
\end{lemma}

\begin{proof}
The morphism $V(X/Y) \rar V(X/Y')$ is naturally induced by the equivalence relations of numerical equivalence over $Y$ and $Y'$, respectively, as any curve $C \subset X$ that is vertical over $Y'$ is also vertical over $Y$.
Since a divisor on $X$ is vertical over $Y$ if and only if it is vertical over $Y'$, this morphism is clearly surjective.

The morphism $\nsr(Y'/Y) \rar V(X/Y)$ is induced by $g^\ast $.
By considering curves in $X$ that are vertical for $\pi \circ g$ but not for $g$, and using the projection formula, it follows that $\nsr(Y'/Y) \rar V(X/Y)$ is injective.
It also follows immediately that $g^\ast \nsr(Y'/Y)$ is contained in $\ker (V(X/Y) \rar V(X/Y'))$.

Now, let $D$ be a divisor such that $[D] \in \ker (V(X/Y) \rar V(X/Y'))$:
to conclude, we need to show that $[D] \in g^\ast  \nsr(Y'/Y)$.
By the definition of the relative N\'{e}ron--Severi group, without loss of generality, we may assume that $D$ is a $\qq$-divisor vertical over both $Y$ and $Y'$.
Possibly adding the pull-back of a sufficiently ample divisor on $Y$, we may assume that $D$ is effective.
Hence, for $0 < \epsilon \ll 1$, the log pair $(X,\Delta + \epsilon D)$ is klt and $\K X.+ \Delta + \epsilon D \equiv_{g} 0$.
Thus, $X$ is a minimal model for $(X,\Delta +\epsilon D)$ over $Y'$.
Since $D$ is vertical and $\K X.+\Delta \sim \subs \qq,g. 0$, it follows from \cite{HX13}*{Theorem 1.1} and \cite{Lai11}*{Proposition 2.4}, that $X$ is a good minimal model for $(X,\Delta + \epsilon D)$:
in particular, $D$ is semi-ample over $Y'$, and thus, $D \sim_{\mathbb Q, g} 0$.
Hence, $[D] \in g^\ast \nsr(Y'/Y)$.
\end{proof}

\subsection{Boundedness} We now recall the notion of boundedness for a set of log pairs, and we introduce a suitable notion of boundedness for fibrations.
First, we recall the notion of log pair.
A \emph{log pair} $(X,B)$ is the datum of a normal quasi-projective variety and an $\rr$-divisor $B$, called \emph{boundary}, such that $\K X. + B$ is $\rr$-Cartier and $0 \leq B \leq \Supp (B)$.

\begin{definition}
\label{def.bounded.set.pairs}
Let $\mathfrak{D}$ be a set of projective log pairs.
\begin{enumerate}
    \item 
We say that $\mathfrak{D}$ is \emph{log bounded} if there exist a log pair $(\mathcal{X},\mathcal{B})$ and a projective morphism $\pi \colon \mathcal{X} \rar T$, where $T$ is of finite type, such that for any log pair $(X,B) \in \mathfrak{D}$ there exist a closed point $t \in T$ and an isomorphism $f_t \colon \mathcal{X}_t \rar X$ such that $(f_{t})_{\ast} \mathcal{B}_t= B$.
    \item 
We say that $\mathfrak{D}$ is \emph{log birationally bounded} if there exist a log pair $(\mathcal{X},\mathcal{B})$ and a projective morphism $\pi \colon \mathcal{X} \rar T$, where $T$ is of finite type, such that for any log pair $(X,B) \in \mathfrak{D}$ there exist a closed point $t \in T$ and a birational map $f_t \colon \mathcal{X}_t \dashrightarrow X$ such that $\Supp(\mathcal B _t)= \Supp((f_{t}^{-1})_\ast \Supp(B) + E)$, where $E$ is the exceptional divisor of $f_t$.
    \item 
If $\mathfrak D$ is log birationally bounded and for any log pair $(X, B) \in \mathfrak D$ the map $f_t$ in (2) is an isomorphism in codimension 1, then we say that $\mathfrak{D}$ is \emph{log bounded in codimension 1}.
\end{enumerate}
When a set 
$\mathfrak D$ 
of log pairs is actually a set of varieties, i.e., for any pair 
$(X, \Delta) \in \mathfrak D$,
the condition 
$\Delta=0$ 
is satisfied, then we shall say that 
$\mathfrak{D}$ 
is \emph{bounded}, (resp. \emph{birationally bounded}, \emph{bounded in codimension 1}) rather than  \emph{log bounded}, (resp. \emph{log birationally bounded}, \emph{log bounded in codimension 1}).
\end{definition}

\begin{definition} \label{def bdd fibrations}
Let $\mathfrak{F}$ be a set of triples $((X,B),(Y,D),\phi)$, where $(X,B)$ and $(Y,D)$ are projective log pairs and $\phi \colon X \rar Y$ is a contraction.
\begin{enumerate}
    \item
We say that $\mathfrak{F}$ is \emph{log bounded} if there exist log pairs $(\mathcal{X},\mathcal{B})$, $(\mathcal{Y},\mathcal{D})$, a variety of finite type $T$, and a commutative diagram of projective morphisms
\begin{align}
\label{diag.bound.fibr}
\xymatrix{
\mathcal X \ar[rr]^{\sigma} \ar[rd]^{\pi} 
& &
\mathcal Y \ar[dl]_{\rho}
\\
& T &
}
\end{align}
such that for any $((X,B),(Y,D),\phi) \in \mathfrak{F}$, there is a closed point $t \in T$ together with morphisms $f_t \colon \mathcal{X}_t \rar X$ and $g_t \colon \mathcal{Y}_t \rar Y$ inducing a commutative diagram
\begin{align*}
\xymatrix{
\mathcal X_t \ar[rr]^{f_t} \ar[d]_{\sigma\vert_{\mathcal X_t}}
& &
X \ar[d]^\phi 
\\
\mathcal Y_t \ar[rr]^{g_t}
& &
Y
}
\end{align*}
such that $(X,B) \cong (\mathcal{X}_t,\mathcal{B}_t)$  and $(Y,D) \cong (\mathcal{Y}_t,\mathcal{D}_t)$.
\item 
We say that $\mathfrak F$ is \emph{log birationally bounded} if there exist log pairs $(\mathcal{X},\mathcal{B})$, $(\mathcal{Y},\mathcal{D})$, a variety of finite type $T$, the same commutative diagram as in~\eqref{diag.bound.fibr} holds and for any $((X,B), (Y,D),\phi) \in \mathfrak{F}$, there is a closed point $t \in T$ together with birational  morphisms $f_t \colon \mathcal{X}_t \dashrightarrow X$ and $g_t \colon \mathcal{Y}_t \dashrightarrow Y$ inducing a commutative diagram
\begin{align*}
\xymatrix{
\mathcal X_t \ar@{-->}[rr]^{f_t} \ar[d]_{\sigma\vert_{\mathcal X_t}}
& &
X \ar[d]^\phi 
\\
\mathcal Y_t \ar@{-->}[rr]^{g_t}
& &
Y
}
\end{align*}
such that $\Supp(\mathcal B _t)$ contains the strict transform of $\Supp(B)$ and all the $f_t$ exceptional divisors) and  $\Supp(\mathcal D _t)$ contains the strict transform of $\Supp(D)$ and all the $g_t$ exceptional divisors.
\item 
If $\mathfrak F$ is log birationally bounded and the maps $f_t$, $g_t$ in (2) are isomorphisms in codimension 1, we say that $\mathfrak{F}$ is \emph{log bounded in codimension 1}.
\end{enumerate}
\end{definition}

When in a set 
$\mathfrak F$ 
of triples, for any triple 
$((X, B), (Y, D), \phi) \in \mathfrak F$, 
the condition 
$B=0=D$ 
is satisfied, then we say that 
$\mathfrak F$ 
is \emph{bounded}, (resp. \emph{birationally bounded}, \emph{bounded in codimension 1}) rather than  \emph{log bounded}, (resp. \emph{log birationally bounded}, \emph{log bounded in codimension 1}).

\subsection{Crepant birational models}
The following statement is known to the experts and follows from \cite{BCHM}*{Theorem E}.
For the reader's convenience, we include a short proof.

\begin{lemma}[Finiteness of crepant models] \label{finiteness models}
Let $(Y,\Delta )$ be a klt pair with $Y$ quasi-projective.
Let us consider the set $\mathcal M$ of all  $\nu \colon \widehat Y\to Y$ projective birational morphisms of normal varieties such that if $K_{\widehat Y}+\Delta _{\widehat Y}=\nu ^\ast (K_Y+\Delta _Y)$, then $\Delta _{\widehat Y}\geq 0$.
Then $\mathcal M$ is finite.
\end{lemma}

\begin{proof}
Let $A_Y$ be an ample effective divisor on $Y$.
For any model 
$\nu \colon \widehat Y \to Y$
in 
$\mathcal M$, 
we set 
$A_{\widehat Y} := \nu^\ast A_Y$.

Let 
$(Y',\Delta')$ 
be a 
$\qq$-factorial 
terminal model of
$(Y,\Delta)$ 
realized by a birational contraction 
$r \colon Y' \to Y$
such that 
$K_{Y'}+\Delta'=r^\ast(K_Y+\Delta)$-- such model exists by 
\cite{BCHM}*{Corollary 1.4.3}.
By the construction in {\it loc.cit.}, for any model
$\nu \colon \widehat Y \to Y$
in 
$\mathcal M$, 
there exists a birational contraction 
$r_{\nu} \colon Y' \dashrightarrow \widehat Y$
such that 
\begin{align*}
\xymatrix{
Y'
\ar@{-->}[r]^{r_{\nu}}
\ar@/^{1.5pc}/[rr]^{r}
&
\widehat{Y}
\ar[r]^{\nu}
&
Y
},
\end{align*}
and 
$r_{\nu}^\ast(K_{\widehat Y}+\Delta_{\widehat Y})= K_{Y'}+\Delta'$, 
and also
$r_{\nu}^\ast A_{\widehat Y} = r^\ast A_Y$.
Here the pull-back under the birational map 
$r_\nu$ 
is well-defined, since 
$r_\nu$ 
is a birational contraction, i.e., 
$r_\nu$ 
does not contract any divisor.

Let us write 
$A_{Y'}
\sim_{\mathbb Q} 
H_{Y'} + E_{Y'}$, 
where 
$H_{Y'}$
is a
$r$-ample
and effective
$\mathbb Q$-divisor, 
while
$E_{Y'}$ 
is effective.
Since 
$(Y', \Delta')$
is klt by construction, then there exists a positive rational number 
$0< \epsilon \ll 1$
such that
$(Y', \Delta'+\epsilon(H_{Y'} + E_{Y'}))$
is still klt.
Then, given any pair
$({\widehat Y},\Delta _{\widehat Y})$ 
with 
$\nu \colon \widehat Y\to Y$ 
which belongs the set
$\mathcal{M}$ 
defined in the statement of the lemma, then 
\begin{align}
\label{eqn:weak.lc.model}
r_{\nu}^\ast (K_{\widehat Y}+ \Delta_{\widehat Y} + \epsilon A_{\widehat Y}) = 
K_{Y'}+\Delta'+\epsilon A_{Y'}
\sim_{\mathbb Q}
K_{Y'}+\Delta'+\epsilon(H_{Y'} + E_{Y'}).
\end{align}
In particular, 
\eqref{eqn:weak.lc.model}
implies that 
$(\widehat Y, \Delta_{\widehat Y} + \epsilon (r_{\nu, \ast}H_{Y'} + r_{\nu, \ast}E_{Y'}))$
is a weak log canonical model of 
$(Y', \Delta'+\epsilon(H_{Y'} + E_{Y'}))$
relatively over 
$Y$,
in the sense of 
\cite{BCHM}*{Definition 3.6.6}.
Then
\cite{BCHM}*{Theorem E}
implies the finiteness of all 
possible distinct weak log canonical models of 
$(Y', \Delta'+\epsilon(H_{Y'} + E_{Y'}))$, 
which in turn also proves the finiteness of the models contained in 
$\mathcal M$.
\end{proof}

\begin{proposition} 
\label{prop models}
Let $X$ be a projective klt variety admitting an elliptic fibration $f \colon X \rar Y$.
Assume that $\K X. \sim \subs \qq,Y. 0$.
There exist finitely many birational morphisms $h_i \colon Y_i \to Y$, $i=1, \dots, n$ such that the following property holds:
\\
given a commutative diagram
\begin{align}
\label{diag.finite.models}
\xymatrix{
X \ar@{-->}[rr]^{\phi} \ar[dr]^f 
& & 
X' \ar[dl]^{f'} \ar[dr]^g & 
\\
 & Y & & Y' \ar[ll]^h
}
\end{align}
where $\phi$ is an isomorphism in codimension 1 over $Y$, and $h$ is birational, then there exists $1\leq i \leq n$ such that $Y'=Y_i$ and $h=h_i$.
\end{proposition}

The definition of a relative isomorphism of codimension 1 can be found in 
\S~\ref{sect:not.maps}.

\begin{proof}
The canonical bundle formula, see, for example,~\cite{Amb05}, guarantees the existence of a generalized log pair\footnote{For the theory of generalized log pairs and their canonical bundle formula we refer the reader to~\cite{FS.conn}.} $(Y,B_Y+M_Y)$ with generalized klt singularities such that $\K X. \sim f^\ast (\K Y. + B_Y+M_Y)$.
Furthermore, given a commutative diagram as in~\eqref{diag.finite.models},
then
\begin{enumerate} \setcounter{enumi}{-1}
    \item the canonical bundle formula for $g$ provides a generalized pair $(Y',B_{Y'}+M_{Y'})$ such that $h^\ast (K_Y+B_Y+M_Y)=K_{Y'}+B_{Y'}+M_{Y'}$, $B_Y=h_\ast B_{Y'}$, and $M_Y=h_\ast M_{Y'}$;
	\item 
there exists an effective $\mathbb Q$-divisor 
$N \sim_\mathbb{Q} M_Y$ 
independent of 
$Y'$ 
such that, setting
$\Delta_Y\coloneqq B_Y+N$, 
then 
$(Y, \Delta_Y)$ 
is klt; and
	\item
the log pull-back $(Y',\Delta_{Y'})$ of $(Y,\Delta_Y)$ to $Y'$ satisfies $\Delta_{Y'} \geq 0$. 
\end{enumerate}
Item (0) immediately follows from the canonical bundle formula and the definition of generalized pair.
To achieve (1) and (2), we argue as follows.
Consider a resolution $\xymatrix{Y'' \ar[r] & Y_{\rm term} \ar[r] & Y}$ of a terminalization $Y_{\rm term}$ of the generalized pair $(Y, B_Y+M_Y)$ where the trace $M_{Y''}$ of the moduli part on $Y''$ descends.
Since $B_{Y'}$ is effective on any model as in~\eqref{diag.finite.models} and $Y_{\rm term}$ is a generalized terminalization for $(Y,B_Y+M_Y)$, we have that $Y_{\rm term} \drar Y'$ is a birational contraction.

Let $\phi \colon Y'' \rar Y$ denote the induced morphism, and write $\K Y''. + B \subs Y''. + M \subs Y''.=\phi^\ast (K_Y+B_Y+M_Y)$.
Since $(Y,B_Y+M_Y)$ is generalized klt, then $(Y'',B \subs Y''.)$ is sub-klt.
By~\cite{PS09}*{Example~7.16}, $12M_{Y''}$ is globally generated divisor on $Y''$.
By Bertini's theorem, we may choose a general element $12N'' \in \vert 12M_{Y''} \vert$ such that $(Y'',B \subs Y''. + N'')$ is sub-klt.
Then, by construction, $(Y,B_Y+N)$ is a klt pair, where $N$ denotes the push-forward of $N''$ to $Y$.
Then, (1) follows.

By construction, $(Y'',B \subs Y''. + N'')$ is the log pull-back of $(Y,\Delta_Y)$, and $Y'' \drar Y'$ is a rational contraction.
Therefore, the log pull-back $(Y',\Delta_{Y'})$ of $(Y,\Delta_Y)$ coincides with the sub-pair obtained by pushing forward the sub-pair $(Y'',B \subs Y''. + N'')$ via the rational map $Y'' \drar Y'$.
Then, by construction, $\Delta \subs Y'.$ is the sum of $B_{Y'}$, which is effective by (0), and the strict transform of $N''$.
In particular, $\Delta_{Y'}$ is effective, and (2) follows.

In particular, as $(Y,\Delta_Y)$ is klt, it follows from Lemma \ref{finiteness models}, that there are finitely many log pairs $(Y',\Delta \subs Y'.)$ that can arise in the above construction.
\end{proof}

\begin{lemma} 
\label{movable bir 1}
Let $\pi \colon Y' \rar Y$ be a birational contraction, where $Y'$ is $\qq$-factorial.
Assume there exists a boundary $\Delta' \geq 0$ on $Y'$ such that $(Y',\Delta')$ is klt and $\K Y'. + \Delta' \sim_{\qq,\pi} 0$.
Then, $M(Y'/Y)=\overline{M}(Y'/Y)$.
\end{lemma}

\begin{proof}
Let $[D'] \in \overline{M}(Y'/Y)$.
Since $\pi$ is birational, $D'$ is relatively big and we may assume that $D' \geq 0$.
Since $[D']\in \overline{M}(Y'/Y)$, there exists a sequence of divisors $D'_i \geq 0$ such that $[D'_i] \in M(Y'/Y)$ and $[D'_i] \to [D']$ in $\nsr(Y'/Y)$.
For $0 < \epsilon \ll 1$, the log pair $(Y',\Delta' + \epsilon D')$ is klt, and
\begin{align*}
\K Y'. + \Delta ' + \epsilon D' \sim_{\qq,\pi} \epsilon D'.
\end{align*}
In particular, we may run a $D'$-MMP with scaling over $Y$, and this terminates with a good model $Y'' \rar Y$, since $\pi$ is birational.
Let $D''$ be the push-forward of $D'$ to $Y''$.
Thus, $D''$ is semi-ample over $Y$ and $[D''] \in M(Y''/Y)$.

To conclude, it suffices to show that $Y' \drar Y''$ is an isomorphism in codimension 1.
For this purpose, we observe that, as the MMP $Y' \drar Y''$ has finitely many steps, and since $[D'_i]$ converges to $[D']$ in $\nsr(Y'/Y)$, it follows that $Y' \drar Y''$ is a composition of steps of the $D'_i$-MMP over $Y$, for all $i \gg 1$.
Since each $[D'_i]$ is in $M(Y'/Y)$, the MMP is forced to be an isomorphism in codimension 1.
\end{proof}

\begin{lemma} 
\label{movable bir 2}
Let $\pi \colon Y' \rar Y$ be a birational contraction of $\qq$-factorial normal varieties.
Assume there exists a boundary $\Delta' \geq 0$ on $Y'$ such that $(Y',\Delta')$ is klt and $\K Y'. + \Delta' \sim_{\qq,\pi} 0$.
Let $E'_1, \ldots, E'_k$ denote the prime $\pi$-exceptional divisors.
Then, the classes of the $E'_i$ form a basis of $\nsr(Y'/Y)$.
Furthermore, if $[D'] \in M(Y'/Y)$, and we write $[D']=\sum a_i [E'_i]$, then $a_i \leq 0$ for all $i$.
\end{lemma}

\begin{proof}
For $0 < \epsilon \ll 1$, the log pair $(Y',\Delta' + \sum_{i=1}^k \epsilon E'_i)$ is klt and big over $Y$; 
therefore it admits a good minimal model $\widetilde{Y}$ over $Y$. 
Since the $E_i'$ are contained in the stable base locus of $K_{Y'}+\Delta' + \sum_{i=1}^k \epsilon E'_i$, it follows that  $\widetilde{Y} \to Y$ is a small birational morphism of $\mathbb Q$-factorial varieties and hence an isomorphism.
But then $\rho (Y'/Y)=k$ as it is well known that every divisorial contraction contracts an irreducible exceptional divisor.

Fix $[D'] \in M(Y'/Y)$.
We can write $D'=\sum a_i E'_i$.
As $D'$ is movable and $\K Y'. + \Delta' \sim_{\qq,\pi}0$, up to replacing $Y'$ with a model that is isomorphic in codimension 1, we may assume that $D'$ is semi-ample over $Y$.
Then, since $D'$ is supported on the exceptional locus of $Y' \rar Y$, by the negativity lemma, it follows that $D' \leq 0$.
\end{proof}

\subsection{Calabi--Yau fiber spaces}
\label{ssect:cy.fibre.sp}
Let $X$ and $Y$ be normal quasi-projective varieties, and let $f \colon X \rar Y$ be a projective morphism.
We say that $f \colon X \rar Y$ is a {\it Calabi--Yau fiber space} if the following conditions hold:
\begin{itemize}
\item[(CYF1)] $X$ is terminal and $\qq$-factorial;
\item[(CYF2)] $f$ is a contraction; and
\item[(CYF3)] $\K X. \equiv_f 0$.
\end{itemize}

\begin{remark}
In view of \cite{Gon13}*{Theorem 1.2}, \cite{HX13}*{Theorem 1.1}, and \cite{Lai11}*{Proposition 2.4}, condition (CYF3) is equivalent to $\K X. \sim_{\qq,f} 0$.
\end{remark}

For a given Calabi--Yau fiber space $f \colon X \to Y$, a {\it relatively minimal model} of $X$ over $Y$ (or, of $f$) is a contraction $f' \colon X' \to Y$ such $X'$ is terminal, $\qq$-factorial, $K_{X'} \equiv_{f'}0$, and $X'$ is birationally equivalent to $X$.
It is a well-known fact that, if $f' \colon X' \rar Y$ is a relatively minimal model of $f$, then $X$ and $X'$ are isomorphic in codimension 1.
Furthermore, $X$ and $X'$ are connected by a sequence of $\K X.$-flops over $Y$, see~\cite{Kaw08}.

Given two Calabi--Yau fiber spaces $f \colon X \to Y$, $f' \colon X' \to Y$ with $X, X'$ birationally equivalent, let $\alpha \colon X' \drar X$ be the isomorphism in codimension 1 over $Y$.
We refer to $\alpha$ as the {\it marking} of the minimal model $f'$.
A {\it marked minimal model} of 
$f$ 
is the datum of an ordered couple 
$(f' \colon X' \to Y, \alpha)$ 
where 
$f'$ 
is a relatively minimal model of 
$f$ 
together with the marking 
$\alpha$.
A {\it marked birational model} of 
$Y$ 
is the datum of a birational projective morphism 
$r \colon Y' \rar Y$.

Let $f \colon X \rar Y$ be a Calabi--Yau fiber space and $(f' \colon X' \rar Y,\alpha)$ be a marked minimal model.
Then, $\alpha_\ast$ induces an isomorphism between $\nsr(X'/Y)$ and $\nsr(X/Y)$ such that 
\begin{align*}
\alpha_\ast  B(X'/Y)=B(X/Y), \quad
\alpha_\ast  M(X'/Y)=M(X/Y), \quad 
\text{ and }  \quad 
\alpha_\ast V(X'/Y)=V(X/Y).
\end{align*}
We define
\begin{align*}
A(X/Y,\alpha) \coloneqq \alpha_\ast  A(X'/Y), \quad
\overline{A}(X/Y,\alpha) \coloneqq \alpha_\ast  \overline{A}(X'/Y), \quad
\text{ and } \quad
A^e(X/Y,\alpha) \coloneqq \alpha_\ast  A^e(X'/Y).
\end{align*}
By~\cite{Kaw97}*{Lemma 1.5}, 
$A(X/Y,\alpha) \cap A(X/Y) \neq \emptyset$
if and only if $\alpha$ is an isomorphism.

Having introduced this notation, it is easy to show that, assuming the termination of a relative MMP for any element of $M^e(X/Y)$, we have 
\begin{align}
\label{eq:decomp.mov.cone}
M^e(X/Y)=
\bigcup_{(f' \colon X'\rar Y,\alpha)} A^e(X/Y,\alpha),
\end{align}
where $(f' \colon X'\rar Y,\alpha)$ runs over all distinct $\qq$-factorial marked relatively minimal models of $f$.
In particular, \eqref{eq:decomp.mov.cone} holds true whenever $\dim X - \dim Y \leq 3$;
see, e.g., \cite{Fil20}*{Theorem 1.5}.

\begin{notation}
We call the decomposition in~\eqref{eq:decomp.mov.cone} the \emph{chamber decomposition} of ${M}^e(X/Y)$.
We call each cone ${A}^e(X/Y,\alpha)$ in~\eqref{eq:decomp.mov.cone} a \emph{chamber} of the decomposition.
\end{notation}

Let $f \colon X \rar Y$ be a Calabi--Yau fiber space, and let 
\begin{align*}
\xymatrix{
X \ar[r]^{g} \ar@/^15pt/[rr]^{f} &W \ar[r]^{h} &Y}
\end{align*}
be a factorization of $f$ such that $h$ is a contraction of normal varieties which is not an isomorphism. 
In particular, $g$ is itself a Calabi--Yau fiber space.
Moreover, 
$g^\ast \colon \nsr(W/Y) \rar \nsr(X/Y)$ 
is injective, and 
$g^\ast  {A}^e(W /Y) =
g^\ast \nsr(W/Y) \cap  {A}^e(X/Y)$ 
is an \emph{extremal face} of ${A}^e(X/Y)$. 
Analogously, if $(f'\colon X' \rar Y, \alpha)$ is a marked minimal model of $f$ factoring as
\begin{align}
\label{eqn.fact.extr.face1}
\xymatrix{
X' \ar[r]^{g'} \ar@/^15pt/[rr]^{f'} &W' \ar[r]^{h'} &Y},
\end{align}
where $h'$ is a contraction of normal varieties which is not an isomorphism, then 
$\alpha_\ast \circ (g')^\ast \colon \nsr(W'/Y) \rar \nsr(X/Y)$ 
is injective.
If $\dim(X')>\dim(W')$ or if $g'$ is a birational morphism that contracts at least one divisor, 
$\alpha_\ast ((g')^\ast {A}^e(W' /Y)) =
\alpha_\ast ((g')^\ast \nsr(W'/Y)) \cap  {M}^e(X/Y)$ 
is an \emph{extremal face} of ${M}^e(X/Y)$.
If $g'$ is a small birational morphism, $\alpha_\ast ((g')^\ast {A}^e(W' /Y))$ is a cone intersecting the interior of ${M}^e(X/Y)$ and is called {\it wall}.
The terminology comes from the following observation: if $g'\colon X' \rar W'$ is a small contraction with $\rho(X'/W')=1$, then $\alpha_\ast ((g')^\ast {A}^e(W' /Y))$ is the wall separating the chambers corresponding to 
$(f'\colon X' \rar Y,\alpha)$ and $((f')^+\colon (X')^+ \rar Y,\alpha^+)$,
where $(g')^+ \colon (X')^+ \rar W'$ is the flop of $g'$.
\begin{notation}
\label{notation.extr.faces}
With the notation and assumptions introduced above, if $g'$ in~\eqref{eqn.fact.extr.face1} is a birational morphism that contracts at least a divisor, we say that 
$\alpha_\ast (g')^\ast {A}^e(W' /Y))$ 
is an extremal face of ${M}^e(X/Y)$ corresponding to a birational contraction.
If $\dim X' > \dim W'$, we say that 
$\alpha_\ast (g')^\ast {A}^e(W' /Y)$ 
is an extremal face of ${M}^e(X/Y)$ corresponding to a fiber space structure.
\end{notation}

If $f \colon X \rar Y$ is a Calabi--Yau fiber space, as $X$ is terminal and minimal over $Y$, then for any $\phi \in \mathrm{Bir}(X/Y)$, $\phi \colon X\dashrightarrow X$ is a small birational map over $Y$.
Hence, $\phi_\ast$ induces a bijection on the lattice of Weil divisors and on its quotient modulo numerical equivalence $\lbrace \text{Weil divisors on $Y$} \rbrace / \equiv_Y \subset \nsr(X/Y)$.
Thus, there exists a natural induced representation 
\begin{align*}
\xymatrix @R=0.1pc 
{
\sigma \colon \mathrm{Bir}(X/Y) \ar[r] &
\mathrm{GL}(\nsr(X/Y),\mathbb{Z})
\\
\phi \ar@{|->}[r] &
\phi_\ast.
}
\end{align*}
Moreover, $\phi_\ast$ preserves the subspace $V(X/Y)$ and permutes the chambers of the partition of ${M}^e(X/Y)$ given in~\eqref{eq:decomp.mov.cone}.

The following result shows that the chamber decomposition of ${M}^e(X/Y)$, cf.~\eqref{eq:decomp.mov.cone}, is well behaved in the part of the movable cone of a Calabi--Yau fiber space consisting of big divisors.
The result generalizes~\cite{Kaw97}*{Theorem 2.6}.

\begin{lemma} 
\label{lemma chambers}
Let $f \colon X \rar Y$ be a Calabi--Yau fiber space.
Then, the decomposition
\begin{align}
\label{eqn:decomp.chambers}
M^e(X/Y) \cap B(X/Y)=
M(X/Y) \cap B(X/Y)=
\bigcup_{\substack{(f' \colon X'\rar Y,\alpha) 
\\
\text{marked minimal model of $f$}}}
A^e(X/Y,\alpha) \cap B(X/Y)
\end{align}
is locally finite in the open cone $B(X/Y)$.
\end{lemma}

By local finiteness of the decomposition in
\eqref{eqn:decomp.chambers},
we  simply mean that if 
$\Sigma$ 
is a closed convex cone contained in 
$\lbrace 0 \rbrace \cup B(X/Y)$, 
then there exist only finitely many relatively minimal models 
$(f'_i \colon X'_i \rar Y, \alpha_i)$, $i=1, \dots, k$ 
such that the cones 
$A^e(X/Y,\alpha_i) \cap (\lbrace 0 \rbrace \cup B(X/Y))$ 
intersect $\Sigma$.

\begin{proof}
By definition, $M^e(X/Y) \supset M(X/Y)$.
Therefore, the inclusion $M^e(X/Y) \cap B(X/Y)\supset M(X/Y) \cap B(X/Y)$ is clear, so we will prove the reverse inclusion.
Suppose that $[D]\in M^e(X/Y) \cap B(X/Y)$, we must show that $[D]\in M(X/Y)$.
Since $[D]\in B^e(X/Y)$,  there is $\Delta \geq 0$ such that $\Delta \sim_{\rr,Y} \alpha D$, for some positive real number $\alpha$, and $(X,\Delta)$ is terminal.
Then, for a general point $y \in Y$, the divisor $\Delta_y$ is $f$-big and $\K X_y. + \Delta_y \sim_\qq \Delta_y$.
By \cite{BCHM}, the general fiber $(X_y,\Delta_y)$ has a good minimal model.
By \cite{HMX18}*{Theorem 1.2} and \cite{HX13}*{Theorem 1.1}, $(X,\Delta)$ has a  good minimal model $\phi  \colon X\dasharrow X'$ for $(X,\Delta)$ over $Y$ and in particular $\phi _\ast D$ is semi-ample over $Y$.
By continuity, there is a movable divisor $D'$ (sufficiently close to $D$ in $\nsr(X/Y)$) such that $\phi$ is given by a sequence of $D'$-flips and divisorial contractions.
Suppose that $\phi$ contracts a divisor $F$; then $F$ is in the stable base locus of $D'$, which is impossible as $D'$ is movable.
Thus $\phi$ is small and hence $D=\phi ^{-1}_\ast \phi _\ast D $ is also movable.

Hence, if $[D]\in \Sigma$, then $[D]$ is contained in the interior of a rational polyhedral cone spanned by effective big  $\qq$-divisors $D_i$ such that $\epsilon D_i\sim _{\qq,Y}\Delta _i$ for some rational number $0< \epsilon \ll 1$ where $(X,\Delta _i)$ is klt.
Thus, we may apply finiteness of models \cite{BCHM}*{Theorem E}.
Since, for all $i$, $D_i$ is relatively movable if and only if the corresponding minimal models do not contract any divisors, then the claim now follows easily.
\end{proof}

The following result is a generalization of \cite{Kaw97}*{Lemma 3.3.(2)}.
See also \cite{GW}*{Theorem 40} for a similar statement.

\begin{lemma} 
\label{lemma_vertical_divs}
Let $X$ be a $\qq$-factorial terminal variety, and let $f \colon X \rar Y$ be a Calabi--Yau fiber space of relative dimension 1.
There exists a marked minimal model $(f'\colon X' \rar Y,\alpha)$ of $f$ together with a factorization 
\begin{align*}
\xymatrix{X' \ar[r]^{g'} \ar@/^14pt/[rr]^{f'} & Y' \ar[r]^{h'} & Y}
\end{align*} such that 
\begin{enumerate}
\item 
$Y'$ is $\qq$-factorial;
\item
$h'$ is birational; and 
\item 
every prime divisor in $X'$ vertical over $Y'$ dominates a divisor in $Y'$.
\end{enumerate}
\end{lemma}

\begin{proof}
We will proceed by induction on the relative Picard number $\rho(X/Y)$.
If $\rho(X/Y)=1$ there is nothing to prove, as every vertical prime divisor is relatively numerically trivial and, thus, it is numerically the pull-back of a $\qq$-divisor on $Y$.
Thus, we can assume that $\rho(X/Y) > 1$ and that the conclusions of the lemma hold for all Calabi--Yau fiber spaces of Picard rank lower than $\rho(X/Y)$.

Let $E \subset X$ be a prime divisor such that $\mathrm{codim}(f(E)) \geq 2$.
Let us consider a sufficiently positive very ample divisor $A$ on $Y$. 
Fix a general element $H$ of the non-complete sub-series $V \subset |A|$ of divisors in $|A|$ containing $f(E)$.
Passing to a higher multiple of $A$, if needed, we may assume that the sub-series $V$ is non-empty and has no fixed divisor.
Writing $f^\ast H=D_1+D_2$, where each component of $D_1$ dominates a divisor on $Y$, while the image of each prime component of $D_2$ has codimension at least $2$ on $Y$, then, $D_1 \neq 0$, and $E \leq D_2$. 
By construction, $D_1$ is $f$-movable, as it is a general member of the moving part of the linear series obtained by pull-back.\footnote{ To see this, it suffices to consider the linear series $\vert W \vert =f^\ast \vert V \vert$: two general elements $T_1, T_2 \in \vert W \vert$ have the same multiplicity along any prime divisor $D$ on $X$ such that $f(D)\subseteq f(E)$ and they share no other component; 
thus, $T_1=F_1+D_2$ and $T_2=F_2+D_2$, where the support of $D_2$ is mapped to $f(E)$; hence, $F_1\sim F_2$ and they share no prime divisor, hence 
$\vert F_1 \vert$ 
is movable. 
It then suffices to take $D_1 \coloneqq F_1$.}
Thus, by running a $(\K X. +  \epsilon D_1)$-MMP over $Y$, where $0 < \epsilon \ll 1$, this must terminate with a model $\widetilde f \colon \widetilde X \rar Y$ where the strict transform $\widetilde D_1$ of $D_1$ is relatively nef over $Y$.
Since $\dim(Y)=\dim(\widetilde X)-1$, then $\widetilde D_1$ is semi-ample over $Y$, cf.~\cite{Fil20}*{Theorem 1.5}.
In particular, $\widetilde D_1$ induces a morphism $\widetilde g \colon X \rar \widetilde Y$ over $Y$:
as $\widetilde D_1$ is vertical over $Y$, then $\widetilde Y \rar Y$ is birational and $\widetilde D_1$ is trivial over $\widetilde Y$.
Since $[D_1] \neq 0 \in \nsr(X/Y)$, then $\widetilde Y$ is not isomorphic to $Y$.
Thus, $\rho(X/Y)=\rho(\widetilde X/Y)>\rho(\widetilde X/ \widetilde Y)$, and the claim follows by the inductive hypothesis applied to $X' \rar Y'$.
Finally, the $\qq$-factoriality of $Y'$ follows from~\cite{Fil20}*{Proposition 2.9}.
\end{proof}

\begin{lemma} \label{lemma_cones}
Let $(X,\Delta)$ be a $\qq$-factorial log canonical pair and let $f \colon X \rar Y$ be a contraction. 
Assume that $\K X. + \Delta \sim \subs \qq,f. 0$ and that $f$ admits a factorization
\begin{align*}
\xymatrix{X \ar[r]^g \ar@/^14pt/[rr]^{f} & Z \ar[r]^h & Y}.
\end{align*}
Let $\gamma \colon X \drar X'$ be a sequence of $(\K X. + \Delta)$-flops over $Z$, and let $f' \colon X' \rar Y$ and $g' \colon X' \rar Z$ be the induced morphisms.
Then, 
\begin{align*}
g^\ast \nsr(Z/Y)=\gamma^{-1}_\ast((g')^\ast \nsr(Z/Y)).
\end{align*}
Furthermore, all the cones inside $\nsr(Z/Y)$ are identified by this identity.
\end{lemma}

\begin{proof}
By induction, it suffices to show the statement for one flop.
Thus, we may assume that $X$ admits a small contraction 
$\phi \colon X \rar X''$ 
over $Z$ such that 
$\gamma \colon X \drar X'$ arises as the flop of $\phi$.
Let $g'' \colon X'' \rar Z$, $\psi \colon X' \rar X''$ be the induced morphisms.
Since $g^\ast =\phi^\ast  \circ (g'')^\ast $ and $(g')^\ast =\psi^\ast  \circ (g'')^\ast $, it suffices to show that, if $D$ is an $\rr$-Cartier divisor on $X''$, then $\psi^\ast D= \alpha_\ast \phi^\ast D$, which follows at once from the construction.
This equality also implies the claim about the cones of $\nsr(Z/Y)$.
\end{proof}

\begin{lemma} \label{lemma_cones2}
Let $f \colon X \rar Y$ be a Calabi--Yau fiber space.
Assume that $f$ admits a factorization
\begin{align*}
\xymatrix{X \ar[r]^g \ar@/^14pt/[rr]^{f} & \widetilde{Y} \ar[r]^h & Y}
\end{align*}
such that $h$ is a birational contraction with $\widetilde{Y}$ $\qq$-factorial.
Let $(\widetilde{Y} , \Delta \subs \widetilde{Y}.)$ be a klt log pair such that $K_X\sim_\mathbb{R} g^\ast(\K \widetilde{Y}.+ \Delta \subs \widetilde{Y}.)$, and let $\beta \colon \widetilde{Y} \drar \widetilde{Y}^+$ be a $(\K \widetilde{Y}. + \Delta \subs \widetilde{Y}.)$-flop over $Y$.
Then, there exists a marked minimal model $(f^+\colon X^+ \to Y, \alpha)$ of $f$ together with a commutative diagram
\begin{align}
    \label{diag:comm.diag.2}
\xymatrix{
X^+ \ar@{-->}[rrrr]^{\alpha} \ar[d]_{g^+}
& & & &
X \ar[d]^{g}
\\
\widetilde Y^+ \ar@{-->}[rrrr]^{\beta^{-1}} \ar[drr]_{h^+}
& & & &
\widetilde Y \ar[dll]^{h}
\\
& & Y. & &
}
\end{align}
In particular, 
$
g^\ast \nsr(\widetilde{Y}/Y)= 
\alpha_\ast ((g^+)^\ast \nsr(\widetilde{Y}^+/Y) )
\subset \nsr(X/Y).
$
\end{lemma}

\begin{proof}
The existence of the marked minimal model and of the diagram in~\eqref{diag:comm.diag.2} follows from~\cite{Fil20}*{Proposition 2.9}.

To prove the final claim, it suffices to notice that the linear map $\alpha_\ast \colon \nsr({X}^+/Y) \rar \nsr(X/Y)$ is an isomorphism, as $\alpha$ is an isomorphism in codimension 1;
similarly, $\beta_\ast \colon \nsr(\widetilde{Y}/Y) \rar \nsr(\widetilde{Y}^+/Y)$ is also an isomorphism;
the conclusion then follows from the commutativity of~\eqref{diag:comm.diag.2}.
\end{proof}

\begin{corollary} \label{corollary finite fiber spaces general}
Let $f \colon X \rar Y$ be a Calabi--Yau fiber space of relative dimension 1.
There exist only finitely many extremal faces of ${M}^e(X/Y)$ corresponding to fiber space structures.
\end{corollary}

Let us recall, cf. Notation~\ref{notation.extr.faces}, that, since the relative dimension of $f$ is 1, an extremal face of ${M}^e(X/Y)$ corresponding to a fiber space structure is an extremal face of ${M}^e(X/Y)$ of the form $\alpha_\ast ((g')^\ast {A}^e(Y'/Y))$, where $(f' \colon X' \rar Y, \alpha)$ is a marked minimal model of $X \rar Y$ together with a factorization
\begin{align*}
\xymatrix{
X' \ar@/^15pt/[rr]^{f'} \ar[r]^{g'} & Y' \ar[r]^{h'} & Y
}
\end{align*}
such that $h'$ is birational.

\begin{proof}
By Proposition~\ref{prop models}, there exist finitely many birational morphisms $\widetilde{Y} \rar Y$ such that, if $\widetilde{X}$ is a relatively minimal model for $X \rar Y$, then $\widetilde{X} \rar Y$ factors through $\widetilde{Y}$.

Fix  one such choice of $\widetilde X$ and $\widetilde{Y}$, and let $g \colon \widetilde{X} \rar \widetilde{Y}$ be the corresponding morphism.
By Lemma \ref{lemma_rel_bir}, $g^\ast {A}^e(\widetilde{Y}/Y) \subset \nsr(X/Y)$ is invariant under the action of $\mathrm{Bir}(X/Y)$;
notice that here we are identifying $\nsr(X/Y)$ and $\nsr(\widetilde{X}/Y)$.
By Lemma \ref{lemma_cones}, if $h \colon \widehat{X} \rar \widetilde{Y}$ is another model whose structure morphism over $Y$ factors through $\widetilde{Y}$, we have $h^\ast {A}^e(\widetilde{Y}/Y)=g^\ast {A}^e(\widetilde{Y}/Y)$.
Then, by the finiteness of the models $\widetilde{Y}$, the claim follows.
\end{proof}

\subsection{Calabi--Yau varieties}\label{def_CY} 
A normal projective variety $X$ is a \emph{Calabi--Yau variety} if 
\begin{itemize}
	\item[(CY1)]
$K_X \sim 0$;
	\item[(CY2)]
$X$ has terminal $\qq$-factorial singularities; and
	\item[(CY3)]
$h^i(X,\O X.)=0$, for $0< i < \dim (X)$.
\end{itemize}

Some authors define Calabi--Yau varieties using instead of (CY1) the slightly weaker condition that the canonical bundle is a torsion rank 1 divisorial sheaf.
Passing to the index $1$ cover of such variety, one can always reduce to the case where the canonical bundle is trivial.
Nonetheless, this reduction may affect conditions (CY2-3).

In our treatment, the condition (CY1) will be used to guarantee that any elliptic Calabi--Yau $f \colon X \to S$ can be reconstructed via the Tate--Shafarevich group (over a big open set of $S$) of the associated Jacobian fibration $j \colon J(X) \to S$, see \S~\ref{sect.ell.3folds}. 
In order to do that, we need to know that over any codimension 1 point of $S$ the general fiber is not a multiple one, which is implied by adjunction and the fact that $K_X$ is linearly equivalent to $0$ rather than torsion, cf. Remark~\ref{rmk.TS.ell.CY}.

\subsection{The cone conjecture}

Consider a Calabi--Yau fiber space $f \colon X \rar Y$.
Then, as explained in \S~\ref{ssect:cy.fibre.sp}, the cone of effective movable divisors ${M}^e(X/Y)$ admits a decomposition into chambers ${A}^e(X/Y,\alpha)$, where $\alpha \colon X \drar X'$ is some marked minimal model of $X$ over $Y$.
Under this decomposition, either ${A}^e(X/Y,\alpha)={A}^e(X/Y)$ and $\alpha$ is an isomorphism, or $\alpha_\ast  \mathrm{Int}({A}^e(X/Y)) \cap \mathrm{Int}({A}^e(X'/Y))= \emptyset$, where ${\rm Int}$ indicates the interior of a set, see \cite{Kaw97}*{Lemma 1.5}.

Therefore, to study all the possible minimal models of $f \colon X \rar Y$ we can analyze the cones ${M}^e(X/Y)$ and ${A}^e(X/Y)$.
It can happen that a minimal model $X'$ is isomorphic to $X$, while the rational map over $Y, \; \alpha \colon X \drar X'$ is not an isomorphism \cite{Kaw97}*{Example 3.8.(2)}.
Thus, we may have more chambers corresponding to the same isomorphism class of varieties.
Therefore, if we are only interested in the isomorphism classes as schemes over $Y$ of the relative minimal models of $X \drar Y$, we should study when different marked minimal models are actually isomorphic over $Y$.

The so-called Kawamata--Morrison cone conjecture \cite{Tot10}*{Conjecture 2.1} addresses the discrepancy mentioned above between isomorphism classes of varieties $X'$ that appear as total spaces of a relatively minimal model $f' \colon X' \rar Y$ of $f$ and isomorphism classes over $Y$ of relatively minimal models of $f$.
\begin{conconj}[Kawamata--Morrison]\label{km.conj}
Let $f \colon X \rar Y$ be a projective morphism with connected fibers between normal varieties. Let $(X,\Delta)$ be a klt pair such that $\K X. + \Delta \equiv 0/Y$.
Let $A^e(X/Y)$ and $M^e(X/Y)$ be defined as in \S~\ref{cones.of.divs}.
Then, the following holds.
\begin{itemize}
    \item[1] The number of ${\rm Aut}(X/Y,\Delta)$-equivalence classes of faces of the cone $A^e(X/Y)$ corresponding to birational contractions or fiber space structures is finite.
    Moreover, there exists a rational polyhedral cone $\Pi$ which is a fundamental domain for the action of ${\rm Aut}(X/Y,\Delta)$ on $A^e(X/Y)$ in the sense that
    \begin{itemize}
        \item [a] $A^e(X/Y)=\bigcup \subs g \in {\rm Aut}(X/Y,\Delta). g_\ast  \Pi$; and
        \item[b] ${\rm Int}\Pi \cap g_\ast {\rm Int}\Pi = \emptyset$ unless $g_\ast =1$.
    \end{itemize}
    \item[2] The number of ${\rm PsAut}(X/Y,\Delta)$-equivalence classes of chambers $A^e(X/Y,\alpha)$ in $M^e(X/Y)$ corresponding to marked small $\qq$-factorial modifications $X' \rar Y$ of $X \rar Y$ is finite.
    Equivalently, the number of isomorphism classes over $Y$ of small $\qq$-factorial modifications of $X$ over $Y$ (ignoring the birational identification with $X$) is finite.
    Moreover, there exists a rational polyhedral cone $\Pi'$ which is a fundamental domain for the action of ${\rm PsAut}(X/Y,\Delta)$ on $M^e(X/Y)$.
\end{itemize}
\end{conconj}

In the statement of the conjecture, ${\rm PsAut}(X/Y,\Delta)$ denotes the gorup of pseudo-automorphisms of the pair $(X,\Delta)$ relative to $Y$.
Here, pseudo-automorphism means a birational automorphism that does not contract nor extract any divisor.
In particular, if $X \rar Y$ is a Calabi--Yau fiber space as defined in \S~\ref{ssect:cy.fibre.sp}, we have $\mathrm{Bir}(X/Y)={\rm PsAut}(X/Y)$, see \cite{Kaw97}*{\S~1}.

This is a very deep conjecture connecting the birational geometry of a log Calabi--Yau fibration to the structure of the (birational) automorphism group.
The intuition behind such connection is rooted in mirror symmetry and physics, see, for example \cite{MR1265317}, but it is still unclear how exactly to determine the existence of automorphism starting from the geometry of the cone of divisors.
Conjecture \ref{km.conj} is known to hold just in very few cases: Totaro proved it in dimension 2 \cite{Tot10}, Kawamata proved the relative case (i.e., $\dim Y > 0$) for threefold Calabi--Yau fiber spaces \cite{Kaw97}, and there are a few other cases known in dimension $>2$.

\section{Finiteness of models for elliptic Calabi--Yau fiber spaces} 
\label{section cone conj}

The results in this section are a higher-dimensional generalization of the results of \cite{Kaw97}*{\S~3}, originally stated for elliptic threefolds.
The main subtlety in passing to dimension higher than 3 is that the base of the elliptic fibration has a more complicated birational geometry:
in particular, such base may admit birational modifications in codimension 2, while in the case of elliptic threefolds the base is a surface and its birational geometry is completely determined by the set of exceptional divisors in the birational morphisms of interest.

\begin{lemma} 
\label{lemma factorization general}
Let 
$f \colon X \to Y$ 
be a Calabi--Yau fiber space of relative dimension 1.
Let 
$(f' \colon 
X' \to Y, 
\alpha)$ 
be a marked minimal model of $f$.
Let 
\begin{align*}
\xymatrix{
X' \ar[r]^{g'}
\ar@/^16pt/[rr]^{f'}
&
Y' \ar[r]^{h'} &
Y
}
\end{align*}
be a factorization of 
$f$
satisfying the conclusions of Lemma~\ref{lemma_vertical_divs}.
Then, the following hold:
\begin{enumerate}
\item
\label{opt:lemma.factorization.general.maximal}
given any marked minimal model 
$(\bar f \colon \overline X \to Y, \bar \alpha)$
of 
$f$,
there exists a uniquely determined factorization 
\begin{align}
\label{eqn:max.factorization}
\xymatrix{
\overline X \ar[r]^{\bar g} \ar@/^16pt/[rr]^{\bar f}& 
\overline Y  \ar[r]^{\bar h} & 
Y},
\end{align}
where 
$\bar g$ 
is a Calabi--Yau fiber space of relative dimension 1 and 
$\bar h$ 
is a projective birational morphism which satisfies the following maximality property:
any other factorization 
$\xymatrix{
\overline X \ar[r]^{\bar g'} & 
\overline Y'  \ar[r]^{\bar h'} & 
Y}$ 
with 
$\bar g'$
a Calabi--Yau fiber space of relative dimension 1 and 
$\bar h'$
a projective birational morphism
factors through $\overline{g}$,
i.e., 
there exists a projective birational morphism 
$l \colon \overline Y \to \overline Y'$ 
and a factorization
\begin{align*}
\xymatrix{
\overline X \ar[r]^{\bar g} \ar@/^35pt/[rrr]^{\bar f}
\ar@/^25pt/[rr]_{\bar g'}
& 
\overline Y \ar[r]^{l} 
\ar@/_20pt/[rr]_{\bar h}& 
\overline Y' \ar[r]^{\bar h'} & 
Y}.
\end{align*}
In particular, any factorization 
$\xymatrix{
X'' \ar[r]^{g''}
\ar@/^18pt/[rr]^{f''}
&
Y'' \ar[r]^{h''} &
Y
}$
of a marked minimal model 
$(f'' \colon X'' \to Y, \alpha')$
of 
$f$
satisfying the conclusions of 
Lemma~\ref{lemma_vertical_divs} also satisfies the maximality property just described;

\item
\label{opt:lemma.factorization.general.factorization}
there exists a birational contraction 
$\beta \colon Y' \drar \overline Y$
making the following diagram commute
\begin{align*}
\xymatrix{
X' \ar@{-->}[rrrr]^{\bar \alpha^{-1} \circ \alpha} \ar[d]^{g'}
& & & &
\overline X \ar[d]_{\bar g}
\\
Y' 
\ar@{-->}[rrrr]^{\beta}
\ar[rrd]^{h'}
& & & &
\overline Y
\ar[lld]_{\bar h'}
\\
& & Y & &
}
\end{align*}
where $\xymatrix{
\overline X \ar[r]^{\bar g} & 
\overline Y}$ is as in \eqref{eqn:max.factorization}.
Furthermore, up to replacing 
$g'\colon X' \to Y'$ 
with another Calabi--Yau fiber space 
$g''' \colon X''' \to Y'''$ 
satisfying the same properties and assumptions of the lemma, and such that 
$X'''$ (resp. 
$Y'''$) is isomorphic in codimension 1 to $X'$ (resp. $Y'$), then we can assume the map $\beta$ above is a morphism; 
\item
\label{opt:lemma.factorization.general.flops}
let 
$(\tilde f \colon \widetilde X \to Y, \tilde \alpha)$
be a marked minimal model of 
$f$
admitting a factorization 
$\xymatrix{
\widetilde X \ar[r]^{\tilde g}
\ar@/^16pt/[rr]^{\tilde f}
&
\widetilde Y \ar[r]^{\tilde h} &
Y
}
$
satisfying the conclusions of 
Lemma~\ref{lemma_vertical_divs}.
Then 
$\widetilde Y$ and $Y'$ are isomorphic in codimension 1.
In particular, they are connected by a sequence of flops over 
$Y$
with respect to any klt pair 
$(Y', \Delta')$
(resp. 
$(\widetilde Y, \widetilde \Delta)$)
induced by the canonical bundle formula for 
$f'$
(resp.
$\tilde f$); 
and,
\item
\label{opt:lemma.factorization.general.contraction}
if $D'$ is a $g'$-vertical prime divisor and $[D'] \neq 0 \in \nsr(X'/Y')$, then $D'$ is the exceptional divisor of a birational contraction of $X'$ over $Y'$.
\end{enumerate}
\end{lemma}

\begin{proof}
\begin{enumerate}
\item 
Since 
$\bar f \colon 
\overline X \to Y$ 
is a Calabi--Yau fiber space of relative dimension 1, a divisor is 
$\bar f$-semi-ample 
if and only if it is 
$\bar f$-nef 
and 
$\bar f$-effective.
Furthermore, if two $\bar f$-semi-ample divisors are not 
$\bar f$-big, 
neither is their sum, as 
$\bar f$-bigness 
is characterized by the intersection with a general fiber.
Thus, if 
$D_1$, 
$D_2$ 
are 
$\bar f$-semi-ample divisors that are not $f$-big, 
then so is 
$D_1+D_2$.
Thus, 
$D_1+D_2$ 
induces a factorization that dominates the ones induced by 
$D_1$ 
and 
$D_2$, 
respectively.
In particular, any two factorizations 
$\xymatrix{
\overline X \ar[r] & \overline Y_1 \ar[r] &
Y}$ 
and 
$\xymatrix{
\overline X \ar[r] & \overline Y_2 \ar[r] &
Y}$, 
where 
$Y_i$ 
is birational to 
$Y$, 
for $i=1, 2$, are dominated by a third factorization 
$\xymatrix{
\overline X \ar[r] & \overline Y_3 \ar[r] &
Y}$.
This shows the uniqueness of the maximal element.
Assuming that 
\begin{align*}
\xymatrix{
\overline X \ar[r] & \overline Y_4 \ar[r] &
\overline Y_5 \ar[r] &
Y}
\end{align*}
is a non-trivial factorization of 
$\bar f$, where 
$\xymatrix{
\overline Y_4 \ar[r] &
\overline Y_5 \ar[r] &
Y}$ 
is a composition of  birational morphisms, 
then
$\rho(\overline X/\overline Y_4)<
\rho(\overline X/\overline Y_5)<
\rho(\overline X/Y)$.
Since the relative Picard number is a positive integer, a maximal element must exist.

To show that any factorization 
$\xymatrix{
X'' \ar[r]^{g''}
\ar@/^18pt/[rr]^{f''}
&
Y'' \ar[r]^{h''} &
Y
}$
of a marked minimal model 
$(f'' \colon X'' \to Y, \alpha')$
of 
$f$
satisfying the conclusions of 
Lemma~\ref{lemma_vertical_divs} also satisfies the maximality property introduced in the statement of the lemma, it suffices to notice that Lemma~\ref{lemma_vertical_divs} implies that any  
$g''$-vertical 
divisor dominates a divisor on 
$Y''$.
Then, if there was a further factorization of the form 
$\xymatrix{
X'' \ar[r]^{g_m}
\ar@/^19pt/[rrr]^{f''}
&
Y^m \ar[r]^{h_m} &
Y'' \ar[r]^{h''} &
Y
}$,
since 
$Y''$
is 
$\mathbb Q$-factorial, 
$h_m$
would have to contract a prime divisor
$D_m$
and then 
$D''\coloneqq g_m^{-1}(D_m)$
would be a 
$g''$-vertical 
divisor such that 
$g''(D'')$
has codimension at least 
$2$
in 
$Y''$.

\item 
Let 
$\overline H$ 
be a relatively ample divisor on 
$\overline Y$ 
over 
$Y$.
Setting 
$M \coloneqq 
(\bar \alpha^{-1} \circ \alpha)_\ast ^{-1} \bar g^\ast  
\overline H$, 
then 
$M$ 
is movable over 
$Y$, 
as 
$\bar g^\ast \overline H$ 
is semi-ample and 
$(\bar \alpha^{-1} \circ \alpha)$ 
is a small birational map; 
a fortiori, 
$M$ 
is movable also over 
$Y'$, see Lemma~\ref{tower movable}.
Therefore, there is a sequence of flops 
$\gamma \colon X'\dashrightarrow \widehat X$ 
over 
$Y'$, 
making the strict 
transform
$\widehat M$
of 
$M$ 
on 
$\widehat X$ 
nef over 
$Y'$. 
Thus, 
$\widehat M$ 
is semi-ample over 
$Y'$.
Let 
$\widehat X \to \widehat Y \to Y'$ 
be the corresponding morphism. 
Since every divisor that is vertical for $X' \to Y'$ 
dominates a divisor in 
$Y'$, 
and since 
$\gamma$ 
is a small birational map, 
then any divisor that is vertical for 
$\widehat X \to Y'$ 
must also dominate a divisor on 
$Y'$.
Thus, for dimensional reasons, the morphism 
$\widehat{Y} \to Y'$ 
cannot contract any divisor, that is,  
$\widetilde{Y} \to Y'$ 
is a small birational morphism.
Since 
$Y'$ 
is 
$\mathbb Q$-factorial 
by assumption, see Lemma~\ref{lemma_vertical_divs}, then 
$\widehat{Y} \to Y'$ 
is an isomorphism.
But then $M = (g')^\ast  H '$ holds for some divisor ${H}'$ on $Y'$.
As $M$ is movable over $Y$, the same holds for ${H}'$.
Let $(Y', \Delta')$ be a klt pair induced by the canonical bundle formula for $g' \colon X' \to Y'$;
in particular $K_{Y'}+ \Delta' \sim_{\mathbb Q, Y}0$.
As 
$H'$ 
is movable over 
$Y$, 
running a relative 
$(K_{Y'}+ \Delta'+\epsilon H')$-MMP, 
for 
$0< \epsilon \ll 1$, 
over $Y$, 
we obtain a sequence of 
$(K_{Y'} + \Delta')$-flops 
$Y' \drar Y'''$ 
over 
$Y$ 
such that the strict transform 
$H'''$ 
of 
$H'$ 
on 
$Y'''$ 
is semi-ample over 
$Y$.
In particular, 
$Y' \drar Y'''$ 
is a birational contraction.
By construction, taking the relatively ample model of $H'''$ over 
$Y$ 
induces a morphism $Y''' \to \overline Y$.

Furthermore, by repeatedly applying~\cite{Fil20}*{Proposition 2.9}, there exists a 
$\mathbb Q$-factorial 
Calabi--Yau fiber space 
$g''' \colon X''' \to Y'''$ 
such that 
$X'''$ 
is isomorphic to 
$X'$ 
in codimension 1.

\item 
By
Lemma \ref{lemma_vertical_divs},
$Y'$ and $\widetilde Y$ are $\mathbb Q$-factorial.
Moreover, as both 
$Y'$ and 
$\widehat Y$ 
satisfy the assumptions of part \eqref{opt:lemma.factorization.general.factorization}, 
there exist birational contractions 
$\widetilde Y \drar \widehat Y$ 
and 
$\widehat Y \drar \widetilde Y$
which implies that they are isomorphic in codimension 1.
If
$(Y', \Delta')$
(resp. 
$(\widetilde Y, \widetilde \Delta)$)
is  induced by the canonical bundle formula for 
$f'$
(resp.
$\tilde f$), 
then 
$K_Y'+\Delta' \sim_{\mathbb Q, Y} 0$
(resp. 
$K_{\widetilde Y}+\widetilde \Delta \sim_{\mathbb Q, Y} 0$) holds, 
and 
$Y'$ and
$\widetilde Y$
are connected by a sequence of 
$(K_Y'+\Delta')$-flops
(resp.
$(K_{\widetilde Y}+\widetilde \Delta )$-flops).

\item
Let $D'$ be a prime and $g'$-vertical divisor such that $[D'] \neq 0 \in \nsr(X'/Y')$.
Then, $g'(D')=\bar D$ is a divisor on $Y'$ and $D'\ne \lambda  f^\ast\bar D$ for any $\lambda \in \rr$.
Thus, we may assume that $D'+D''=\lambda_0 f^\ast \bar D$, where $\lambda_0 \in \rr_{>0}$, $D', D'' \geq 0$ and they do not have components in common.
Moreover, every component of $D''$ dominates $\bar D$.
Then, $D'$ is $g'$-very exceptional in the sense of~\cite{Bir12}*{Definition 3.1} and it is contracted by running a $(\K X'. + \epsilon D')$-MMP over $Y'$, for $0 < \epsilon \ll 1$, since $\K X'. + \epsilon D' \sim_{\qq,Y'} \epsilon D'$.
\end{enumerate}
\end{proof}

\begin{remark} \label{rmk models 2 general}
We use the setup and the notation of 
Lemma~\ref{lemma factorization general}.
Proposition~\ref{prop models}
implies that there exist only finitely many marked birational models 
$\xymatrix{
\widetilde Y 
\ar[r]^{\tilde h} &
Y}$ 
that appear in a factorization of the form
$\xymatrix{
\widetilde X \ar[r]^{\tilde g} & 
\widetilde Y 
\ar[r]^{\tilde h} &
Y}$, 
where 
${\tilde h} \circ {\tilde g} \colon \widetilde X \rar Y$ 
is a relatively minimal model of $f$.
\end{remark}

Given an elliptic fibration 
$f \colon X \rar Y$, 
and a class 
$z \in \nsr(X/Y)$, 
we define
$\deg(z) \in \mathbb R$
to be the intersection number
$D \cdot F$
where 
$D$
is a 
$\mathbb R$-Cartier divisor 
such that 
$[D] = z \in \nsr(X/Y)$
and 
$F$
is a smooth fiber of 
$f$.

\begin{lemma} 
\label{lemma compact}
Let $X$ be a terminal $\qq$-factorial variety, and let 
$f \colon X \rar Y$ 
be a Calabi--Yau fiber space of relative dimension 1.
Let 
$\sigma \colon \mathrm{Bir}(X/Y) \rar \mathrm{GL}(\nsr(X/Y),\mathbb{Z})$ 
be the induced representation.
Then, the image of $\sigma$ contains an Abelian subgroup $G(X/Y)$ which is the image of a finite index subgroup of $H< \mathrm{Bir}(X/Y)$ that acts on the affine space 
$W(X/Y) \coloneqq \lbrace z \in \nsr(X/Y)/V(X/Y)\vert\deg(z)=1 \rbrace$ 
as a group of translations. 
Moreover, the quotient space 
$W(X/Y)/G(X/Y)$ 
is a real torus.
\end{lemma}

\begin{proof}
We follow the strategy of proof of~\cite{Kaw97}*{Lemma~3.5}.

Let $\eta \in Y$ be the generic point.
For a Weil divisor $D$ on $X$, we shall denote by $D_\eta$ its restriction to the schematic fiber $X_\eta$ of $X$ over $\eta$.
Since $\mathrm{Bir}(X/Y)={\rm Aut}(X_\eta)$, the degree of any divisor on $X$ is preserved under the push-forward by elements of $\mathrm{Bir}(X/Y)$.
Similarly, the subspace $V(X/Y)$ is fixed by the push-forward action by elements of $\mathrm{Bir}(X/Y)$.

{\bf Case 1}. 
{\it 
We prove the lemma under the additional assumption that $f$ has a rational section}.

Fix such a section $D_0$, which will serve as the origin for $X_\eta$.
By the structure of the automorphism group of an elliptic curve, the group of rational sections $M$, known as the Mordell--Weil group, can be identified as a subgroup $H$ of finite index of ${\rm Aut}(X_\eta)$ and hence of  $\mathrm{Bir}(X/Y)$ that acts via translations.

Let $\theta \in \mathrm{Bir}(X/Y)$ be an element corresponding to a rational section $D_1=\theta_\ast  D_0$.
For a $\mathbb K$-divisor $D$ with $\deg(D)=1$, 
$\theta_\ast  D_\eta - D_\eta \sim_{\mathbb K} D \subs 1,\eta. - D \subs 0,\eta.$ 
on $X_\eta$.
Thus, $\theta$ acts on $W(X/Y)$ as the translation by $[D_1 - D_0]$.
We define the map 
\begin{align*}
\xymatrix @R=.1pc 
{
\sigma' \colon
M \ar[r] & (\nsr(X/Y)/V(X/Y))_0
\\
D_1 \ar@{|->}[r] &
[D_1-D_0],
}
\end{align*}
where 
\[
(\nsr(X/Y)/V(X/Y))_0 \coloneqq 
\left \{
\gamma \in \nsr(X/Y)/V(X/Y)
\ \vert \
\deg(\gamma)=0
\right \}.
\]
It is immediate from its definition that $\dim (\nsr(X/Y)/V(X/Y))_0= \rho(X/Y)-v(X/Y)-1$.
\newline

\noindent
{\bf Claim}. 
Under the assumption that $D_0$ corresponds to the identity of $M$, $\sigma'$ is a homomorphism of Abelian groups and its image $G(X/Y)$ is a finitely generated subgroup.

\begin{proof}[Proof of the Claim]
Let $D_1$ and $D_2$ be two rational sections, and let $\theta_1$ and $\theta_2$ be the corresponding birational automorphisms.
To avoid confusion with the summation between divisors, we denote by $D_1 \star D_2$ the sum of the two sections in the Mordell--Weil group of $f$, that is, the group law of the elliptic curve $X_\eta$.
Since we have fixed $D_{0,\eta}$ as the identity of $X_\eta$, then $D_1 \star D_2\sim (D_1-D_0)_\eta +(D_2-D_0)_\eta +D_{0,\eta}=(D_1+D_2-D_0)_\eta$. Thus
\begin{align*}
D_1 \star D_2 \sim (D_1+D_2-D_0)_\eta,
\end{align*}
or, equivalently,
\begin{align} 
\label{eq_lin_eq}
(D_1 \star D_2)-D_{0,\eta} \sim (D_1-D_0)_\eta +(D_2-D_0)_\eta.
\end{align}
Since the linear equivalence in \eqref{eq_lin_eq} holds over an open subset of $Y$, and since we are considering the vector space $\nsr(X/Y)/V(X/Y)$, that is vertical divisors are negligible, then
\begin{align*}
[(D_1 \star D_2) - D_0]=[D_1-D_0]+[D_2-D_0] \in (\nsr(X/Y)/V(X/Y))_0.
\end{align*}
It just remains to show that $G$ is finitely generated.
Since 
$X$
is 
$\mathbb Q$-factorial and by definition of numerical equivalence,
the relative (over 
$Y$) 
first Chern class map is well defined on 
${\rm Cl}(X)$ with values in 
$\nsr(X/Y)$ 
and its image 
$G_1$ is a full rank lattice.
Moreover, the degree function also yields a group homorphism
$\deg \colon {\rm Cl}(X) \rar \mathbb Z$.
We denote by 
${\rm Cl}(X)_0$ 
its kernel.
Let us also notice that, by its definition, 
$V(X/Y)$ 
is spanned over 
$\mathbb R$
by classes of Weil divisors that have degree $0$.
Thus, the image 
$G_2$ 
of 
${\rm Cl}(X)_0$
via the first Chern class map into the quotient 
$(\nsr(X/Y)/V(X/Y))_0$
yields in turn a full rank lattice.
By the definition of 
$\sigma'$ 
then 
$G(X/Y):={\rm Im}(\sigma')$ 
is a subgroup contained in 
$G_2$. 
Thus, $G$ is finitely generated.
\end{proof}

For any Weil divisor $D$ on $X$, we define $D_{\deg=1} \coloneqq D-(\deg(D)-1)D_0$.
Then, $\deg(D_{\deg=1})=1$ and $f_\ast \O X.(D_{\deg=1})$ is a torsion free sheaf of rank 1 on $Y$.
In particular, by Riemann--Roch on $X_\eta$, 
there exists a rational section 
$S_{D_{\deg=1}}$ 
such that 
$[D_{\deg=1}]=[S_{D_{\deg=1}}] \in \nsr(X/Y)/V(X/Y)$.
Thus, 
$[D_{\deg=1}-D_0]=[S_{D_{\deg=1}} -D_0] \in G$.
Since we are free to choose 
$D$ 
to be any Weil divisor on 
$X$, 
it follows that $G(X/Y)$ 
is a 
$\zz$-module of maximal rank in $(\nsr(X/Y)/V(X/Y))_0$, i.e.,  $\rank{\rm Im}(\sigma')=\rho(X/Y)-v(X/Y)-1$.
Therefore, as $W(X/Y)$ is an affine space under the (fully faithful) action of $(\nsr(X/Y)/V(X/Y))_0$, the statement of the lemma follows.

{\bf Case 2}. 
{\it 
We prove the lemma without assuming the existence of a section}.

Let $d$ be the minimal positive integer such that $X_\eta$ has a divisor of degree $d$ defined over $k(Y)$.
Fix $D_0$ a horizontal divisor on $X$ such that $D_{0, \eta}$ has degree $d$.
Let $J_\eta$ denote the Jacobian of $X_\eta$.
As before, we will denote by $M$ 
the group 
of 
$k(Y)$-rational points of 
$J_\eta$.
Since 
$M$
acts on $X_\eta$ as a group of translation, 
$M$ 
naturally embeds in 
$\mathrm{Bir}(X/Y)$
and its image 
$H \subset \mathrm{Bir}(X/Y)$ 
has finite index.
Let 
$\theta \in M$ 
and let 
$D$ 
be a divisor with $\deg(D)=d$.
Then, $\theta_\ast  D_\eta -D_\eta \sim \theta _\ast  D \subs 0,\eta. - D \subs 0,\eta.$ on $X_\eta$.
Hence, dividing by $d$, we deduce that $\theta$ acts as a translation by $\frac{1}{d}[\theta_\ast D_0-D_0]$ on $W(X/Y)$.
As in Case 1, we define the morphism 
$\sigma' \colon M \rar (\nsr(X/Y)/V(X/Y))_0$
by 
$\sigma'(\theta)
\coloneqq
\frac{1}{d}[\theta_\ast  D_0-D_0]$.
As in Case 1, $\sigma'$ 
is a group homomorphism and its image 
$G(X/Y) \subset (\nsr(X/Y)/V(X/Y))_0$
is a full rank lattice.
Let us consider the group homomorphism 
$\phi \subs D_{0,\eta}. \colon J_\eta \rar J_\eta$ 
defined by 
$\phi \subs D_{0,\eta}.(x)
\coloneqq
\tau_x(D \subs 0,\eta.) - D \subs 0,\eta.$, 
where 
$\tau_x$ 
denotes the translation by 
$x \in J_\eta$.
Then, $\phi \subs D_{0,\eta}.$ is an \'{e}tale morphism of degree $d^2$, by the Theorem of the square for Abelian varieties, see \cite{Mum70}*{p. 59}.
For every (integral) divisor $D$ on $X$, we define $\qq$-divisor $D_{\deg=d}=D-(\frac{1}{d}\deg(D)-1)D_0$.
Then, $\deg(D_{\deg=d})=d$ and, by the minimality of $d$, $d\vert\deg(D)$, so that $D_{\deg=d}$ is actually a Weil divisor.
Again, by Riemann--Roch on $X_\eta$ and the minimality of $d$, the class $[D_{\deg=d}] \in \nsr(X/Y)/V(X/Y)$ can be represented by a prime divisor $S_{D_{\deg=d}}$ on $X$.
We regard the degree $0$ 
Weil divisor
$S_{D_{\deg=d}, \eta} - D \subs 0,\eta.$
on 
$X_\eta$
as a 
$k(Y)$-rational point
$\specialpoint
\in
J_\eta$.
The fiber 
$F_p 
\coloneqq
\phi \sups -1. \subs D_{0,\eta}.(\specialpoint)$ 
is a 
$0$-dimensional scheme defined over 
$k(Y)$
of length 
$d^2$.
In 
$J_{\overline{\eta}}
\coloneqq 
J_\eta 
\times_{{\rm Spec} \ k(Y)}
{\rm Spec} \ \overline{k(Y)}$, 
where 
$\overline{k(Y)}$
is the algebraic closure of 
$k(Y)$, 
$F_p  \times_{{\rm Spec} \ k(Y)}
{\rm Spec} \ \bar{k(Y)}$ 
is an effective divisor
of degree 
$d^2$.
On the other hand,
since 
$F_p$ 
is defined over $k(Y)$, 
the sum in $J_{\overline{\eta}}$of these
$d^2$
points is in turn a closed point of 
$J_{\overline{\eta}}$
defined over 
$k(Y)$, 
that is,
it is a 
$k(Y)$-rational point of 
$J_\eta$
which we denote by
$p'$.
The point 
$p'$ 
is by definition an element of 
$M$.
By definition, $\tau_{p'}(D \subs 0,\eta.) - D \subs 0,\eta. \sim d^2 (S_{D_{\deg=d},\eta} - D \subs 0,\eta.)$.
As 
$\tau_{p'}(D \subs 0,\eta.) - D \subs 0,\eta.=
d \sigma'({p'})$, then one can now conclude exactly as at the end of Case 1.
\end{proof}

\begin{theorem} 
\label{thm KM conj general}
Let $f \colon X \rar Y$ be a Calabi--Yau fiber space of relative dimension 1.
Then, there are only finitely many orbits for the action of $\mathrm{Bir}(X/Y)$ on:
\begin{enumerate}
    \item 
the set of chambers of 
${M}^e(X/Y)$:
\begin{align*}
\left\{ 
{A}^{e}(X/Y, \alpha)
\ 
\middle \vert 
\
(f' \colon X' \rar Y, \alpha)
\text{ is a marked minimal model of 
}
f
\right\};
\text{and,}
\end{align*} 

    \item 
the set of extremal faces of 
${M}^e(X/Y)$ 
induced by non-trivial factorizations of marked minimal models of 
$f$: 
\begin{align*}
\left\{
\alpha_\ast (g')^\ast
A^{e}(Z'/Y)
\
\middle \vert 
\
\begin{array}{l}
\xymatrix{
X'  \ar[r]_{g'} \ar@/^7pt/[rr]^{f'} &  Z' \ar[r]_{h'} & Y
}
\text{ is a non-trivial factorization}
\\
\text{of a marked minimal model
$(f' \colon X' \to Y, \alpha)$ 
of 
$f$}
\end{array}
\right\}.
\end{align*}  
\end{enumerate}
\end{theorem}

\begin{proof}
For the reader's convenience, we divide the proof into several steps.

We observe that:
\begin{itemize}
\item 
 replacing $f$ with a marked relatively minimal model of $f$ does not affect the conclusions of the theorem, cf.~\S~\ref{ssect:cy.fibre.sp};
\item 
since $f$ is a contraction of relative dimension 1, the only contractions that can factor $f$ are either birational models of $X$ or birational models of $Y$.
\end{itemize}

{\bf Step 0}.
{\it In this step, we make a first reduction and then we introduce the strategy of proof.}

By Lemma \ref{lemma_vertical_divs}, up to replacing $f$ with a relatively minimal model, we can assume that $f$ factors as 
\begin{align}
\label{diag.fact.Y'}
\xymatrix{X \ar@/^10pt/[rr]^f\ar[r]_g & Y' \ar[r]_h & Y,}
\end{align}
where every divisor that is $g$-vertical dominates a divisor in $Y'$, $h$ is birational, and $Y'$ is $\qq$-factorial.
To keep the notation light, we define 
\[
v\coloneqq v(X/Y), 
\quad
\rho\coloneqq\rho(X/Y), 
\quad 
\text{and}
\quad
k \coloneqq \rho(Y'/Y).
\]
By Lemma \ref{SES vertical divisors}, $g^\ast \nsr(Y'/Y) \subset V(X/Y)$.
If 
\begin{align*}
\label{diag.fact.Y'.2}
\xymatrix{
X 
\ar@/^10pt/[rr]^f
\ar[r]_{g'} & 
Y'' 
\ar[r]_{h'} & 
Y,}
\end{align*}
is another factorization of 
$f$ 
satisfying the same properties as the one 
in~\eqref{diag.fact.Y'}, 
then by 
Lemma~\ref{lemma_cones2} 
and 
Lemma~\ref{lemma factorization general}.\ref{opt:lemma.factorization.general.flops}, $Y', Y''$ are isomorphic in codimension 1 and 
$g^\ast \nsr(Y'/Y) = (g')^\ast \nsr(Y''/Y)$ in $\nsr(X/Y)$.
In particular, $g^\ast \nsr(Y'/Y)$ is an intrinsically defined $k$-dimensional subspace of $V(X/Y)$.
Let 
$(y_1, \ldots , y_k)$ 
be a basis of 
$g^\ast\nsr(Y'/Y)$.
Then, we may complete it to a basis 
$(y_1, \dots, y_k, y_{k+1}, \dots ,y_v)$ 
of 
$V(X/Y)$.
In turn, we complete this basis to a basis 
$(y_1, \dots, y_k, y_{k+1}, \dots ,y_v, y_{v+1}, \dots,  y_\rho)$ 
of 
$\nsr(X/Y)$.

By Corollary~\ref{corollary finite fiber spaces general}, there exist only finitely many extremal faces of 
$M^e(X/Y)$ 
corresponding to fiber space structures of a marked minimal model of 
$f$:
that proves the finiteness of the extremal faces of 
$M^e(X/Y)$ 
corresponding to a factorization of a marked minimal model of 
$f$ 
which lie on the boundary of the big cone.
Thus, in the remainder of the proof, we will focus on the extremal faces of 
$M^e(X/Y)$ 
corresponding to a birational contraction factoring 
$f$, 
that is, those extremal faces that intersect
$B(X/Y)$.

Our strategy for the proof of the theorem is to now proceed by induction on $v$.
Let us recall that in Lemma~\ref{lemma compact} we defined $W(X/Y) \coloneqq \lbrace z \in \nsr(X/Y)/V(X/Y)\vert\deg(z)=1 \rbrace$.
\newline

{\bf Step 1.}
{\it In this step, we prove the base case of the induction, that is, the case where $v=0$.}

If 
$v=0$, 
then 
$W(X/Y)=\{z\in \nsr(X/Y)\vert\deg z=1\}$.
We then prove the following claim.

\medskip

\noindent
{\bf Claim 1}.
If 
$v=0$, 
then 
$W(X/Y)\subset M^e(X/Y)\cap B(X/Y)$.
\begin{proof}[Proof of Claim 1]
Let 
$z\in W(X/Y)$.
As 
$\deg z=1>0$, 
then 
$z\in B(X/Y)$, 
thus its class can be represented by an effective $\mathbb Q$-divisor $D$, that is, 
$z=[D]\in B^e(X/Y)$.

Now, let 
$D$ 
be a divisor with 
$\deg D>0$. 
If 
$F$ 
is a component of the relative stable base locus of 
$D$ 
over 
$Y$, 
then 
$F$ 
is a vertical divisor. 
As 
$v=0$, 
all vertical divisors are movable over 
$Y$, 
since they are numerically equivalent to the pull-back of a divisor on 
$Y$. 
It then follows easily that 
$D$ 
itself is movable over 
$Y$.
\end{proof}
To prove the statement of this step, as $W(X/Y)/G(X/Y)$ is compact, it suffices to invoke Lemma~\ref{lemma chambers}.
\newline

We will now proceed to prove the inductive step:
we assume that $v>0$ and that the inductive hypothesis holds.
We define $I(X/Y)$ to be the collection of all cones in $M^e(X/Y)$ of the form 
$\alpha_\ast (g_{\widetilde Z}^\ast {A}(\widetilde Z/Y))$
for a marked minimal model 
$(\widetilde f \colon \widetilde X \rar Y,\alpha)$ of $f$ 
factoring as 
\begin{equation}
\label{diag.fact.f.tilde2}
\xymatrix{
\widetilde{X} \ar[r]_{g_{ \widetilde Z}} \ar@/^10pt/[rrr]^{\widetilde{f}} & 
\widetilde{Z} \ar[r] & 
\widetilde{Y} \ar[r]_{h_{\widetilde Y}} & 
Y,
}
\end{equation}
where 
$\widetilde Y \rar Y$ 
is a birational contraction that is not an isomorphism, whereas
$\widetilde X \rar \widetilde Z$ 
is a birational morphism, where, in this last case, an isomorphism is also allowed.
If 
$\widetilde X \rar \widetilde Z$ 
is an isomorphism, then 
$\alpha_\ast (g_{\widetilde Z}^\ast {A}(\widetilde Z/Y))=A(X/Y, \alpha)$
which is the interior of a chamber of $M^e(X/Y)$.
If $\widetilde X \rar \widetilde Z$ is not an isomorphism, then the cone $\alpha_\ast (g_{\widetilde Z}^\ast {A}(\widetilde Z/Y))$ is the relative interior of an extremal face of $M^e(X/Y)$ corresponding to a birational contraction, that is, the extremal face intersects $B(X/Y)$.
\newline

{\bf Step 2.}
{\it In this step, we show that the theorem holds for those chambers and extremal faces of $M^e(X/Y)$ belonging to the collection $I(X/Y)$.}

Let $(\widetilde f \colon \widetilde X \rar Y,\alpha)$ be a marked minimal model together with a factorization of $\widetilde f$ as in~\eqref{diag.fact.f.tilde2}.
Then, the cone 
$\alpha_\ast (g_{\widetilde Z}^\ast {A}(\widetilde Z/Y))$ 
belongs to 
$I(X/Y)$.
By Lemma~\ref{SES vertical divisors}, 
$v(\widetilde X/\widetilde Y)< v(\widetilde X /Y)=v$.
Thus, we can apply the inductive hypothesis to the morphism
$\widetilde X \rar \widetilde Y$.
Then, there are only finitely many orbits for the action of $\mathrm{Bir}( \widetilde X/ \widetilde Y)$ on:
\begin{enumerate}
    \item[(i)] 
the chambers of $M^e(\widetilde X/ \widetilde Y)$; and
    \item[(ii)] 
the extremal faces of $M^e( \widetilde X/ \widetilde Y)$ induced by factorizations of a marked minimal model 
of $\widetilde  X \rar \widetilde Y$.
\end{enumerate}
By the inductive hypothesis and since 
\begin{align*}
\lefteqn{\overbrace{\phantom{\mathrm{Bir}(X/Y)=\mathrm{Bir}(\widetilde X/Y)}}^{\text{$\alpha$ is an isom. in codim. 1}}}
\mathrm{Bir}(X/Y)=\underbrace{\mathrm{Bir}(\widetilde X/Y)=\mathrm{Bir}(\widetilde X/ \widetilde Y)}_{\text{by Lemma~\ref{lemma_rel_bir}}},
\end{align*}
then there are only finitely many orbits in (i-ii) also for the action of $\mathrm{Bir}(X/Y)$ on $\nsr(\widetilde X/ \widetilde Y)$.
Hence, up to this action, there are finitely many marked minimal models
$(f_i \colon X_i \rar Y, \alpha_i)$, $i=1, \dots, s$  of $f$ admitting a factorization 
$$
\xymatrix{X_i \ar[r] \ar@/^10pt/[rr]^{f_i} & \widetilde Y \ar[r]_{h_{\widetilde Y}} & Y}.
$$
Since by Lemma \ref{finiteness models} there exist only finitely many birational models ${h_{\widetilde Y}} \colon \widetilde Y \rar Y$ that may appear in a factorization of marked minimal models of $f$, then there are finitely many orbits of the action of $\mathrm{Bir}(X/Y)$ on 
\begin{enumerate}
\item[(i')]
the chambers of $M^e(X/Y)$ corresponding to marked minimal models of $f$ admitting a non-trivial factorization through a higher birational model of $Y$; and
\item[(ii')]
the extremal faces of $M^e(X/Y)$ induced by factorizations of a marked minimal model as in~\eqref{diag.fact.f.tilde2}.
\end{enumerate}
By Remark \ref{rmk models 2 general}, we may restrict our attention to the chambers and the faces thereof not corresponding to fiber space structures.
\newline

To conclude the proof, we study orbits the of 
$\mathrm{Bir}(X/Y)$ 
on the extremal faces and chambers of 
$M^e(X/Y)$ 
not contained in 
$I(X/Y)$.
To this end, we define 
$\widehat{I}(X/Y) \subset \nsr(X/Y)$ 
to be the union of all the cones that are contained in 
$I(X/Y)$, 
that is, of all the cones of the form $\alpha_\ast(g_{\widetilde Z}^\ast A(\widetilde Z /Y))$,
where
$(\widetilde f \colon \widetilde X \rar Y,\alpha)$ 
is a marked minimal model of 
$f$
admitting a factorization as in \eqref{diag.fact.f.tilde2}.
Moreover, we define
\begin{align*}
J(X/Y) \coloneqq \lbrace z \in \nsr(X/Y)\vert\deg(z)=1, \; z \in {M}^e(X/Y) \setminus \widehat{I}(X/Y) \rbrace.
\end{align*}

{\bf Step 3}. 
{\it In this step, we show that $J(X/Y)$ is closed in $\nsr(X/Y)$}.

The condition $\deg (z)=1$ is clearly a closed one. 
The set $\{z\in  {M}^e(X/Y)\vert\deg (z)=1\}$ is also closed, as every element is big over $Y$ and hence $\rr$-linearly equivalent to an effective divisor over $Y$.

Let 
$z$ 
be a point in the closure of 
$J(X/Y)$: 
we will show that 
$z \not \in \widehat I (X/Y)$.
To this end, we assume that $z \in \widehat I (X/Y)$ and we shall proceed to obtain a contradiction.

By Lemma \ref{lemma chambers}, in a neighborhood $U$ of $z$, $M^e(X/Y)$ decomposes as a finite union of chambers 
\begin{align*}
U \cap M^e(X/Y)
=
\bigcup_{t=1}^b
A^e(X/Y,\alpha_t).
\end{align*}
Let us fix one of the chambers appearing in the decomposition above, which we denote by 
$A^e(X/Y,\alpha_t)$.
If 
$z$ 
is in the interior of 
$A^e(X/Y, \alpha_t)$, 
then 
$A(X/Y, \alpha_t)$ 
is one of the cones in 
$I(X/Y)$, 
since 
$z \in \widehat I (X/Y)$.
In particular, 
$z$ 
belongs to the interior of 
$\widehat I (X/Y)$ 
and we immediately obtain the sought contradiction, as 
$z$ 
is a limit point for 
$J(X/Y)$.
Thus, 
$z$ 
must belong to the relative interior of a face of the form 
$\alpha_\ast l_1^\ast A(X''/Y)
\subset
A^e(X'/Y, \alpha)$ 
induced by a factorization of the form 
$\xymatrix{X' \ar[r]^{l_1} & X'' \ar[r]^{l_2} & Y}$, 
where $l_1$ is a birational morphism which is not the identity.
As $X' \rar Y$ is a Calabi--Yau fiber space of relative dimension 1, then by~\cite{BCHM},
$X''$ can be constructed as the relatively ample model of the relatively big and movable class $z$.
But then, as $z \in \widehat I (X/Y)$, we have that the cone $\alpha_\ast l_1^\ast A(X''/Y)$
is an element of $I(X/Y)$, and $X'' \rar Y$ admits a non-trivial factorization $X'' \rar Y'' \rar Y$, which in turn induces a non-trivial factorization $X' \rar Y'' \rar Y$.
Thus, in a neighborhood of $z$, we have $A^e(X'/Y,\alpha) \subset \widehat I (X/Y)$.
Since the same reasoning can be used for any other chamber $A^e(\tilde X/Y,\alpha)$ arising from the decomposition claimed in Lemma \ref{lemma chambers}, then $z$ must be in the interior of $\widehat I (X/Y)$, thus reaching the sought contradiction.

Let us underline here that, in the definition of 
$I(X/Y)$,
we must use cones of the form 
${A}(\widetilde Z/Y)$ 
and not of the form
${A}^e(\widetilde Z/Y)$:
indeed, if we chose instead to work with ${A}^e(\widetilde Z/Y)$, it would not be true that the set 
$J(X/Y)$ 
is closed in 
$\nsr(X/Y)$.
In fact, to define 
$J(X/Y)$, 
we wish to consider all birational models of 
$X$ 
over 
$Y$
for which the structure morphism to 
$Y$
does not admit a factorization through a higher birational model of 
$Y$.
If 
$I(X/Y)$ 
contained cones of the form 
$\alpha_\ast (g_{\widetilde Z}^\ast A^e(\widetilde Z/Y))$,
where $\widetilde Z$ is part of a factorization as in \eqref{diag.fact.f.tilde2},
it may happen that the boundary of such cone contains as an extremal face 
the nef cone of another birational model 
$\widetilde Z'$ 
of 
$X$
over 
$Y$
which instead cannot be factorized in any way.
\newline

{\bf Step 4}.
{\it 
In this step, we show that, in order to prove the theorem, it suffices to show that the map 
$p \colon J(X/Y) \rar W(X/Y)$ 
is proper.
}

By Lemma~\ref{lemma compact}, there exists a finite index subgroup 
$H < \mathrm{Bir}(X/Y)$ 
such that the image 
$G(X/Y)$ 
of $H$ under the natural representation $\sigma \colon \mathrm{Bir}(X/Y) \to\mathrm{GL}(\nsr(X/Y), \mathbb Z)$
acts on $W(X/Y)$ as a group of translations and $W(X/Y)/G(X/Y)$ is a compact torus.
Thus, $G(X/Y)$ acts also on $J(X/Y)$: indeed, the action of $\mathrm{Bir}(X/Y)$ on $\nsr(X/Y)$ preserves
\begin{itemize}
\item 
the degree of a divisor on the generic fiber of $f$; and
\item 
the property that a marked minimal model $(\widetilde f \colon \widetilde X \to Y, \alpha)$ of $f$ admits a non-trivial factorization of $\widetilde f$.
\end{itemize}
Thus, there exists a natural morphism $J(X/Y)/G(X/Y) \rar W(X/Y)/G(X/Y)$, where the latter is compact.
The properness of of the map 
$p \colon J(X/Y) \rar W(X/Y)$, 
in turn, implies the properness of 
$J(X/Y)/G(X/Y) \rar W(X/Y)/G(X/Y)$, 
and, in particular, the compactness of 
$J(X/Y)/G(X/Y)$.
Then, the claim follows by combining the local finiteness of the movable cone inside the big cone and the compactness of 
$J(X/Y)/G(X/Y)$.
In particular, the claim follows from Lemma \ref{lemma chambers}: 
indeed, the faces on the boundary of the big cone are taken care of by Corollary~\ref{corollary finite fiber spaces general}.
\newline

{ \bf Step 5}.
{\it 
In this step, we verify the properness of the map $p$.
}

Since the topological spaces of interest have the Heine--Borel property, we can check the properness of $p \colon J(X/Y) \rar W(X/Y)$ using sequences.
Arguing by contradiction, we assume that there exists a  sequence $(z_n) \subs n \in \mathbb N. \subset J(X/Y)$ such that the sequence $(p(z_n)) \subs n \in \mathbb N.$ converges in $W(X/Y)$ whereas $(z_n) \subs n \in \mathbb N.$ does not admit any convergent subsequence.
Both of these conditions are not affected by passing to a subsequence of $(z_n) \subs n \in \mathbb N.$
and of
$(p(z_n)) \subs n \in \mathbb N.$
relative to the same subset of indeces.

For all $n \in \mathbb N$, we write $z_n = \sum \subs i=1.^\rho a_n^i y_i$ and we set 
$$
w_n \coloneqq \sum \subs i=1.^k a^i_n y_i, 
\qquad 
x_n \coloneqq \sum \subs i=k+1.^\rho a^i_n y_i, 
\quad
\text{and}
\quad 
t_n \coloneqq \sum \subs i=v+1.^\rho a^i_n y_i.
$$
For all 
$n \in \nn$, 
$z_n=x_n+w_n$,  
$z_n$,  
$t_n$,
and 
$x_n$ 
are big over 
$Y$, 
as well as over any birational model 
$\widetilde Y \to Y$
of $Y$.
By construction, $\mathrm{span}(y \subs v+1.,\ldots , y \subs \rho.)$ maps isomorphically onto $\nsr(X/Y)/V(X/Y)$.
\newline

{\bf Step 5.1}.
{\it
In this step, we show that the sequence 
$(x_n)_{n\in \mathbb N}$ 
contains a converging subsequence.
}

Let $Y'$ be the higher model of $Y$ defined in Step 0.
For all $n \in \mathbb N$, $[z_n] \in M^e(X/Y') \subset \nsr(X/Y')= \nsr(X/Y)/\nsr(Y'/Y)$, by Lemma \ref{tower movable}.
Here $[ - ]$ indicates the equivalence class in the quotient.
As $(y_1, \ldots , y_k)$ is a basis of $\nsr(Y'/Y)$, $[z_n]=[x_n] \in \nsr(X/Y')$.
Thus, $([z_n])_{n\in \mathbb N} \subset \nsr(X/Y')$ contains a converging subsequence if and only if $(x_n)_{n\in \mathbb N}$ does.

We then argue as in~\cite{Kaw97}*{proof of Theorem~3.6} and look at the intersection numbers with the general fibers of the $g$-exceptional divisors over their images in $Y'$.
By Lemma \ref{hyperplane movable}, this can be reduced to a lower-dimensional question by considering a very general hyperplane section of $Y'$.
Thus, proceeding inductively, we can reduce to the case when $Y'$ is a surface and $X$ a threefold, which is treated in \cite{Kaw97}*{proof of Theorem 3.6}.
Thus, $(x_n)_{n\in \mathbb N}$ admits a converging subsequence.
\newline

Since we are assuming that $(z_n) \subs n \in \nn.$ contains no convergent subsequence, by Step 5.1, the same must hold for $(w_n) \subs n \in \nn.$, in view of the discussion above.
As both of these conditions are not affected by passing to a subsequence, then we pass to the subsequence of $(x_n)\subs n \in \nn.$ whose existence was shown in Step 5.1 and we also pass to the subsequences of $(z_n) \subs n \in \nn.$, $(w_n) \subs n \in \nn.$ corresponding to the same indices.
Hence, we can assume that 
$(x_n) \subs n \in \nn.$ 
is converging, while 
$(z_n) \subs n \in \nn.$, 
$(w_n) \subs n \in \nn.$
do not contain any convergent subsequence.

For all $n\in \mathbb N$, we set $w'_n$ to be the unique element of $\nsr(Y'/Y)$ such that $g^\ast(w'_n)=w_n$.
As $(w_n) \subs n \in \nn.$
does not contain any convergent subsequence, also
$(w'_n) \subs n \in \nn.$
shall not contain any convergent subsequence.
\newline

{\bf Step 5.2}
{\it 
In this step, we show that there exists a birational contraction 
$Y' \dashrightarrow Y_N$ 
over 
$Y$ 
which is the outcome of a run of the MMP for a suitable subsequence of $(w'_n) \subs n \in \nn.$.
}
\newline
Since $\K Y'. + \Delta' \sim \subs \qq,Y. 0$, $(Y',\Delta')$ is klt, and $Y' \rar Y$ is birational, for any divisor class in $\nsr(Y'/Y)$ we can run a relative MMP over $Y$, by~\cite{BCHM}*{Corollary~1.3.2}.
Hence, for all $n\in \mathbb N$, the $w'_n$-MMP over $Y$ can be run and it must terminate with a relatively minimal model for $w'_n$.
By Lemma~\ref{finiteness models}, there are just finitely many marked birational models of 
$Y$ 
that can appear in those runs of the MMP.
Furthermore, by the negativity lemma, for a fixed 
$n \in \mathbb N$ 
no model can appear more than once in the 
$w'_n$-MMP.
Thus, up to passing to a subsequence  $(w'_{n_k})_{k \in \mathbb N} \subset (w'_n)_{n \in \mathbb N}$, we may assume that there exists a sequence of divisorial contractions and isomorphisms in codimension 1
\begin{align}
\label{diag.MMP.w'_n}
\xymatrix{
Y'=:Y_0 \ar@{-->}[r]^{\psi_0} \ar[drr] &
Y_1 \ar@{-->}[r]^{\psi_1} \ar[dr]&
\dots \ar@{-->}[r]^{\psi_{N-2}} &
Y_{N-1} \ar@{-->}[r]^{\psi_{N-1}} \ar[dl]&
Y_{N} \ar[dll]
\\
& & Y & & 
}
\end{align}
which yields, for any $k\in\mathbb N$, a run of the $w'_{n_k}$-MMP over $Y$.
In particular, for all $k\in\mathbb N$, $(\psi_{N-1} \circ \dots \circ \psi_0)_\ast w'_{n_k}$ is nef over $Y$.
Moreover, for all $i=1, \dots, N$, by~\cite{dCS17}*{Propostion~2.26}, there exists a marked minimal model $(f_{i} \colon X_{i} \rar Y, \alpha_i)$ of $f$ together with a factorization 
\begin{align*}
 \xymatrix{
 X_i \ar[r]^{l_i} \ar@/^14pt/[rr]^{f_i} & 
 Y_i \ar[r] & 
 Y
 }
\end{align*}
such that the induced diagram
\begin{align*}
    \xymatrix{
   X_i \ar[d]^{l_i} 
   \ar@{-->}[rrr]^{\phi_i\coloneqq \alpha_{i+1}^{-1} \circ \alpha_i}_{\text{isom. in codim. 1}}
   & & &
   X_{i+1} \ar[d]^{l_{i+1}}
   \\
 Y_i \ar@{-->}[rrr]_{\psi_i} 
 & & &
 Y_{i+1} 
    }
\end{align*}
is commutative.
Finally, passing to a suitable subsequence $(w'_{n_{k_j}})_{j \in \mathbb N} \subset (w'_n)_{n \in \mathbb N}$, we can assume that, for all $j\in \mathbb N$, the $w'_{n_{k_j}}$ all have the same ample model  $\psi_{N} \colon Y_N \rar Y_{N+1}$ over $Y$.
\newline

We pass to the subsequence of $(w'_n)\subs n \in \nn.$ whose existence was shown in Step 5.2; 
we also pass to the subsequences of
all the other sequences involved in the proof corresponding to the same indices.

We set 
$w_{n,0} \coloneqq w_n$,
$w'_{n,0} \coloneqq w'_n$,
$z_{n,0} \coloneqq z_n$,
and
$x_{n,0} \coloneqq x_n$.
We define inductively 
for $i=0, \dots, N-1$,
\begin{align}
\label{eqn.defn.various.seqs}
&\nsr(X_{i+1}/Y )\ni z_{n, i+1} \coloneqq
(\phi_i)_\ast z_{n, i},
\quad 
&\nsr(Y_{i+1}/Y) \ni w'_{n,i+1}\coloneqq 
(\psi_i)_\ast w'_{n,i}, 
\\ 
\nonumber
&\nsr(X_{i+1}/Y )\ni w_{n,i+1}\coloneqq 
l_{i+1}^\ast w'_{n,i+1}, 
\quad
&\nsr(X_{i+1}/Y) \ni x_{n, i+1}\coloneqq 
(\phi_{i})_\ast(x_{n,i}+w_{n,i})-w_{n,i+1}.
\end{align}
With these definitions, since 
$[z_n, 0]=[x_{n, 0}] \in \nsr(X/Y')$, 
then, for all $i=0, 1, \dots, N$,
\begin{align}
\label{eqn.prop.x_{n,i+1}}
x_{n, i}=
z_{n,i}-w_{n, i}
\qquad 
\text{and}
\qquad 
 [z_{n,i}] = [x_{n,i}] \in \nsr(X/Y_{i}).   
\end{align}
As $\psi_N$ is the ample model for the $w'_{n}$, we define $w'_{n, N+1}$ to be the only element of $\nsr(Y_{N+1}/Y)$
such that
$w'_{n, N}=\psi_{N}^\ast w'_{n, N+1}$; 
thus, we have
$$
[z_{n, N}] = [x_{n,N}] \in \nsr(X/Y_{N+1}).
$$

{\bf Step 5.3.}
{\it
In this step, we show that for all $i=0, 1, \dots, N$, there exits a converging subsequence $(x_{n_k,i}) \subs k \in \nn. \subset (x_{n,i}) \subs n \in \nn.$ 
which can be chosen independently of $i$.
}

We proceed to prove the claim by induction on $i=0, \dots, N$.

By Step 5.1, the claim is true for $i=0$ as $x_{n,0}=x_n$.
Hence, we can assume that $i>0$ and that we have converging subsequences $(x_{n_k, l})_{k \in \mathbb N}$ for all $l=0, \dots, i-1$, corresponding to the same set of indices.

If $i\leq N$ and $\psi_{i-1}$ in~\eqref{diag.MMP.w'_n} is a flip, there is nothing to prove, as $x_{n,i}=(\phi_{i-1})_\ast x_{n,i-1}$, by definition,  in~\eqref{eqn.defn.various.seqs}, 
and since $(\phi_{i-1})_\ast$ descends to an isomorphism between 
$\nsr(X_{i-1}/Y_{i-1})$ 
and 
$\nsr(X_{i}/Y_{i})$, 
cf. Lemma~\ref{lemma_cones2}.

If $i\leq N$ and $\psi_{i-1}$ in~\eqref{diag.MMP.w'_n} is a divisorial contraction, denoting $E_{i-1}$ its exceptional divisor, then 
$w'_{n,i-1}-\psi_i^\ast w'_{n,i}=c_{n,i-1} [E_{i-1}] \in \nsr(Y_{i-1}/Y)$, $c_{n,i-1} >0$.
Since $c_{n, i-1} >0$, by Lemma \ref{movable bir 2} 
$c_{n,i-1} E_{i-1} \not \in \overline{M}(Y_{i-1}/Y_{i})$.
On the other hand, setting 
$F_{i-1}\coloneqq (\psi_{i-2}\circ \dots\circ \psi_1 \circ g)^\ast E_{i-1}$,
$[z_n] = 
[x_{n,i}]= 
[x_{n,i-1} +c_{n,i-1}  F_{i-1}] \in \nsr(X/Y_{i})$.
By Lemma \ref{tower movable}, 
$[x_{n,i-1} +c_{n,i-1}  F_{i-1}] \in \overline{M}(X/Y_{i})$.
Since 
$(x_{n_k, i-1}) \subs k \in \nn.
\subset
(x_{n, i-1}) \subs n \in \nn.$ 
is a convergent subsequence (independent of $i$) and $c_{n, i-1} F_{i-1}$ is not movable, then 
$(c_{n_k, i-1})_{k \in \mathbb N}$ 
must be a sequence of positive real numbers bounded from above:
hence, it contains a converging subsequence
$(c_{n_{k_j}, i-1})_{j \in \mathbb N}$.
Hence, taking the  subsequences
$(x_{n_{k_j}, l})_{k \in \mathbb N}$ for all $l=0, \dots, i$ proves the inductive step in this case.
\newline

We pass to the subsequences of 
$(x_{n, i})\subs n \in \nn.$, 
$i=0, \dots, N$ 
whose existence was shown in Step 5.3;
we also pass, for all $i$, to the subsequences of 
$(z_{n, i}) \subs n \in \nn.$, 
$(w_{n, i}) \subs n \in \nn.$,
$(w'_{n, i})\subs n \in \nn.$, 
corresponding to the same indices.
Since for all $i$,
$(\phi_i)_\ast \colon \nsr(X_i/Y) \rar \nsr(X_{i+1}/Y)$ 
is an isomorphism,
then for all $i$, 
$(z_{n, i}) \subs n \in \nn.$, 
(resp.
$(w_{n, i}) \subs n \in \nn.$,
$(w'_{n, i})\subs n \in \nn.$)
does not contain any converging subsequence.
\newline

{\bf Step 5.4}.
{\it
In this step, we show that $Y_{N+1} \rar Y$ is not an isomorphism.
}

As 
$(x_{n,N}) \subs n \in \nn.$ 
is convergent, and 
$z_{n,N}=x_{n, N} + w_{n, N}$, 
cf.~\eqref{eqn.prop.x_{n,i+1}}, then $Y_{N+1} \rar Y$ is not an isomorphism:
otherwise, $z_{n, N}=x_{n, N} \in \nsr(X_N/Y)$ and
$(z_{n, N})\subs n \in \nn.$ would be convergent.
\newline

{\bf Step 5.5}.
{\it In this step, we show that there exists a marked minimal model 
$(f_{N+1} \colon X_{N+1} \rar Y, \alpha_{N+1})$ 
of $f$ together with a factorization
$$
\xymatrix{
X_{N+1}\ar[r]_{l_{N+1}} \ar@/^10pt/[rr]^{f_{N+1}} & Y_{N+1} \ar[r] & Y}
$$
such that for infinitely many $n \in \nn$, 
$[x_{n, N+1}]$ 
is nef over $Y_{N+1}$, where
$\phi_{N}\coloneqq 
\alpha_{N+1}^{-1} \circ \alpha_{N}$
and
$x_{n, N+1} \coloneqq (\phi_{N})_\ast x_{n, N}$.
}

Since 
$(x_{n, N})_{n \in \nn}$ 
is convergent, calling 
$\overline x_N \in \nsr(X_{N}/Y)$ 
its limit, then, by construction 
$\overline x_N, x_{n, N}$ 
are big over both 
$Y$, 
$Y_{N+1}$, 
for all 
$n \in \nn$.
Hence, Lemma~\ref{lemma chambers} applied to $X_N \to Y_{N+1}$, implies that there exists a subsequence 
$(x_{n_k, N})_{k \in \nn} \subset (x_{n, N})_{n\in \nn}$ and an isomorphism in codimension 1
$\phi_N \colon X_N \dashrightarrow X_{N+1}$ such that $(\phi_N)_\ast x_{n_k, N}$ is nef for all $k \in \nn$.
\newline

We pass to the subsequence 
$(x_{n_{k}, N+1})\subs k \in \nn.
\subset 
(x_{n, N+1})\subs n \in \nn.$ 
just defined; 
we also pass to the subsequences of
all the other sequences involved in the proof corresponding to the same indices.
We define for all $n \in \nn$,
$z_{n, N+1}\coloneqq
(\phi_{N})_\ast z_{n, N}$,
$w_{n, N+1} \coloneqq 
l_{N+1}^\ast w'_{n, N+1}$.
\newline

{\bf Step 5.6}.
{\it In this step, we show that there exists a positive real number $\epsilon_{N+1}$ such that 
for any curve $C \subset X_{N+1}$ contained in the fiber of $f_{N+1}$, $x_{n, N+1} \cdot C > -\frac{2\dim X_{N+1}}{\epsilon_{N+1}}$ .}

By Step 5.3 and Lemma \ref{lemma divisors} applied to $f_{N+1}$, there exist effective divisors
$D_{n, N+1}$, $n \in \mathbb N$ 
big over $Y$ such that $x_{n, N+1}=[D_{n, N+1}] \in \nsr(X_{N+1}/Y)$ and 
$0 < \epsilon_{N+1} < 1$ 
(independent of $n$) such that 
for all $n \in \nn$, $(X_{N+1},\epsilon_{N+1} D_{n, N+1})$ is klt.
Hence, the conclusion follows by the Cone Theorem.
\newline

We fix an integer 
$T_{N+1} \geq 2\frac{2\dim X_{N+1}}{\epsilon_{N+1}}$,
and fix a Cartier divisor $H_{N+1}$ on $Y_{N+1}$ ample over $Y$.
\newline

{\bf Step 5.7}.
{\it
In this step, we show that for infinitely many $n \in \mathbb N$, $z_n \in \widehat I(X/Y)$.
This prompts the desired contradiction, since $z_n \in J(X/Y)$ and hence it cannot be an element of $\widehat I(X/Y)$ and, hence, concludes the proof.
}

We must distinguish two separate cases at this point.
\newline

{\bf Case 5.7.a}. 
{\it
In this case we assume that there exists a subsequence 
$(w'_{n_k, N+1})\subs k \in \nn.
\subset 
(w'_{n, N+1})\subs n \in \nn.$ 
such that for all 
$k \in \mathbb N$,
$w'_{n_k, N+1}- T_{N+1}H_{N+1}$
is nef over $Y$.
}

Since 
$H_{N+1}$ 
is ample over 
$Y$ 
and Cartier, then for any irreducible curve 
$C' \subset Y_{N+1}$ 
in the fibers of the structure morphism
$Y_{N+1} \to Y$, 
$H_{N+1} \cdot C'$
is a positive integer.
In turn, the assumption in Case 5.7.a implies that for any irreducible curve $C' \subset Y_{N+1}$ in the fibers of $Y_{N+1} \to Y$, 
$w'_{n_k, N+1} \cdot C' \geq T_{N+1}$ holds.
Therefore, for any irreducible curve $C \subset X_{N+1}$ in the fibers of $f_{N+1}$, 
$w_{n_k, N+1} \cdot C'$ is either $0$ or $\geq T_{N+1}$.
As $X_{N+1}$ has been chosen so that $x_{n, N+1}$ is nef over $Y_{N+1}$, cf. Step 5.5, then, by Step 5.6,
$x_{n_k, N+1} + \frac{1}{2}w_{n_k, N+1}$
is nef over $Y$.
In turn, 
$z_{n_k, N+1}=x_{n_k, N+1} + w_{n_k, N+1}$ is nef and big over $Y$, and hence it is relatively semi-ample over $Y$.
Let 
$f_{a,n_k} \colon X_{a,n_k} \rar Y$
denote its relatively ample model
over 
$Y$.
We claim that it
admits a factorization of the form 
\begin{equation}\label{eq:factorization_contradiction}
\xymatrix{
X_{N+1} 
\ar[r]_{r_{n_k}}
\ar@/^10pt/[rrr]^{f_{N+1}}
&
X_{a, n_k}
\ar[r]
&
Y_{N+1}
\ar[r]
&
Y
}.
\end{equation}
To show that the claim holds it suffices to show that for any curve 
$C'_{N+1} \subset X_{N+1}$ 
contracted by 
$r_{n_k}$,
then 
$w_{n_k, N+1} \cdot C'_{N+1} =0$, 
as that would imply that 
$w_{n_k, N+1}$
descends to 
$X_{a, n_k}$ 
and thus it must induce a morphism 
$X_{a, n_k} \to Y_{N+1}$
inducing the desired factorization, by definition of 
$w_{n_k, N+1}$.
By contradiction, let us assume that 
$w_{n_k, N+1} \cdot C'_{N+1} \neq 0$, 
then 
$w_{n_k, N+1} \cdot C'_{N+1}$
must be positive, since 
$w_{n_k, N+1}$ 
is nef over 
$Y$.
As 
$C_{N+1}'$
is contracted by 
$r_{n_k}$,
then 
$z_{n_k} \cdot C'_1 =0$ holds.
But then these two observations imply  
$(x_{n_k}+\frac 12 w_{n_k}) \cdot C_{N+1}' < 0$, 
which is impossible, since $x_{n_k}+\frac 12 w_{n_k}$ is nef over $Y$.

By definition, cf. the end of Step 2, and by the fact that $Y_{N+1}\rar Y$ is not an isomorphism, cf. Step 5.4, the factorization in \eqref{eq:factorization_contradiction} implies that for all $k \in \nn$, 
$z_{n_k} \subset \widehat I(X/Y)$, 
which contradicts our initial choice of the sequence which required
$(z_n)_{n\in \nn} \subset J(X/Y)$
and concludes the proof.
\newline

{\bf Case 5.7.b}. 
{\it Now we deal with the case where the assumption of Case 5.7.a is not satisfied}.

For all 
$n \in \nn$, 
we set
\begin{align*}
\lambda_{n} \coloneqq
\inf \left\{
\lambda \in \mathbb R 
\ \middle \vert \
w'_{n, N+1}+ (\lambda-1) T_{N+1}H_{N+1} \
\text{is nef over $Y$}
\right \},
\end{align*}
and we fix an extremal contraction 
$\psi_{n, N+1} \colon Y_{N+1} \to Y_{n, N+2}$
over $Y$ contracting an extremal ray 
$R_{n} \subset \overline{NE}(Y_{N+1}/Y)$ 
such that 
$R_{n} \cdot 
(w'_{n, N+1}+ (\lambda_{n}-1) T_{N+1}H_{N+1})
=0$.
In particular, there exists a subsequence 
$
(w'_{n_{k}, N+1})_{k \in \mathbb N}
\subset
(w'_{n, N+1})_{n \in \mathbb N}$
such that for all 
$k \in \mathbb N$,
\begin{itemize}
    \item[(a)] 
$\psi_{n_{k}, N+1} $ and $Y_{n_{k}, N+2}$ are the same, by Lemma~\ref{finiteness models}; 
and
    \item[(b)] 
$(\lambda_{n_{k}}) \subs k \in \nn.$ converges, since
for all $n \in \nn$,
$\lambda_{n} \in [0,1]$ holds.
\end{itemize}
To simplify the notation, we pass to the subsequence 
$(w'_{n_{k}})\subs k \in \nn.
\subset 
(w'_n)\subs n \in \nn.$ 
(resp. 
$(\lambda_{n_{k}})\subs k \in \nn. 
\subset
(\lambda_n)\subs n \in \nn.$) 
just defined; 
we also pass to the subsequences of
all the other sequences involved in the proof corresponding to the same indices.
We also define $Y_{N+2}$ to be the variety from property (a) above and 
$\psi_{N+1} \colon Y_{N+1} \rar Y_{N+2}$ 
be the induced morphism.
We set 
$x'_{n, N+1} 
\coloneqq
x_{n, N+1}+(1-\lambda_n)l_{N+1}^\ast T_{N+1}H_{N+1}$ 
in $\nsr(X_{N+1}/Y)$.
Thus, $(x'_{n, N+1}) \subs n \in \nn.$ converges and for all $n\in \nn$. 
Moreover, for all $n\in \nn$,
$w_{n, N+1} -(1-\lambda_n)l_{N+1}^\ast T_{N+1}H_{N+1} \equiv_{Y_{N+2}} 0$.
Hence,
$[z_{n, N+1}]=[x'_{n, N+1}] \in \nsr(X_{N+1}/Y_{N+2})$ 
which implies that 
$Y_{N+2} \rar Y$
is not an isomorphism, as otherwise
$(z_n)_{n \in \nn}$ would be convergent, which contradicts the assumption made before Step 5.2.

At this point, we iterate this procedure.
Exactly the same proof as in Step 5.5 shows that there exist a marked minimal model $(f_{N+2} \colon X_{N+2} \rar Y, \alpha_N)$ and a 
commutative diagram
$$
\xymatrix{
X_{N+1} 
\ar@{-->}[rrrr]^{\phi_{N+1} \coloneqq \alpha_{N+2}^{-1} \circ \alpha_{N+1}}
\ar[d]^{l_{N+1}}
& & & &
X_{N+2} 
\ar[d]^{l_{N+2}}
\\
Y_{N+1} 
\ar[rrrr]^{\psi_{N+1}}
\ar[drr]
& && &
Y_{N+2} 
\ar[dll]
\\
& & Y & &
}
$$
such that, defining 
$x'_{n, N+2}\coloneqq (\phi_{N+1})_\ast x'_{n_k, N+1}$,
by construction
$(x'_{n_k, N+2})$ is nef over $Y_{N+2}$ for a suitable subsequence $(x'_{n_k, N+2})_{k \in \nn} \subset (x'_{n, N+2})_{n \in \nn}$.
Proceeding as in Step 5.6, there exists $\epsilon_{N+2} > 0$ such that for any curve $C \subset X_{N+1}$ contained in the fiber of $f_{N+1}$, 
$x_{n, N+2} \cdot C > 
-\frac{2\dim X_{N+2}}{\epsilon_{N+2}}$.
We fix an integer 
$T_{N+2} \geq 2\frac{2\dim(X_{N+2})}{\epsilon_{N+2}}$
and a Cartier divisor $H_{N+2}$ on $Y_{N+2}$ ample over $Y$ and we set
$w'_{n, N+2}\coloneqq
(\psi_{N+1})_\ast w'_{n, N+1}+ (\lambda_{n}-1) T_{N+1}H_{N+1}$ for all $n\in \nn$.
If there exists a subsequence 
$(w'_{n_k, N+2})\subs k \in \nn.
\subset 
(w'_{n, N+2})\subs n \in \nn.$ 
such that for all 
$k \in \mathbb N$,
$w'_{n_k, N+2}- T_{N+2}H_{N+2}$
is nef over $Y$, we apply Case 5.7.a and reach a contradiction.
Otherwise, we repeat the procedure of Step 5.7.b.

This procedure must stop after finitely many steps.
In fact, our construction yields a commutative diagram of birational morphisms 
\begin{align*}
    \xymatrix{
    Y_{N+1} \ar[r]^{\psi_{N+1}} \ar[drrr]&
    Y_{N+2} \ar[r]^{\psi_{N+2}} \ar[drr]&
    Y_{N+3} \ar[r]^{\psi_{N+3}} \ar[dr]&
    \dots \ar[r]^{\psi_{N+l-1}} &
    Y_{N+l} \ar[r]^{\psi_{N+l}} \ar[dl]&
    Y_{N+l+1} \ar[r]^{\psi_{N+l+1}} \ar[dll]&
    \dots
    \\
    & & & Y & & &
    }
\end{align*}
such that for all $i \geq 1$ morphism $Y_{N+i} \to Y$ is not an isomorphism.
Hence, the finiteness of the diagram follows from Lemma~\ref{finiteness models}.
But this means that on the last step of this procedure, we must be in the situation of case 5.7.a and again we reach a contradiction.
\end{proof}

\section{Deforming divisors in a family} \label{section deform}

In this section, we study how divisor classes deform in families of certain types of $K$-trivial varieties.

We start with the following generalization of \cite{Tot12}*{Theorem 4.1}.
The proof has been kindly suggested by Totaro.

\begin{theorem}
\label{t-BT}
Let $X_0$ be a projective variety with rational singularities.
Assume that $H^1(X_0, \O {X_0}.) =H^2(X_0, \O {X_0}.) = 0$, and that $X_0$ is smooth in codimension 2 and $\qq$-factorial in codimension 3.
Then, given a deformation $X \rar (T,0)$ of $X_0$ over a smooth variety $T$, there is an \'{e}tale morphism $(T',0) \rar (T,0)$ such that the class group of $X_{T'}$ maps split surjectively to the class group of ${\rm Cl}(X_t)$ for all $t\in T'$, and all these surjections have the same kernel.
\end{theorem}

We summarize the property proven in Theorem~\ref{t-BT} by saying that the divisor class group is unchanged under nearby deformations of $X_0$.

\begin{proof}
Let $X \rar T$ be a deformation of $X_0$ as in the statement.
Under these assumptions, we can apply \cite{Tot12}*{Theorem 3.1} showing that $\mathrm{Cl}(X_t) \rar H \subs 2n-2.(X,\zz)$ is an isomorphism for $t \in T$, where $n = \dim(X_0)$.
By Grothendieck's six functor formalism, these homology groups form a constructible\footnote{In this context, we only mean that the sheaf is locally constant on a suitable stratification, while we do not require that the stalks are finite.} sheaf of Abelian groups on $T$.
Thus, $T$ is stratified by a union of finitely many locally closed algebraic subsets $T = \cup_{i=1}^n T_i$ such that $\mathrm{Cl}(X_t)$ is locally constant on each of the $T_i$.
Hence, if there is just one stratum near $0 \in T$, then the claim follows by the definition of locally constant sheaf.

Let us assume by contradiction that there is more than one stratum near $0 \in T$.
Then, there is a smooth curve $C$ through $0$ that crosses a stratum exactly at $0 \in T$.
Thus, the groups $\mathrm{Cl}(X_t)$ would not be locally constant around $0 \in C$, contradicting \cite{Tot12}*{Theorem~4.1}.
\end{proof}

\begin{theorem} \label{thm cones}
Let $X_0$ be a terminal $\qq$-factorial variety.
Assume that $K_{X_0}\equiv 0$, $H^1(X_0, \O {X_0}.) = 0$ and $H^2(X_0, \O_{X_0}.) = 0$.
Given a deformation $X \to (T,0)$ of $X_0$ over a smooth variety $T$, then $X$ is terminal $\qq$-factorial, $K_{X}\sim _{\qq,T} 0$ over a neighborhood of $0\in T$.
Furthermore, after an \'{e}tale base change, the following facts hold:
\begin{enumerate}
\item \label{thm.cones.incl1}
$\overline{A}(X_\eta )\supset \overline{A}(X_0)$, where the inclusion is possibly strict;
\item \label{thm.cones.incl2}
$\overline{B}(X/T)\subset \overline{B}(X_0)$, where the inclusion is possibly strict; and
\item \label{thm.cones.incl3}
$\overline{M}(X/T)\supset \overline{M}(X_0)$, where the inclusion is possibly strict.
\end{enumerate}
\end{theorem}

The inclusion in \eqref{thm.cones.incl2} in~Theorem \ref{thm cones} can be strict, see Section~\ref{section example} and, in particular, Lemma~\ref{lemma_reflection} for an example of such phenomenon.

\begin{proof}
\begin{enumerate}
\item
By repeatedly applying \cite{dFH11a}*{Corollary 3.2 and Proposition 3.5}, it follows that $X$ is $\qq$-factorial and terminal.
Since $\pm K_{X_0}$ is nef, then so is $\pm K_{X_\eta}$, where $\eta \in T$ denotes the generic point of $T$.
Thus, passing to the algebraic closure $\bar \eta$, $K_{X_{\bar \eta}}\sim _\qq 0$, by~\cite{Gon13}, and hence also $K_{X_{\eta}}\sim _{\qq} 0$.
Therefore, there is an open neighborhood $ T^0\subset T$ such that  $K_{X_{T^0}}\sim _{\qq,T^0} 0$.
By~\cite{Bir12}*{Theorem~1.5}, $0 \in T_0$.
By Theorem \ref{t-BT}, after an \'{e}tale base change $T'\rar T^0$, we may assume that the class group of ${\rm Cl}(X_t)$ is constant for all $t\in T'$.
To simplify the notation, we replace $X \rar T$ by $X_{T'}\rar T'$.
Hence, if $D\vert_{X_0}$ is ample then so is $D\vert_{X_\eta}$. 
On the other hand, Wilson's example \cite{Wil92}*{Example 4.6} shows that we could have a strict inclusion.

\item
Let $\overline{B}(X/T)$ be the relative pseudo-effective cone and $\overline{B}(X_0)$ be the pseudo-effective cone of the central fiber.
By Theorem \ref{t-BT}, There is a natural restriction map $r \colon \overline{B}(X/T)\rar \nsr(X_0)$ and, by semicontinuity \cite{Har77}*{Theorem III.12.8}, the inclusion $\overline{B}(X/T)\subset \overline{B}(X_0)$ is clear.
The inclusion could be strict by the example discussed in \S~\ref{section example}.

\item 
Note that by (1), $\overline{A}(X_0)\subset \overline{A}(X_\eta) \subset \overline{B}(X_\eta)$.
Therefore, $r(\overline{B}(X/T))\cap \overline{M}(X_0)$ contains an open subset.
Indeed, we may consider a rational polyhedral cone $\Pi \subset A(X_0)$, and, up to shrinking $T$ around $0$, every non-zero divisor class in $\Pi$ lifts to a class in $A(X/T)$.

First, we claim that $r(\overline{B}(X/T))\supset \overline{M}(X_0)$.
Assume by contradiction that this is not true.
Then, as these are full dimensional cones with a non-empty and full-dimensional intersection, we may pick a divisor $D$ in the boundary of $\overline{B}(X/T)$ such that $D_0=D\vert_{X_0}$ is in the interior of $\overline{M}(X_0)$ and in particular ${\rm vol}(D_0)>0$.
We will show that ${\rm vol}(D_t) \geq {\rm vol}(D_0)>0$ for every $t\in T$.
This contradicts the assumption that $D$ is in the boundary of $\overline{B}(X/T)$ and so $r(\overline{B}(X/T))\supset \overline{M}(X_0)$.

To see the claim we proceed as follows.
First, as our goal is to show that $\vol(D_t)\geq \vol(D_0)>0$ for every $t \in T$, we may assume that $T$ is a smooth affine curve, as any point in $t$ can be joined to $0$ with a smooth curve.
Let $D_i$ be a sequence of $\qq$-divisors contained in the interior of $\overline{B}(X/T)$ such that $\lim D_i=D$.
For any $i$, we may choose $\qq$-divisors $B_i$ and rational numbers $\beta _i>0$ such that $D_i\sim_\qq \beta _i B_i$ and $(X,X_0+B_i)$ is plt and $(X_0,B_i\vert_{X_0})$ is terminal. 
 Since $B_i$ is big over $T$, there is a minimal model $\phi \colon X\dasharrow X'$ over $T$.
 Since $D_0=D|_{X_0}$ is in the interior of $\overline{M}(X_0)$, then so is $D_i|_{X_0}$ for $i \gg 0$.
 Thus, since the stable base locus of $D_i|_{X_0}$ contains no divisors, each step of this minimal model program is an isomorphism at codimension 1 points of $X_0$. It follows that $\phi$ is an isomorphism on a neighborhood of each codimension 1 point in $X_0$. Since $(X_0,B_i|_{X_0})$ is terminal, it follows easily that $\phi _0 \colon X_0\dasharrow X'_0$ extracts no divisors and so $\phi _0$ is a small birational map. Let $B_i'=\phi _* B_i$, then $(\phi _0)_*(B_i|_{X_0})=B_i'|_{X'_0}$.
 Since $B_i'$ is nef over $T$, the volume of its restriction to any fiber is computed by self intersection.
 Thus, it follows that
 \[ \vol(X_\eta,B_i\vert \subs X_\eta.)=\vol(X'_\eta,B'_i\vert \subs X'_\eta.)= \vol(X'_0,B'_i\vert \subs X'_0.)= \vol(X_0,B_i\vert \subs X_0.),\]
where the last equality follows from the fact that $\phi_0$ is an isomorphism in codimension 1.

Thus, we have the following chain of equalities
\[
\vol(X_\eta,D_i\vert \subs X_\eta.) = \beta_i^d \vol(X_\eta,B_i\vert \subs X_\eta.)=\beta_i^d \vol(X_0,B_i\vert \subs X_0.)= \vol(X_0,D_i\vert \subs X_0.),
\]
where $d=\dim(X_0)$.
As the volume function is a continuous function on the pseudo-effective cone, it follows that
\[
\vol(X_\eta,D\vert \subs X_\eta.) = \vol(X_0,D_0) >0.
\]
By upper-semicontinuity of the volume function, it follows that
\[
\vol(X_t,D\vert \subs X_t.) \geq \vol(X_0,D_0) >0
\]
for every $t$, which is the sought contradiction.

Now assume by contradiction that (3) does not hold.
Then, we may find a divisor $D$ that is not in $\overline{M}(X/T)$ and $D_0$ is in the interior of $\overline{M}(X_0)$.
Since $D_0$ is in the interior of $\overline{M}(X_0)$, it follows that $D$ is in the interior of $\overline{B}(X/T)$. Proceeding as above, there is a $D$ minimal model $\phi \colon X\dasharrow X'$ over $T$ such that $\phi _0 \colon X_0\dasharrow X_0'$ is a small birational map. In particular $\phi _0$ contracts no divisors and hence by semicontinuity of the fiber dimension, $\phi$ also contracts no divisors over a neighborhood of $0\in T$.
Therefore, there cannot be a divisorial component of $\mathrm{Bs}(D/T)$ dominating $T$. Thus $D\in \overline{M}(X/T)$ which is the required contradiction, and (3) follows.
Wilson's example \cite{Wil92}*{Example 4.6} shows that we could have a strict inclusion.
\end{enumerate}
\end{proof}

\begin{remark} 
\label{rmk sections}
The proof of part (3) of Theorem \ref{thm cones} can be adapted to show the following: if $D_0$ is in the interior of $\overline{M}(X_0)$, then for $m>0$ sufficiently divisible the natural morphism $H^0(X,mD) \rar H^0(X_0,mD_0)$ is surjective.
Indeed, working on the model $X'$ constructed in the proof of (3), we can apply the relative Kawamata--Viehweg vanishing theorem on $X'$ to argue that the Euler characteristic $\chi(X'_t,mD'_t)$, which is constant by flatness, is given by $h^0(X'_t,mD'_t)$.
\end{remark}

\section{Interludium: an example} \label{section example}

This section aims to study the behavior of the different cones of divisors in a family of Calabi--Yau threefolds.
We will study the following example in the category of smooth Calabi--Yau threefolds which first appeared in Wilson's work, cf.~\cite{Wil92}.
The example that we explain in this section is an analog for 1-parameter families of Calabi--Yau threefolds of the Atiyah flop that naturally appears in families of K3 surfaces, see, for example,~\cite{huyb.book}*{\S~6.5 and \S~7.5}.

\subsection{Setup}
Let $f \colon X \to C$ be a family of smooth  Calabi--Yau threefolds.
In particular, for any $t \in C$, the log pair $( X, X_t)$ is plt and $X_t$ is terminal.
Let $0 \in C$ be a closed point and $X_0$ the corresponding fiber.
We assume that $X_0$ contains an elliptic ruled surface, that is, $E_0 \to G_0$ is a minimal ruled surface over the elliptic curve $G_0$.
In~\cite{Wil92}, Wilson showed that it is possible to construct families $ X \to C$ in which $E_0$ is rigid, that is, $E_0$ does not deform in the family.
Up to shrinking $C$ around $0$, we can identify the cohomology of $X$ with the cohomology of any fiber, via restriction to $X_0$. 
Furthermore, up to an \'etale base change centered at $0 \in C$, we may assume that the conclusions of Theorem \ref{t-BT} are satisfied.
Under these assumptions, We shall show that the pseudo-effective cone cannot be constant in the fibers of this family.

\begin{lemma}
The surface $E_0$ can be contracted on $X_0$ by means of a $(K_{X_0}+E_0)$-extremal contraction $\pi_0 \colon X_0 \to Y_0$ to an elliptic curve isomorphic to $G_0$.
\end{lemma}

\begin{proof}
By assumption, $E_0\vert_{E_0}=K_{E_0}$ and $K_{E_0} \cdot R=-2$, where $R$ is a fiber of the projective bundle structure on $E_0$.
As $R^2=0$ as a divisor on $E_0$, then $R$ is an extremal ray both in the nef and the pseudo-effective cones of $E_0$.
Since $E_0$ is a minimal ruled surface, then $\nsdr (E_0)=\mathbb{R}[R]\oplus \mathbb{R}[K_{E_0}]$.
Since $K_{X_0}+E_0$ is not nef, we may consider a $(K_{X_0} + E_0)$-extremal contraction $\pi_0 \colon X_0 \rar Y_0$.
\\ \\
{\bf Claim}. 
The embedding 
$i \colon E_0 \to X_0$ 
induces an embedding 
$i_\ast \colon \nsdr (E_0) \to \nsdr (X_0)$ 
and 
$\pi_0\vert_{E_0}$ 
is the contraction 
$E_0 \to G_0$.
\begin{proof}
Indeed, we may consider the plt pair $(X_0,E_0)$.
By adjunction, $(K_{X_0} + E_0) \cdot R < 0$, and all the $(K_{X_0}+E_0)$-negative curves are $E_0$-negative, hence, contained in $E_0$.
By adjunction and the fact that $R^2=0$ inside $E_0$, we have $(K_{X_0} + E_0) \cdot R = -2$.
Now, let $\tilde G_0$ be the section of $E_0 \to G_0$ of minimal self-intersection.
To show that $i_\ast \colon \nsdr (E_0) \to \nsdr (X_0)$ is an embedding, we will show that $\tilde G_0$ is not numerically equivalent to $R$ in $\nsdr (X_0)$.
By the classification of ruled surfaces over an elliptic curve, see \cite{Har77}*{\S~V.2}, we either have $\tilde G_0 ^2 \leq 0$ or $\tilde G_0 ^2 = 1$.
In the former case, since $g(G_0)=1$, then $\deg (K_{\tilde G_0})=0$ and, by adjunction 
\[(K_X+E_0) \cdot \tilde G_0 = K_{E_0}\cdot \tilde G_0=-\tilde G_0^2\geq 0.
\]
Thus, $R$ and $\tilde G _0$ are linearly independent in $\nsdr (X_0)$.
Now, assume that $\tilde G_0^2=1$.
In this case, we have that $(K_X + E_0)\cdot \tilde G_0 = -1$.
As $(K_X + E_0)\cdot R=-2$, to conclude, it suffices to rule out that $[R]=2[\tilde G_0]$ in $\nsdr (X_0)$.
Now, let $L$ be an ample divisor on $X_0$.
Then, by \cite{Har77}*{Proposition V.2.21}, up to rescaling the numerical class of $L$, we have that $L|_{E_0} \equiv \tilde  G_0 + b R$ with $b > -\frac{1}{2}$.
Thus, we have $L \cdot \tilde G_0 = (\tilde G_0 + b R)\cdot \tilde G_0=1+b$, and $L \cdot R = (\tilde G_0 + b R) \cdot R = 1$.
Thus, we have $[R] \neq 2[\tilde G_0]$, as $b \neq -\frac{1}{2}$.
Hence, we conclude that $i_\ast \colon \nsdr (E_0) \to \nsdr (X_0)$ is an embedding.
Now, since any $(\K X_0. + E_0)$-negative extremal ray is spanned by the class of a rational curve in $E_0$ and since $E_0$ is ruled over a curve of positive genus, it follows that the only possible ray is $\rr \subs \geq 0. [R]$.
Now, as we showed that $\tilde G_0$ is not in $\rr \subs \geq 0. [R]$, $\tilde G_0$ cannot be contracted by $\pi_0$ and no irreducible curve $C \subset E_0$ horizontal over $G_0$ can be contracted, as this would force $E_0$ and hence also $\tilde G_0$ to be contracted.
Thus, $E_0$ cannot be contracted to a point by a $(K_{X_0}+E_0)$-extremal contraction, as otherwise $\dim i^\ast(\nsr (X_0))=1$.
\end{proof}
\noindent
As all rational curves in $E_0$ are vertical above the elliptic base, they must all be numerically equivalent to $R$.
Thus, $R$ is a $(K_{X_0}+E_0)$-extremal curve, and its contraction $\pi_0$ is a divisorial contraction that maps $E_0$ to a curve.
The conclusion on the image of $E_0$ follows from the fact that $E_0$ is the projectivization of a vector bundle over $G_0$ and we are contracting the fibers of this bundle.
\end{proof}

\subsection{Goal} 
The primitive contraction $\pi_0 \colon X_0 \to Y_0$ is the first (and only) step in the $E_0$-MMP on $X_0$.
Let $H_0$ be a big and nef Cartier divisor in the relative interior of the facet of $\overline{A}(X_0)$ given by $\pi^\ast_0(\overline{A}(Y_0))$.
As $\pi_0$ is divisorial, then $Y_0$ is canonical,
and it contains an elliptic curve $G_0$ of canonical singularities.
As the conclusions of Theorem \ref{t-BT} are met for $X \rar C$, any divisor class on $X_0$ comes from the ambient space $ X$.
We denote by 
$H$, 
$E$ 
the corresponding cohomology classes on 
$X$ 
restricting to 
$H_0$, 
$E_0$
on 
$X_0$.
Up to replacing 
$H_0$ 
with a multiple, 
we may and shall assume that 
$H$ 
is a Cartier divisor on 
$X$.

Our goal is now to understand the models that appear on 
$X$ 
when moving along the segment 
$[H, E]$ 
in 
$\nsr (X/C)$.
In \cite{Wil92}*{Proposition 4.4}, Wilson showed that 
$H_0$ 
is big and nef but not ample, while 
$H_t \coloneqq H\vert_{X_t}$ 
is ample for any 
$t\neq 0$.
In view of this, by Kawamata--Viehweg vanishing, for all 
$t \in C$,  
for all 
$m \geq 0$ 
and all 
$i>0$,
$H^i(X_t, mH_t)=0$.
Thus, by cohomology and base change, since 
$X \to C$ 
is flat by construction, 
it follows that for all 
$i>0$ 
and all
$m \geq 0$, 
$R^i f_\ast \mathcal O_X(mH) =0$,
which in turn implies that for all
$m \geq 0$, 
$f_\ast \mathcal O_X(mH)$
is a vector bundle over $C$, 
or, equivalently, 
that the restriction map
\begin{align*}
H^0(X, mH) \to H^0(X_t, mH_t)
\end{align*}
is surjective for any $t \in C$ and for any $m \geq 0$.
Hence, the natural morphism $\pi \colon X \to Y$ over $C$ induced by $H$ lifts $\pi_0$.
For $0< \epsilon \ll 1$, the divisor $H + \epsilon E$ is relatively big over $C$ and relatively ample over $C \setminus \{0\}$.
Therefore, the relative stable base locus of $H + \epsilon E$ is a proper subset of $X_0$, and it follows that $H + \epsilon E$ is relatively movable over $C$.
Now, fix $0 \leq \Delta \sim_{\qq} H + \epsilon E$.
For $0 < \delta \ll 1$, the log pair $(X,X_0+\delta \Delta)$ is plt, and the log pair obtained by adjunction $(X_0,\delta \Delta_0)$ is terminal.
Thus, we can interpret the contraction $\pi \colon X \rar Y$ as a step of a $(\K X. + X_0 + \delta \Delta)$-MMP over $C$, that is, $\pi$ is a flipping contraction for such MMP, as its exceptional locus is small, and the flip 
\[
\xymatrix{
X \ar@{-->}[rr]^\psi  \ar[dr]^{\pi}
& &
X^+ \ar[dl]_{\pi^+}
\\ 
&
Y
&
}
\]
of $\pi$ exists.
We shall denote the strict transform of a Weil divisor $M$ on $X$ by $M^+$ on $X^+$.
\begin{lemma} \label{lemma_iso_special_fibers}
$X_0$ is isomorphic to $X_0^+$.
\end{lemma}

\begin{proof}
The map $\psi$ is a step of the $(K_X+X_0+\delta \Delta)$-MMP over $C$, $(X, X_0+\delta \Delta)$ is plt, and $(X_0,\delta \Delta_0)$ is terminal.
Then, $(X, X_0^+ + \delta \Delta^+)$ is plt as well;
similarly, $(X_0^+,\delta \Delta_0^+)$ is terminal.
Thus, $\pi_0^+ \colon X^+_0 \to Y_0$ is a terminalization of $Y_0$.
Since $K_{X^+} \equiv 0$, as $\psi$ is an isomorphism in codimension 1, then $K_{X^+_0} \equiv 0$.
So, $X_0$ and $X^+_0$ are isomorphic in codimension 1, as they are both terminal and minimal, see~\cite{KM98}*{Corollary~3.54}.
Since $X_0$ is $\qq$-factorial and $\pi_0$ is an extremal divisorial contraction, then also $Y_0$ is $\qq$-factorial.
Then, as also $\pi_0^+$ is extremal, it follows that $X_0^+$ is $\qq$-factorial as well.
In particular, as $X_0$ and $X_0^+$ are $\qq$-factorial, terminal, and minimal, they are connected by a sequence of flops 
\begin{align*}
\xymatrix{
X_0 \ar@{-->}[rrr]^{\xi} & & & X_0^+.
}
\end{align*}
Thus, also 
$\pi_0^+$ 
is a divisorial contraction; 
we shall denote by 
$\overline{E}_0$ 
its exceptional divisor which is contracted by 
$\pi_0^+$ 
to the curve 
$G_0 \subset Y_0$.

Let $R^+$ denote the class of the curves contracted by $\pi_0^+$.
As $\pi_0^+$ is an isomorphism outside $\overline{E}_0$, the first flop in the chain of flops connecting $X_0$ to $X^+_0$ must flop a rational curve in $E_0$. 
The only rational curves here are the fibers of the projective bundle structure $E_0 \rar G_0$ and they are all contained in a 1-dimensional family.
Hence, they cannot be possibly flopped.
This shows that $X_0$ and $X^+_0$ are isomorphic.
In particular, $\overline{E}_0 \cdot R^+ = -2$.
\end{proof}

Abusing notation, we denote by $\psi_0 \colon X_0 \rar X_0^+$ the isomorphism whose existence is demonstrated in the proof of Lemma~\ref{lemma_iso_special_fibers}.
Given that $\nsr(X_0)$ (resp., $\nsr(X_0^+)$) comes equipped with a natural marking given by the identification with $\nsr(X)$, the two markings cannot possibly be identified by $\psi_0^\ast$:
in fact, for example, $(H +\epsilon E)\vert_{X_0}$ is not ample on $X_0$ for $0 < \epsilon \ll 1$, while the corresponding class $(H^+ +\epsilon E^+)\vert_{X^+_0}$ is ample on $X_0^+$.
Hence, under such markings, the nef cones are not even identified.

\begin{lemma} \label{lemma_intersections}
Let $\overline{E}_0$ 
be the exceptional divisor of the morphism 
$\pi_0^+$.
For any class $D_0 \in \nsr(Y_0)$,
\[
E^+_0 \cdot ((\pi_0^{+})^\ast D_0)^2=0 = \overline{E}_0 \cdot ((\pi_0^{+})^\ast D_0)^2=0 = E_0 \cdot (\pi^\ast_0 D_0)^2.
\]
\end{lemma}

\begin{proof}
As $\overline{E}_0$ is contracted by $\pi_0^+$ to a curve, for any general ample divisors $J_0$ and $J_0'$ on $Y_0$, we have $\overline{E}_0 \cdot (\pi_0^{+ \ast} J_0) \cdot (\pi_0^{+ \ast} J'_0)=0$, as we may assume that the intersection $J_0 \cap J_0'$ avoids $G_0$.
As the N\'eron--Severi group is generated by ample divisors, the conclusion of the lemma follows for $\overline{E}_0$.
\\
The same reasoning shows also that the conclusion holds for $E_0$ on $X_0$, as $E_0$ is contracted by $\pi_0$.
\\
Let $V \subset \nsr(X/C)$ be the subspace that is generated by the classes that restrict to $\pi_0^\ast (\nsr (Y_0))$ on $X_0$.
By the conclusion of the lemma for $E_0$ and the deformation invariance of intersection products, it follows that, for any $t \in C$ and any $L \in V$, we have $E_t \cdot (L_t)^2=0$.
However, when we apply $\psi$, as nothing happens on $X\setminus \{X_0\}$, then $E^+_t \cdot (L^+_t)^2=0$ for $t \neq 0$.
This also implies, by the constancy of the intersection numbers in the family, that $E^+_0 \cdot ( H_0^+)^2=0$.
\end{proof}

\begin{remark}
\label{rmk:Y_0}
\begin{enumerate}
    \item 
\label{l.c.i.}
We observe that $Y_0$ is a local complete intersection.
Indeed, as observed by Wilson \cite{Wil92}*{p. 567}, $Y_0$ has cDV singularities.
Thus, $Y_0$ has locally analytically hypersurface singularities.
Hence, by \cite{stacks-project}*{Tag 09PY}, $Y_0$ is a local complete intersection.
In particular, cf.~\cite{Laz04a}*{Remark 3.1.34}, the Lefschetz hyperplane theorem can be applied to $Y_0$.
    \item
    \label{intersection.Y_0}
Let 
\begin{align*}
\xymatrix  @R=.4pc{
\nsr(Y_0) \otimes_\mathbb{R} \nsr(Y_0) \ar[r] & 
\nsdr(Y_0)
    \\
\alpha \otimes \beta 
\ar@{|->}[r]& 
\alpha \cdot \beta
}
\end{align*}
be the morphism induced by the intersection pairing on $Y_0$.
Let $M_0 \subset \nsdr(Y_0)$ be the image of this morphism.
We claim that $M_0=\nsdr(Y_0)$, or, equivalently, that $M_0^\perp =\{0\}$, where $M_0^\perp \subset \nsr(Y_0)$.
Let us assume that that is not case, i.e., that there is an element $0\ne v\in \nsr(Y_0)$ such that $v\cdot h\cdot h'=0$ for any divisors classes $h,h' \in \nsr(Y_0)$.
Taking $h=[H]$, where $H$ is a general very ample divisor and hence a surface with canonical singularities, then, $v\cdot h\cdot h'=v\vert_H\cdot h'\vert_H=0$.
By the previous part of this remark, the morphism $H^2(Y_0,\mathbb{Z}) \rar H^2(H,\mathbb{Z})$ is injective, and therefore $v\vert_H \neq 0$.
Now, assume that $h'$ is an ample class.
Then, since $v\vert_H \cdot h'\vert_H =0$, by the Hodge index theorem and the fact that $v\vert_H \neq 0$, it follows that $v\vert_H \cdot v\vert_H <0$.
Thus, we reach the required contradiction, as $v \cdot v \cdot h \neq 0$.
\end{enumerate}
\end{remark}

We will denote by $\overline E$
the class in 
$\nsr(X^+/C)$
whose restriction to 
$X_0^+$
is 
$\overline E_0$.

\begin{lemma} \label{lemma_reflection}
There exists a negative real number $\lambda$ such that $[E^+]=\lambda [\overline{E}]$ in $\nsr(X^+)$.
\end{lemma}

\begin{proof}
By Lemma \ref{lemma_intersections}, we know that $E_0^+$, 
$\overline{E}_0 \in (\pi_0^\ast M_0)^\perp$.
As 
$\dim \pi_0^\ast \nsr(Y_0) = \dim \nsr (X^{+}_0) -1$ 
and 
$\pi_0^\ast \nsr (Y_0)\cap (M_0)^\perp =\{0\}$, then $(\pi_0^\ast M_0)^\perp$ is a line generated by either one of the divisors $E_0^+, \overline{E}_0$.
Let $R^+$ be the class contracted by $\pi^+$ ($R^+$ is just the class of any curve in the ruling of $\overline{E}_0$).
Hence, as already observed at the end of the proof of Lemma \ref{lemma_iso_special_fibers}, $\overline{E}_0 \cdot R^+<0$.
On the other hand, by construction, $E^+_0 \cdot R^+>0$, as $E^+$ is the strict transform of $E$ through the flip.
\end{proof}

\begin{remark}
\label{rem.W.example}
Lemma~\ref{lemma_reflection} provides an instance in which the inclusion (2) in~Theorem \ref{thm cones} is strict.
Indeed, using the notation of \S~\ref{section example}, if $E$ were pseudo-effective, then so would $E^+$ be.
Yet, this would imply that $E_0^+$ is pseudo-effective.
Then, by Lemma \ref{lemma_reflection}, $\overline{B}(X/C)$ would contain a line, which is impossible.

The conclusions of Lemma \ref{lemma_reflection} are consistent with the observations made by Wilson in \cite{Wil92}*{\S~5}.
Indeed, Wilson showed that the transformation induced on $\nsr(X_0)$ by the flip is a reflection through the plane generated by pull-back of classes from $Y_0$ which sends $E_0$ to its negative.
In particular, the coefficient $\lambda$ in Lemma \ref{lemma_reflection} is $\lambda = -1$.
\end{remark}

\section{Elliptic threefolds}
\label{sect.ell.3folds}
We recall that, by the notation set in \S~\ref{def_CY}, a Calabi--Yau threefold $X$ is a normal projective threefold with $\mathbb Q$-factorial terminal singularities such that $K_X \sim 0$ and $h^1(X,\mathcal{O}_X)=h^2(X,\mathcal{O}_X)=0$.

\subsection{Elliptic Calabi--Yau threefolds and their bases} 
If $f \colon X \rar S$ is an elliptic Calabi--Yau threefold, the base $S$ is either rational, or it is a surface with Du Val singularities whose minimal resolution is an Enriques surface, see~\cite{Gr91}.

\begin{remark} \label{rmk_surfaces}
Let $X$ be a klt variety with $\K X. \sim \subs \qq. 0$.
Assume that $X$ is endowed with an elliptic fibration $f \colon X \rar S$.
Then, by the canonical bundle formula, cf. Proposition~\ref{prop models}, there exists a generalized pair structure $(S,B_S + M_S)$ on $S$ such that $K_X \sim f^\ast(K_S+B_S+M_S)$.
By~\cite{FM00}, the coefficients of $B_S$ lie in the set $\left\{ 1-\frac 1n \ \vert \ n \in \mathbb N_{>0} \right\} \cup \left\{ \frac 16, \frac 14, \frac 13 \right\}$, and , by Kodaira's canonical bundle formula, we can find an effective and integral divisor $D$ such that $\frac{1}{12}D \sim_\qq M_S$, and $(S,\Delta_S)$ is klt, where $\Delta_S \coloneqq B_S + \frac{1}{12}D$, cf.~\cite{PS09}*{Example~7.16}.
In particular, $\mathrm{coeff}(\Delta_S) \subset \left\{ 1-\frac 1n \ \vert \ n \in \mathbb N_{>0} \right\} \cup \left\{ \frac 16, \frac 14, \frac 13, \frac{1}{12} \right\}$.
\end{remark}

We will denote the set $\left\{ 1-\frac 1n \ \vert \ n \in \mathbb N_{>0} \right\} \cup \left\{ \frac 16, \frac 14, \frac 13, \frac{1}{12} \right\}$ by $\celliptic$.

In order to prove the boundedness of the elliptic Calabi--Yau threefolds, one first needs to address the boundedness of the bases of the corresponding elliptic fibrations.

The set of possible rational bases is bounded by work of Alexeev~\cite{Ale94}.

\begin{proposition} \label{prop_surf_bdd}
The set of log pairs
\begin{align*}
\setbases \coloneqq
\left \{
(S, \Delta_S) \ \middle \vert \ 
\begin{array}{l}
\dim S=2, \
\mathrm{coeff}(\Delta_S) \subset \celliptic, \
K_S+\Delta_S \sim_\mathbb{Q} 0, \\
\text{$S$ is rationally connected, and $(S, \Delta_S)$ is projective klt}
\end{array}
\right\}
\end{align*}
is log bounded.
\end{proposition}

\begin{proof}
Fix $(S,\Delta_S) \in \setbases$.
By~\cite{PS09}*{Corollary 1.11}, there exists $N \in \nn$, only depending on the data of our problem such that $N(\K S. + \Delta_S) \sim 0$.
Since $(S,\Delta_S)$ is klt, this implies that it is $\frac{1}{N}$-log canonical.
Then, by~\cite{Ale94}*{Theorem 6.9}, the surface $S$ belongs to a bounded family.
The statement about the boundary follows from~\cite{dCSH}*{Theorem 4.1}.
\end{proof}

On the other hand, we could not find in the literature a result showing the boundedness of singular models of Enriques surfaces.
The following statement fills this gap.
We deduce the boundedness of these varieties from the boundedness of Enriques surfaces and the Kawamata--Morrison cone conjecture.

\begin{theorem} \label{bdd enriques}
The set of varieties
\begin{align*}
\setenr \coloneqq
\left \{
S \ \middle \vert \ 
\begin{array}{l}
S \ 
\text{is a projective surface with at worst Du Val singularities}\\
\text{and its minimal resolution is an Enriques surface}
\end{array}
\right\}
\end{align*}
is bounded.
\end{theorem}

\begin{proof}
It is well known that Enriques surfaces form a bounded family.
For instance, by~\cite{Cos85}*{Theorem 1}, every Enriques surface admits a birational morphism onto a (possibly singular) projective surface of degree 10 in $\pr 5.$.
Projective surfaces of degree 10 in $\pr 5.$ form a bounded family -- it suffices to consider their Hilbert scheme in $\pr 5.$.
Then, also the set given by their resolutions forms a bounded family, thus proving the boundedness of Enriques surfaces.

Now, we need to show the boundedness of the set of surfaces admitting Du Val singularities whose minimal resolution is Enriques.
Let $\mathcal X \rar T$ be a family that bounds the set of Enriques surfaces.
Up to replacing this family, we may assume that the conclusions of Theorem~\ref{t-BT} and Theorem
\ref{thm cones} hold.
As there are finitely many of these components, in the following we focus on a single one, with the understanding that the same argument has to be repeated on each one of them individually.

Let $\eta \in T$ be the generic point.
By~\cite{Kaw97}*{Remark 2.2}, there is a rational polyhedral cone
$\Pi_\eta \subset \overline{A}(\mathcal X_\eta)$ that serves as fundamental domain of the action of 
$\mathrm{Aut}(\mathcal X_\eta)$ on 
$A^e(\mathcal X_\eta)=
\overline{A}(\mathcal X_\eta)$, where we use the fact that every nef Cartier divisor on an Enriques surface is semi-ample.
Then, the semi-group of lattice points of $\Pi_\eta$ is finitely generated.
Let $M^1_\eta,\ldots , M^k_\eta$ be a set of generators of this semi-group.
We denote by $M^1,\ldots , M^k$ the corresponding classes in $\nsr(\mathcal X/T)$ given by the identification of $\nsr(\mathcal X_\eta)$ and $\nsr(\mathcal X/T)$.
Since any integral nef divisor on an Enriques surface is semi-ample, each $M^i_\eta$ spreads out to a divisor that is relatively semi-ample over a non-empty open subset of $T$.
Thus, up to shrinking $T$ finitely many times (this is allowed by Noetherian induction), we may assume that each $M^i$ is relatively semi-ample.

Let $L_\eta$ be a divisor in $\Pi_\eta$, and let $L \in \nsr(\mathcal X/T)$ be the corresponding divisor class.
Then, we claim that $L$ is semi-ample over $T$.
Indeed, as $\Pi_\eta$ is a rational polyhedral cone and $M^1_\eta,\ldots , M^k_\eta$ generate its lattice points over $\mathbb{Z}_{\geq 0}$, then $M^1_\eta,\ldots , M^k_\eta$ generate $\Pi_\eta$ over $\rr_{\geq 0}$.
Thus, $L = \sum_{i=1}^k a_i M^i$, where each $a_i \geq 0$.
Since each $M^i$ is semi-ample over $T$, then so is $L$.

We will now show that, under the assumptions of the previous reductions, $\overline{A}(\mathcal X_\eta)=\overline{M}(\mathcal X/T)$, where we identify the vector spaces $\nsr(\mathcal X_\eta)$ and $\nsr(\mathcal X/T)$.
Clearly, we have $\overline{M}(\mathcal X/T) \subset \overline{A}(\mathcal X_\eta)$, as any relatively movable divisor restricts to a movable divisor on the generic fiber, and movable divisors are nef on surfaces.
Now, let $D_\eta \in \overline{A}(\mathcal X_\eta)$, and let $D$ be the corresponding divisor class in $\nsr(\mathcal X/T)$.
By assumption, there is an automorphism $\phi \in \mathrm{Aut}(\mathcal X_\eta)$ and a divisor $L_\eta \in \mathrm{\Pi_\eta}$ such that $\phi_* D_\eta = L_\eta$.
Then, by regarding $\phi$ as an element of $\mathrm{Bir}(\mathcal X/T)$, we have $\phi_*D=L$.
Since $\mathcal X$ is smooth and relatively minimal over $T$, $\phi$ does not contract nor extract any divisor.
Thus, $\phi_*$ preserves linear equivalence.
Thus, as $L$ is semi-ample and the indeterminacy loci of $\phi$ and $\phi^{-1}$ are small, it follows that the relative stable base locus of $D$ over $T$ does not contain any divisor.
In particular, $D \in \overline{M}(\mathcal X/T)$.

In particular, we have that $\Pi_\eta$ is a fundamental domain for the action of $\mathrm{Aut}(\mathcal X_\eta)=\mathrm{Bir}(\mathcal X/T)$ on $\overline{A}(\mathcal X_\eta)=\overline{M}(\mathcal X/T)$.
Since $\Pi_\eta$ is rational polyhedral, it follows that the cone conjecture holds in this particular setup.
Thus, there are only finitely many chambers for the marked minimal models of $\mathcal X \rar T$ and finitely many faces of them up to the action of $\mathrm{Bir}(\mathcal X/T)$.
In particular, there are finitely many varieties $Y_1 ,\ldots , Y_l$ over $T$ such that, for every divisor $D \in \overline{M}(\mathcal X/T)$, the ample model of $D$ over $T$ is isomorphic over $T$ to some $Y_i$, for $i=1, \ldots , l$.

Now, let $S_0$ be a surface with Du Val singularities whose minimal resolution is $f \colon \mathcal X_0 \rar S_0$ for some $0\in T$.
Let $H_0$ be the pull-back via $f$ of an ample divisor on $S_0$.
Also, let $H$ be the corresponding divisor class in $\nsr(\mathcal X/T)$.
By Proposition~\ref{thm cones}, $H\in \overline{A}(\mathcal X_\eta)$.
Let $Y_i$ be the distinguished model that is isomorphic over $T$ to the relative ample model of $H$.
Then, by construction, we have that the fiber of $Y_i \rar T$ over $0 \in T$ is isomorphic to $S_0$.
In particular, the Du Val models of the Enriques surfaces appearing as fibers of $\mathcal X \rar T$ are bounded, as they all appear as fibers of some $Y_i \rar T$.

By iteration of this argument on all the finitely many components of $T$ and by Noetherian induction on the closed subsets of $T$ removed in the construction, this shows the claim.
\end{proof}

\subsection{Elliptic Calabi--Yau threefolds with a rational base}

In~\cite{Gro94}*{Theorem 1}, Gross proved the following result showing that minimal terminal elliptic Calabi--Yau threefolds with rational base are birationally bounded.
Recall that, in this work, a Calabi--Yau threefold $X$ is a $\mathbb Q$-factorial, terminal threefold with $K_X \sim 0$, and $h^1(X,\mathcal{O}_X)=h^2(X,\mathcal{O}_X)=0$. 

\begin{theorem}[{\cite{Gro94}*{Theorem 1}}] 
\label{thm_gross}
The set of triples
\[
\familyellcyrat
\coloneqq\left\{
(X, S, h) \
\middle \vert \
h \colon X\rar S \  
\text{is an elliptic Calabi--Yau threefold and} \
S \ \text{is normal and rational}
\right\}
\]
is birationally bounded.
\end{theorem}

By the above theorem, together with Definition~\ref{def bdd fibrations}, passing to a resolution of a bounding family of fibrations, we can assume that there exist quasi-projective varieties 
$\mathcal X, \mathcal S, T$ 
and a commutative diagram
\begin{align}
\label{diag.bound.ell.rat}
\xymatrix{
\mathcal X \ar[rrrr]^f \ar[drr]_\pi& & & & \mathcal S \ar[dll]^g \\
& & T & &
}
\end{align}
of projective morphisms satisfying the following properties:
\begin{enumerate}
    \item 
$\pi$ and $g$ are smooth;
    \item 
for every $t\in T$, $\mathcal X_t$ is birational to a Calabi--Yau threefold, $\mathcal S_t$ is a smooth rational surface and the general fiber of $f_t$ is an elliptic curve; and 
    \item
if $h \colon X\to S$ is an elliptic Calabi--Yau threefold over a rational surface $S$, then $X\to S$ is birationally equivalent to $\mathcal X_t\to \mathcal S_t$ for some $t\in T$, that is, there exists a commutative diagram
\begin{align}
\label{comm.diag.bir.fibr}
\xymatrix{
X \ar@{-->}[rrr]^\phi \ar[d]_h & & & \mathcal X_t \ar[d]^{f_t}\\
S \ar@{-->}[rrr]^\psi  & & & \mathcal S_t
}
\end{align}
where the horizontal arrows are birational maps.
\end{enumerate}
Using techniques of the MMP, up to stratifying the base $T$, we can modify birationally the family in~\eqref{diag.bound.ell.rat} to obtain a new family of elliptic fibrations
\begin{align}
\label{diag.bound.ell.rat2}
\xymatrix{
\mathcal X' \ar[rrrr]^{f'} \ar[drr]_{\pi'} & & & & \mathcal S' \ar[dll]^{g'} \\
& & T & &
}
\end{align}
such that, for every $t \in T$, $\mathcal{X}_t$ and  $\mathcal{X}'_t$ (resp. $\mathcal{S}_t$ and $\mathcal{S}'_t$) are birationally equivalent, and $\mathcal{X}'_t$ is a minimal model for $\mathcal{X}_t$ (in particular, $\mathcal{X}'_t$ is terminal and $\mathbb Q$-factorial, but not necessarily smooth).
In particular, this implies that $\familyellcyrat$ is log bounded in codimension 1 since the rational map $\phi \colon X \drar \mathcal{X}'_t$ is a sequence of $K_X$-flops, cf.~Definition~\ref{def bdd fibrations}.3.
On the other hand, the rational contraction $f_t \circ \phi \colon X \drar \mathcal{S}'_t$ is not necessarily a morphism, as the birational map $S \drar \mathcal{S}'_t$ may extract some divisor.
Hence, the sequence of flops connecting $X$ and $\mathcal{X}'_t$ is not necessarily a sequence of flops relative to a 2-dimensional base.

To remedy this issue, we can prove the following more precise version of the boundedness in codimension 1 of elliptic Calabi--Yau threefolds.

\begin{proposition} 
\label{prop_bdd_flops}
There exist quasi-projective varieties 
$\mathcal X, \mathcal S, T$ 
and a commutative diagram
\[
\xymatrix{
\mathcal X \ar[rrrr]^f \ar[drr]_\pi& & & & \mathcal S \ar[dll]^g \\
& & T & &
}
\] 
of projective morphisms satisfying the following properties:
\begin{enumerate}
    \item
$\pi$ is a flat family of threefolds and $g$ is a flat family of surfaces;
    \item
for every $t \in T$, $\mathcal X_t$ is a Calabi--Yau threefold.
In particular, $\mathcal{X}_t$ has terminal $\mathbb Q$-factorial singularities; and 
    \item
for every terminal elliptic Calabi--Yau threefold  with rational base $h \colon X \rar S$, there exists $t \in T$
together with an isomorphism in codimension 1 $\phi \colon X \dashrightarrow \mathcal X_t$
such that $\mathcal S_t$ and $S$ are isomorphic and $\phi$ is a birational morphism over $S$.
\end{enumerate}
\end{proposition}

The main feature of the statement of Proposition~\ref{prop_bdd_flops} is given by property (3), i.e., by the fact that the family will contain every base of an elliptic fibration.
This will be a useful feature when trying to prove the boundedness of elliptic fibrations.

To do so, we will apply Theorem~\ref{thm KM conj general} which allows controlling the number of birational models of an elliptic fibration with a fixed base. 
As the statement of Theorem~\ref{thm KM conj general} works for any Calabi--Yau fiber space of relative dimension 1, it can also be applied to control the birational models of a given family of elliptic fibrations.
We utilize  Theorem~\ref{thm KM conj general} to turn the statement of Proposition~\ref{prop_bdd_flops} on boundedness in codimension 1 of elliptic Calabi--Yau threefolds into a full boundedness statement.
In order to use Theorem~\ref{thm KM conj general} effectively, 
we need to guarantee that in the above sketch, the birational map $S \drar \mathcal{S}'_t$ is an isomorphism, that is, that all bases of elliptic Calabi--Yau varieties appear in a bounding family such as the one in~\eqref{diag.bound.ell.rat2}.
One possible way to achieve this would be to adapt the proof of Theorem~\ref{thm_gross} to start from the families of surfaces guaranteed by Proposition~\ref{prop_surf_bdd}, rather than considering suitable smooth models of such surfaces as in~\cite{Gro94}.
A more direct approach, which still relies on the ideas of~\cite{Gro94}, is given by the results of~\cite{Fil20}.

\begin{proof}[Proof of Proposition~\ref{prop_bdd_flops}]
Let $h \colon X \rar S$ be an elliptic Calabi--Yau threefold with rational base.
By Proposition~\ref{prop_surf_bdd}, $S$ belongs to a bounded family.
Therefore, there exist $v \in \nn$ (independent of $S$) and a very ample divisor $H_S$ on $S$ such that $H^2 \leq v$ and $(S,\frac{1}{2}H_S)$ is klt.
By Theorem~\ref{thm_gross}, $(X, S, h)\in \familyellcyrat$ and the latter is birationally bounded.
In particular, $h$ admits a rational $d$-section, for some $d \in \mathbb Z_{>0}$ bounded from above.
Then, the claim follows by applying~\cite{Fil20}*{Theorem 1.1} to the log pair $(X,\frac{1}{2}h^\ast H_S)$.
Finally, the fact that we may assume that every fiber $\mathcal X_t$ is a Calabi--Yau threefold follows readily from Theorem~\ref{thm cones}.
\end{proof}

\subsection{Elliptic Calabi--Yau threefolds with non-rational base}
Let $f \colon X \rar S$ be an elliptic Calabi--Yau threefold such that
$S$ is a surface with at worst Du Val singularities whose minimal resolution is an Enriques surface.
Then, the fibration $f$ is isotrivial.
Furthermore, by~\cite{KL09}*{Theorem 14}, after a quasi-\'etale cover, $X$ splits as the product $E \times Y$, where $E$ is an elliptic curve and $Y$ is either a K3 surface or an Abelian surface.
Thus, the structure of such elliptic Calabi--Yau varieties is rather clear.
On the other hand, the boundedness of these varieties has not been addressed before.
For this purpose, we perform an analysis of this case that is similar to the one carried out by Gross in~\cite{Gro94} for rational bases.

First, we start by analyzing Jacobian fibrations.
Given an elliptic fibration $f \colon X \rar Y$, the generic fiber $X_{\eta}$ is a smooth curve of genus 1;
we denote by $J(X)_{\eta}$ its Jacobian, which is then a smooth curve defined over $\mathbb C(Y)$ of genus 1 with a $\mathbb C(Y)$-rational point.
Then, the Jacobian fibration of $f$ is defined birationally as any elliptic fibration $j \colon J(X) \rar Y$ such that the generic fiber of $j$ is $J(X)_{\eta}$, cf. \cite{Gro94}*{Definition 1.4}.
In general, we may assume that $j$ is relatively minimal over $Y$:
indeed, by passing to a log resolution, we may first assume $J(X)$ is smooth;
then, we may run a relative minimal model program over $Y$, which terminates with a good minimal model by \cite{HX13}*{Theorem 1.1}.
Thus, in general, we may choose a representative $j \colon J(X) \rar Y$ for the Jacobian fibration of $f$ such that $J(X)$ is terminal, $\qq$-factorial, and $\K J(X).$ is semi-ample over $Y$.
In particular, for $m \gg 1$, the relative linear system $|m \K J(X)./Y|$ induces a factorization $J(X) \rar Y' \rar Y$, where $Y' \rar Y$ is birational.

If the base $Y$ is a curve, then the Jacobian fibration is an elliptic fibration with a section, and there is extensive literature about the Weierstrass models of these fibrations.
Furthermore, Weierstrass models for elliptic fibrations with a rational section still exist if the base is a smooth surface, see \cite{DG94}*{Proposition 2.4}.
In particular, if $j \colon J(X) \rar S$ is a Jacobian fibration over a surface and $S'$ is any smooth birational model of $S$ mapping to $S$, we may construct a Weierstrass model of $j$ with base $S'$.
Then, by further blowing up the discriminant locus of the fibration in $S'$, we may further improve the geometry of the fibration to guarantee that the total space of the Weierstrass model is smooth and the corresponding morphism is flat, see \cite{Gro94}*{p. 276}.
These special models are called Miranda models, see Definition \ref{def miranda}.

\begin{definition} \label{def miranda}
Let $f \colon X \rar S$ be an elliptic fibration.
A Miranda model of $f$ is an elliptic fibration $f' \colon X' \rar S'$ such that
\begin{itemize}
    \item 
$f'$ is birationally equivalent to $f$ in the sense of~\eqref{comm.diag.bir.fibr};
    \item 
$X'$ and $S'$ are regular;
    \item     
$f'$ is flat and it admits a section;
    \item 
the discriminant locus $\Sigma=\{s\in S'|X'_s \; \text{is not regular} \}$ 
is simple normal crossing; and
    \item 
all fibers over the singular points of $\Sigma$ have Kodaira type $I_{M_1}+I_{M_2}$, $I_{M_1}+I_{M_2}^*$, $II+IV$, $II+I_0^*$, $II+IV^*$, $IV+I_0^*$, or $III+I_0^*$.
\end{itemize}
\end{definition}

In \cite{Gro94}, the singular points of the simple normal crossing divisor $\Sigma$ are called collision points.
This terminology reflects the fact that, as $\Sigma$ has simple normal crossings, its singularities come from different components $\Sigma_1$ and $\Sigma_2$ meeting transversally.
Then, by definition, the Kodaira type of the fiber over a point $p \in \Sigma_1 \cap \Sigma_2$ is determined by the type of fiber over the general point of the two components.
For instance, if the general fiber over $\Sigma_1$ has Kodaira type $I_{M_1}$ and the general fiber over $\Sigma_2$ has Kodaira type $I_{M_2}$, we say that the fiber over $p$ has Kodaira type $I_{M_1}+I_{M_2}$.
Then, all the types of fibers appearing in the last item of Definition \ref{def miranda} are explained analogously.

In this subsection we will analyze the Jacobian fibration of an elliptic Calabi--Yau threefold with Enriques base to prove the birational boundedness of the latter, cf. Theorem~\ref{enriques 3fold bir bdd}.
We start our analysis by showing that the Jacobian of such an elliptic Calabi--Yau variety is Calabi--Yau as well, cf.~\cite{GW}.

\begin{proposition} \label{Jac is CY}
Let $f \colon X \rar S$ be an elliptic Calabi--Yau threefold.
Assume that the minimal resolution of $S$ is an Enriques surface.
Let $j \colon J(X) \rar S$ be a relatively minimal model over $S$ of the Jacobian fibration of $f$.
Then, $J(X)$ is a Calabi--Yau threefold.
\end{proposition}

\begin{proof}
By assumption, $J(X)$ is terminal and $\qq$-factorial, as it is a relatively minimal model of a smooth variety.
By~\cite{GW}*{Corollary 29}, then $h^0(J(X),\K J(X).)=1$, $\kappa(J(X))=0$, and $h^1(J(X),\O J(X).)=h^2(J(X),\O J(X).)=0$.
Thus, to conclude, it suffices to show that $\K J(X). \sim_\qq 0$, i.e., that $J(X)$ is minimal.

Let 
\[
\xymatrix{& J(X) \ar[dl]_{j'} \ar[dr]^{j}&
\\
S' \ar[rr]^{\tau} & &S}
\]
be the relatively ample model of $J(X)$ over $S$.
By the canonical bundle formula, $\K J(X). \sim_\qq (j')^*(\K S'. + \Delta')$, where $(S',\Delta')$ is a klt pair.
Then, $\K S'. + \Delta'$ is $\tau$-ample.
By assumption, the canonical bundle formula applied to the morphism $f \colon X \rar S$ induces trivial boundary part and trivial moduli part.
Thus, by~\cite{Gro94}*{Lemma 1.6} applied over a big open set of $S$, it follows that $\Delta'$ is $\pi$-exceptional.

Since $S$ has Du Val singularities, then $\K S'. \sim_{\qq,S} E \geq 0$, where $E$ is $\pi$-exceptional.
Thus, $\K S'. + \Delta' \sim \subs \qq,S. E + \Delta'$, where $E + \Delta'$ is effective, $\pi$-exceptional, and $\pi$-ample.
The negativity lemma then implies that $\tau$ is the identity morphism.
\end{proof}

In order to retrieve birational boundedness of the original models $f \colon X \rar S$ from the boundedness in codimension 1 of the Jacobian fibrations $j \colon J(X) \rar S$, we need to understand how many smooth curves of genus 1 over $k(S)$ admit the same Jacobian $J(X)_\eta$.
This association is controlled by the Weil--Ch\^atelet group $WC(J(X)_\eta)$, see~\cite{DG94}.
In particular, $WC(J(X)_\eta)$ parametrizes birational equivalence classes of elliptic firbations over $S$ whose generic fiber has prescribed Jacobian.
On the other hand, an elliptic fibration that arises from a Calabi--Yau variety has very restrictive geometric conditions, which then restrict the class of generic fibers that can possibly arise.
These fibrations can be parametrized by a much smaller subgroup of the Weil--Ch\^atelet group, known as the Tate--Shafarevich group $\Sh_{S}(J(X)_\eta)$.
We refer to~\cites{DG94,Gro94} for a systematic treatment of this topic and for the formal definitions of these groups using \'etale cohomology.
Here, we limit ourselves to the following geometric characterization of $\Sh_{S}(J(X)_\eta)$ as a set:
\[
\Sh_{S}(J(X)_\eta) = \lbrace C \in WC(J(X)_\eta) | X_C \rar S \; \text{has a rational section \'etale locally at $s$ for every point}\;s \in S \rbrace,
\]
where $X_C \rar S$ is some proper model of the curve $C$ defined over $k(S)$, see \cite{Gro94}*{\S~3}.
Thus, $\Sh_{S}(J(X)_\eta)$ imposes pretty restrictive conditions on the type of singular fibers that can occur.
In particular, as there is no local obstruction to admitting a rational section, multiple fibers cannot occur over codimension 1 points of the base.
Lastly, we observe that, as in the case of $WC(J(X)_\eta)$, $\Sh_{S}(J(X)_\eta)$ parametrizes birational equivalence classes of elliptic firbations over $S$.
\begin{remark}
\label{rmk.TS.ell.CY}
When considering an elliptic Calabi--Yau threefold $f \colon X \rar S$, the morphism $f$ does not admit multiple fibers over codimension 1 points of the base $S$, as $K_X\sim 0$.
Thus, $f \colon X \rar S$ corresponds to a class in the Tate--Shafarevich group $\Sh_{U}(J(X)_\eta)$, for a big open subset $U \subset S$.
For more details, see \cite{Gro94}*{p. 276}.
\end{remark}
A first step towards proving birational boundedness of elliptic Calabi--Yau threefolds with a non-rational base is to show that their Tate-Shafarevich groups are finite.

\begin{proposition} \label{punctured Shafarevich}
Let $f \colon X \rar S$ be an elliptic Calabi--Yau threefold.
Assume that the minimal resolution of $S$ is an Enriques surface.
Then, $X \in \Sh_{S \setminus S_{\rm sing}}(J(X)_\eta)$.
\end{proposition}

\begin{proof}
First, we show that we may assume that $f$ is equidimensional.
Let $f \colon X \rar S$ be an elliptic Calabi--Yau threefold and assume that the minimal resolution $S'$ of $S$ is an Enriques surface.
Then, as $K_X \sim 0$ and $K_S \sim_\qq 0$, in the canonical bundle formula $K_X\sim f^\ast (\K S.+ B_S+M_S)$ the boundary part $B_S=0$ while the moduli part $M_S \equiv 0$.
Up to flopping $X$ over $S$, by Lemma~\ref{lemma_vertical_divs}, we may assume that $f$ factors as $X \rar S' \rar S$, where $f' \colon X \rar S'$ is equidimensional, and $\pi \colon S' \rar S$ is birational.
Since $K_X \sim 0$ and $K_S \sim_\qq 0$, it follows that $S' \rar S$ is a partial resolution.
Let $\{ p_1, \ldots ,p_k \}$ be the singular locus of $S$, and let $\pi$ be an isomorphism over $p_i$ for $1 \leq i \leq l$ for some $1 \leq l \leq k$.
Then, $S'_{\rm sing}= \{p_1, \ldots , p_l,q_1,\ldots, q_m \}$, where $q_j \in \mathrm{Ex}(\pi)$ for all $j$.
For each $i=l+1,\ldots k$, let $E_i$ be the (possibly reducible) $\pi$-exceptional curve mapping to $p_i$.
Then, $S \setminus \{p_1, \ldots , p_k \} = (S' \setminus \{ p_1,\ldots ,p_l\})\setminus (\cup_{i=l+1}^kE_i)$.
Then, assuming the claim in the equidimensional case, $X \in \Sh_{S' \setminus S'_{\rm sing}}(J(X)_\eta)$.
As $S'_{\rm sing}= \{p_1, \ldots , p_l,q_1,\ldots, q_m \}$, it then follows that
\[
\Sh_{S' \setminus S'_{\rm sing}}(J(X)_\eta) \subset \Sh_{(S' \setminus \{ p_1,\ldots ,p_l\})\setminus (\cup_{i=l+1}^kE_i)}(J(X)_\eta) = \Sh_{S \setminus S_{\rm sing}}(J(X)_\eta).
\]
Thus, in the following, we may assume that $f$ is equidimensional.
By~\cite{KL09}*{Theorem 14}, there exists a commutative diagram
\begin{center}
\begin{tikzcd}
X \arrow[d, "f"'] & \tilde{X}=\tilde{F}\times \tilde{S} \arrow[d, "g"] \arrow[l, "\psi"'] \\
S                 & \tilde{S} \arrow[l, "\phi"']
\end{tikzcd}
\end{center}
where $\psi$ is \'etale in codimension 1, 
$K_{\bar{X}} \equiv 0$, 
and $\phi$ is a generically finite rational map.  
As $X$ is terminal, then so is $\tilde X$.
Note then that $\tilde F$ is an elliptic curve and $\tilde S$ is a smooth surface with $K_{\tilde S}\sim 0$.
From the above diagram it follows that $\phi$ is a morphism, \'etale over the regular locus of $S$, such that $\deg(\phi)=\deg(\psi)$.
Since $f$ and $g$ are equidimensional, it follows that $\phi$ is finite.
Since $\phi$ is \'etale outside of $S_{\rm sing}$ and $\deg(\phi)=\deg(\psi)$, it follows that $\psi$ is \'etale over the complement of $S_{\rm sing}$.
In particular, $f$ is a smooth fibration over the complement of $S_{\rm sing}$.
Thus, $X \in \Sh_{S \setminus S_{\rm sing}}(J(X)_\eta)$.
\end{proof}

Miranda models are particularly important for the direct computation of the Tate--Shafarevich group of an elliptic fibration of a Calabi--Yau threefold.

\begin{proposition} \label{finite Shafarevich}
Let $f \colon X \rar S$ be an elliptic Calabi--Yau threefold.
Assume that the minimal resolution of $S$ is an Enriques surface.
Let $\tilde{j} \colon \tilde{J}(X) \rar \tilde{S}$ be a Miranda model of the Jacobian fibration $j \colon J(X) \rar S$ of $f$. 
Let $\tilde{E}$ be the exceptional locus of the birational morphism $\tilde{S} \rar S$ with the reduced structure.
Then, $X \in \Sh_{\tilde{S} \setminus \tilde{E}}(J(X)_\eta)$, and this group is finite.
\end{proposition}

\begin{proof}
By~\cite{DG94}*{Theorem 2.24}, $\Sh_{\tilde S}(J(X)_\eta)$ is an extension of $(\qq/\zz)^r$ by a finite group, where
\[
r=b_2(\tilde{J}(X))-\rho(\tilde{J}(X))-(b_2(\tilde S)- \rho(\tilde S)).
\]
By Proposition~\ref{Jac is CY}, $\tilde{J}(X)$ is a resolution of a Calabi--Yau variety, and the minimal model of $\tilde{S}$ is an Enriques surface.
Thus, it follows that $h^2(\tilde{J}(X),\O \tilde J (X).)= h^2(\tilde S, \O \tilde S.)=0$.
Therefore, $b_2(\tilde{J}(X))=\rho(\tilde{J}(X))$ and $b_2(\tilde S) = \rho(\tilde S)$.
In particular, $r=0$ and $\Sh_{\tilde S}(J(X)_\eta)$ is a finite group.

Let $S'$ denote the minimal resolution of $S$.
Then, $\tilde{S} \rar S$ factors through $S'$, as $\tilde S$ is smooth.
By~\cite{DG94}*{Proposition 2.4}, $J(X)$ admits a model over $S'$ that is a Weierstrass fibration.
Then, by~\cite{DG94}*{proof of Theorem 2.8}, we may assume that $\tilde{S}$ is obtained by blowing up the discriminant locus of the Weierstrass model.
By the proof of Proposition~\ref{punctured Shafarevich}, $J(X) \rar S$ is smooth over the complement of $S_{\rm sing}$.
In particular, $\tilde{E}$ is the inverse image of $S_{\rm sing}$.
Then, by Proposition~\ref{punctured Shafarevich}, $X \in \Sh_{S \setminus S_{\rm sing}}(J(X)_\eta)=\Sh_{\tilde{S}\setminus \tilde{E}}(J(X)_\eta)$.

To conclude, we need to show that $\Sh_{\tilde{S}\setminus \tilde{E}}(J(X)_\eta)$ is finite.
This follows from the finiteness of $\Sh_{\tilde S}(J(X)_\eta)$ and~\cite{Gro94}*{Proposition 3.2}.
\end{proof}
Since only
finitely many birational classes of Calabi--Yau threefolds can admit the same Jacobian fibration,
to prove the boundedness of elliptic Calabi--Yau threefold with base a singular Enriques surface, we need to show that the assignment ``fibration to Jacobian'' can be inverted in a family in a finite-to-one way, rather than just on a fixed model.
For this purpose, one needs to arrange for a family of Jacobian fibrations with some special geometric properties, to guarantee that the Tate--Shafarevich group behaves well in the family.

\begin{proposition} \label{family Miranda models}
There exist quasi-projective varieties 
$\tilde{\mathcal J}, \tilde{\mathcal S}, T$ 
and a commutative diagram
\[
\xymatrix{
\tilde{\mathcal{J}} \ar[rrrr]^{\tilde{f}} \ar[drr]_{\tilde{\pi}}& & & & \tilde{\mathcal{S}} \ar[dll]^{\tilde{g}} \\
& & T & &
}
\] 
of projective morphisms satisfying the following properties:
\begin{enumerate}
    \item
$\tilde{\pi}$ is a smooth family of threefolds and $\tilde g$ is a smooth family of surfaces;
    \item 
$\tilde f$ admits a section;
    \item 
for every elliptic Calabi--Yau threefold $h \colon J \rar S$ admitting a rational section and for which the minimal resolution of $S$ is an Enriques surface, there exists a closed point $t \in T$ such that $\tilde{\mathcal{J}}_t \rar \tilde{\mathcal{S}}_t$ is isomorphic to a Miranda model $\tilde J \rar \tilde S$ of $J \rar S$ as in Proposition~\ref{finite Shafarevich}; and
    \item 
there exists a reduced divisor $\tilde{\mathcal{E}} \subset \tilde{\mathcal{S}}$ that is log smooth over $T$ such that any of the isomorphisms $\tilde{\mathcal S_t} \rar \tilde{S}$, whose existence is claimed in (3), maps $\tilde{\mathcal E}_t$ (considered with its reduced structure) on the reduced exceptional locus $\tilde E$ of $\tilde{\mathcal S} \rar S$.
\end{enumerate}
\end{proposition}

\begin{proof}
Let $h \colon J \rar S$ be an elliptic Calabi--Yau threefold admitting a rational section and for which $S$ is an Enriques surface.
By Theorem~\ref{bdd enriques}, $S$ belongs to a bounded family.
Therefore, there exists $C \in \nn$ and a very ample divisor $H$ on $S$ such that $H^2 \leq C$ and $(S,\frac{1}{2}H)$ is klt.
By assumption, $h$ admits a rational section.
Then, we can apply~\cite{Fil20}*{Theorem 1.1} to the log pair $(J,\frac{1}{2}\phi^\ast H)$ -- in this case, we are taking $d=1$ with respect to the notation of~\cite{Fil20}*{Theorem 1.1}.
Furthermore, since we are considering fibrations with a rational section, their boundedness in codimension 1 actually follows from~\cite{Fil20}*{Theorem 7.8}.
In particular, by~\cite{Fil20}*{step 6 of proof of Theorem 7.8}, also the rational section of $J \rar S$ is bounded in codimension 1.
Therefore, the fibrations $f \colon J \rar S$ are bounded in codimension 1 together with a rational section.
Let $\mathcal{J} \rar \mathcal{S} \rar T$ be the family thus obtained.
Then, as the rational section of the fibrations is bounded as well, it follows that $\mathcal{J} \rar \mathcal{S}$ has a rational section which is defined over every $t \in T$.
In particular, $\mathcal{J}_\eta \rar S_\eta$ has a rational section, where $\eta$ denotes the generic point of an irreducible component of $T$.
The fact that we may assume that every fiber $\mathcal J_t$ is a Calabi--Yau threefold follows easily from Theorem~\ref{thm cones}.
To prove the statement of the claim, we will stratify and resolve the family thus obtained.

In the following, we will focus on one irreducible component of $T$ at a time, and we will possibly stratify and resolve such a component.
Since $T$ is of finite type, by Noetherian induction the following process has to be repeated only finitely many times.
By abusing notation, in the following, we will assume $T$ is irreducible.

Let $\eta$ denote the generic point of $T$.
Then, the geometric generic fiber $\mathcal{J}_{\overline{\eta}} \rar \mathcal{S}_{\overline{\eta}}$ admits a Miranda model.
Up to a finite cover of $T$, then so does $\mathcal{J}_{{\eta}} \rar \mathcal{S}_{{\eta}}$.
Thus, we may assume that $\mathcal{J}_{{\eta}} \rar \mathcal{S}_{{\eta}}$ has a birational model $\tilde{\mathcal{J}}_{{\eta}} \rar \tilde{\mathcal{S}}_{{\eta}}$ as in Proposition~\ref{finite Shafarevich}.
We denote by $\tilde{E}_\eta$ the exceptional divisor of $\tilde{\mathcal{S}}_\eta \rar \mathcal{S}_\eta$.
Up to shrinking $T$, we may assume that the generic fiber spreads out, and we obtain a tower of morphisms $\tilde{\mathcal{J}} \rar \tilde{\mathcal{S}} \rar T$, where $\tilde{\mathcal{J}}\rar T$ and $\tilde{\mathcal{S}} \rar T$ are smooth and $\tilde{\mathcal{E}} \rar T$ is log smooth.
Up to shrinking $T$, we may also assume that $\tilde{\mathcal{J}} \rar \tilde{\mathcal{S}}$ has a section.
Thus, we obtain a family of Miranda models as in Proposition~\ref{finite Shafarevich}, and the claim follows.
\end{proof}

\begin{theorem} \label{enriques 3fold bir bdd}
The set of triples
\[
\familyellcyenr
\coloneqq
\left\{
(X, S, f) \
\middle \vert \
\begin{array}{l}
\text{$X$ is a Calabi--Yau threefold, $h \colon X \rar S$ is an elliptic fibration,}
\\
\text{and the minimal resolution of $S$ is an Enriques surface}
\end{array}
\right\}
\]
is birationally bounded.
\end{theorem}

\begin{proof}
By Proposition~\ref{family Miranda models}, the set of corresponding Jacobian fibrations is birationally bounded by a family of Miranda models.
Then, the claim follows by~\cite{Gro94}*{Theorem 4.3}, since the condition on the Tate--Shafarevich group in~\cite{Gro94}*{Theorem 4.3} is guaranteed to hold by Proposition~\ref{finite Shafarevich}.
\end{proof}

We can also prove an analogue of Proposition~\ref{prop_bdd_flops} for the elliptic fibrations in $\familyellcyenr$.

\begin{proposition} \label{prop_bdd_flops enriques base}
There exist quasi-projective varieties 
$\mathcal X, \mathcal S, T$ 
and a commutative diagram
\[
\xymatrix{
\mathcal X \ar[rrrr]^f \ar[drr]_\pi& & & & \mathcal S \ar[dll]^g \\
& & T & &
}
\] 
of projective morphisms satisfying the following properties:
\begin{enumerate}
    \item
$\pi$ is a flat family of threefolds and $g$ is a flat family of surfaces;
    \item
for every $t \in T$, $\mathcal X_t$ is a Calabi--Yau threefold.
In particular, $\mathcal{X}_t$ has terminal $\mathbb Q$-factorial singularities; and 
    \item
for every terminal elliptic Calabi--Yau threefold with non-rational base $h \colon X \rar S$, there exists $t \in T$
together with an isomorphism in codimension 1 $\phi \colon X \dashrightarrow \mathcal X_t$
such that $\mathcal S_t$ and $S$ are isomorphic and $\phi$ is a birational morphism over $S$.
\end{enumerate}
\end{proposition}

\begin{proof}
The proof is identical to the one of Proposition~\ref{prop_bdd_flops}, where we replace Theorem~\ref{thm_gross} with Theorem~\ref{enriques 3fold bir bdd}. 
\end{proof}

\subsection{Threefolds of Kodaira dimension 2}

The tools developed in \S~\ref{section cone conj} can also be applied to study the boundedness of birationally bounded elliptic varieties that are not of Calabi--Yau type.
In this case, the difficulty is that dropping the Calabi--Yau condition, it may be difficult to show that flops deform, as needed in the proof of Theorem~\ref{metatheorem}.
Here, we consider minimal terminal threefolds of Kodaira dimension 2.
By~\cite{Fil20}, it is known that these are bounded in codimension 1 under certain natural and necessary geometric conditions.
In order to show the boundedness of these varieties, we rely on work of Koll\'ar and Mori, showing that the deformation of flops of $\qq$-factorial terminal threefolds is locally unobstructed, see~\cite{KM92}*{\S~11}.

\begin{theorem}
\label{T-extflop}
Let $\mathcal X  \rightarrow T$ be a flat projective family of minimal terminal $\mathbb{Q}$-factorial threefolds of Kodaira dimension 2.
Then, up to stratifying $T$ into a finite union of locally closed Zariski subsets and taking finite covers, the following holds:
\\
Let $0 \in T$ be any closed point, and let $\psi_0 \colon\mathcal X_0 \dashrightarrow \mathcal X_0^+$ be a $K_{\mathcal X_0}$-flop.
Then, there exists a $K_{\mathcal X}$-flop $\mathcal X \dashrightarrow \mathcal X^+$  over $T$ extending $\mathcal X_0 \dashrightarrow \mathcal X_0^+$. 
\end{theorem}

\begin{remark}
After the stratification, $T$ is the disjoint union of finitely many irreducible components.
Thus, the $K_{\mathcal X_0}$-flop extends over the irreducible component of $T$ containing the point $0$.
\end{remark}

\begin{proof}
Let 
\begin{align}
    \label{diag.flop.ext}
\xymatrix{
\mathcal{X}_0 \ar@{-->}[rr]^{\psi_0} \ar[dr]^{r_0} & &  \mathcal{X}_0^+ \ar[dl]^{r_0^+}\\
& Z_0 &
}
\end{align}
be a $K_{\mathcal X_0}$-flop associated to the contraction of an extremal ray $R_0 \subset \neone \mathcal X_0.$.
We now divide the rest of the proof into steps for the reader's convenience.\\

{\bf Step 0}.
{\it In this step, we make some reductions.}
\\
Up to stratifying $T$ into a union of locally closed subsets, we may assume that $T$ is smooth.
By~\cite{FM20}*{Proposition 2.9} and~\cite{Fil20}*{cf. proof of Theorem 6.1}
\begin{enumerate}
\item $\mathcal X$ is a terminal $\mathbb Q$-factorial variety;
\item there exists a commutative diagram
\[
\xymatrix{
\mathcal X \ar[rr]^f \ar[dr] & &
\mathcal S \ar[dl]
\\
& T &
}
\]such that $\mathcal S={\rm Proj}_T\ R(K_{\mathcal X})$ and $f$ is the relative Iitaka fibration $\mathcal X$ over $T$. 
In particular, $\K \mathcal X. = f^\ast H_\mathcal{S}$ where $H_\mathcal{S}$ is a $\mathbb Q$-divisor which is a relatively ample over $T$;
\item for every $t\in T$, $\mathcal S_t={\rm Proj} \ R(K_{\mathcal X_t})$;
\item if $E \subset \mathcal X_t$ is a prime $f_t$-very exceptional divisor in the sense of~\cite{Bir12}*{Definition 3.1}, then there exists a prime divisor $\mathcal E \subset \mathcal X$ horizontal over the connected component $\bar{T}$ of $T$ containing $t$ such that $\mathcal E_t=kE$ for some $k >0$;
Indeed, up to a finite cover, we may assume that all the divisors that are very exceptional for the morphism of geometric generic fibers $\mathcal{X}_{\overline{\eta}} \rar \mathcal{S}_{\overline{\eta}}$ is defined over $k(T)$.
Then, by \cite{Kol13}*{4.38}, we may assume that any such prime divisor restricts to a prime divisor fiber by fiber.
Lastly, we may shrink $T$ around $\eta$ so that, for every $t \in T$, every $f_t$-very exceptional divisor is the restriction of the closure of one of the divisors that are very exceptional for $\mathcal{X}_\eta \rar \mathcal{S}_\eta$; and,
\item 
the local systems $\mathcal{GN}^1(\mathcal X/T), \mathcal{GN}_1(\mathcal X/T)$ defined in~\cite{KM92}*{Definitions 12.2.4 and 12.2.7} are constant:
indeed, by ~\cite{KM92}*{Propositions 12.2.5 and 12.2.8} those have finite monodromy;
hence, substituting $T$ with a suitable finite cover, we can assume that their monodromy is trivial.
\end{enumerate}
Moreover, let us recall that, for a very general $t \in T$, $\mathcal{GN}_1(\mathcal X/T)\vert_{\mathcal X_t} = \nsdr (\mathcal X_t)$, see~\cite{KM92}*{Propositions 12.2.5 and 12.2.8}.
Thus, by (5), we may assume that $\nsr(\mathcal X_t)=\mathcal{GN}^1(\mathcal X/T)\vert_{\mathcal X_t}$ for a very general $t \in T$.
\\
Finally, since  $T$ has finitely many irreducible components, restricting to one of these components we may also assume that $T$ is irreducible, by Noetherian induction. \\

{\bf Step 1}.
{\it 
In this step, we show that there exists a polydisk $0 \in \Delta^k \subset T$ over which the flop $\mathcal X_0 \dashrightarrow \mathcal X_0^+$ deforms (in the analytic topology).
}
\\
By~\cite{KM92}*{Theorem 11.10}, the flop $\mathcal X_0 \dashrightarrow \mathcal X_0^+$ deforms over a germ of $0 \in T$ in the analytic topology: 
the deformation is obtained as base change of a flop of a miniversal deformation space of $\mathcal X_0$, cf.~\cite{KM92}*{Theorem 11.10}.
More precisely, as $T$ is smooth, there exists a polydisk $0 \in \Delta^k \subset T$ over which the flop of $\mathcal X_0$ deforms, that is, there exists a commutative diagram of analytic spaces
\begin{align}
\label{diag.flop.ext2}
\xymatrix{
\mathcal X_{\Delta^k}
\ar@{-->}[rrrr]^{\Psi}
\ar[ddrr] \ar[drr]^{r_{\Delta^k}}
& & & & 
\mathcal X^+_{\Delta^k} \ar[ddll] \ar[dll]_{r_{\Delta^k}^+}\\
& & \mathcal Z \ar[d] & &\\
& & \Delta^k & &
}
\end{align}
where $\mathcal X_{\Delta^k} \coloneqq \mathcal X \times_{\Delta^k} \Delta^k$, and the following properties are satisfied:
\begin{enumerate}
    \item[(a)]
the restrictions of the maps in \eqref{diag.flop.ext2} to $\mathcal X_0$ yield the diagram in \eqref{diag.flop.ext};
    \item[(b)]
$\forall t \in \Delta^k$, $\Psi\vert_{\mathcal X_t}$ is an isomorphism in codimension 1; and 
\item[(c)]
$\forall t \in \Delta^k$, $\mathcal X^+_t$ is $\mathbb Q$-factorial.
\end{enumerate}
\vspace{.3cm}

For $t \in \Delta^k$ we shall consider the restriction of the diagram in~\eqref{diag.flop.ext2} to $\mathcal X_t$
$$
\xymatrix{
\mathcal X_t \ar@{-->}[rr]^{\psi_t\coloneqq \Psi\vert_{\mathcal X_t}} 
\ar[dr]_{r_t} 
& & 
\mathcal X^+_t\ar[dl]^{r^+_t} \\
&\mathcal Z_t &
}
$$
where $\mathcal X^+_t\coloneqq  \mathcal X^+_{\Delta^k, t}$,
$\psi_t\coloneqq \Psi\vert_{\mathcal X_t}\colon \mathcal X_t \dashrightarrow \mathcal X_t^+$ denotes the induced isomorphism in codimension 1 and 
\[
r_t\coloneqq  \left (r_{\Delta^k}\middle )\right\vert_{\mathcal X_t}, 
\qquad 
r_t^+\coloneqq  \left (r_{\Delta^k}^+\middle )\right\vert_{\mathcal X_t}.
\]

{\bf Step 2}.
{\it In this step, we show that  
\begin{enumerate}
    \item[(i)]
$\mathcal X^+_t$ and $\mathcal Z_t$ are projective for $t$  general in (the analytic Zariski topology of) $\Delta^k$; and
    \item[(ii)]
for $t \in \Delta_t$ very general, any irreducible curve $C_t$ contracted by $r_t$ specializes to a curve in $\mathcal X_0$ contracted by $r_0$; in particular, $C_t \cdot K_{\mathcal X_t}=0$.
\end{enumerate}
}
Since $\Delta^k$ is open in the Euclidean topology, $\Delta^k$ contains a point $t \in T$, very general in the Zariski topology, such that $\mathcal{GN}_1(\mathcal X/T)\vert_{\mathcal X_t} = \nsdr (\mathcal X_t)$.
\begin{enumerate}
\item[(i)]
By assumption, $\mathcal X_0^+, \mathcal Z_0$ are projective; 
thus, by~\cite{KM92}*{Theorem 12.2.10}, for $t \in \Delta^k$ general, $\mathcal X_t^+,\mathcal Z_t$ are projective.
\item[(ii)]
This is just a consequence of the countability of the components the relative Douady space of $\mathcal X$ over $\mathcal Z$ together with the fact that those are proper over $\mathcal Z$, cf.~\cite{Fujiki}.
\end{enumerate}
\vspace{.3cm}

{\bf Step 3}.
{\it
In this step, we construct an effective divisor $\mathcal D$ on $\mathcal X$ such that for general $v \in \Delta^k$ (in the analytic topology), $\mathcal X^+_v$ is the ample model of both $\mathcal D_v$ and also of $K_{\mathcal X_v}+\epsilon D_v$ for any choice of $\epsilon >0$.
In particular, $\mathcal X^+_v$ is the relatively ample model of $K_{\mathcal X_v}+\epsilon D_v$ over $\mathcal{S}_v$.
}
\\
Let $t \in\Delta^k$ be a very general point for which properties (i-ii) of Step 2 hold; 
then, $K_{\mathcal{X}_t} \sim_{\mathbb Q, \mathcal Z_t} 0$, by~\cite{KM92}*{Proposition~12.1.4}.
Hence, $K_{\mathcal X_t}=r_t^\ast K_{\mathcal Z_t}$ and $K_{\mathcal X^+_t}=(r_t^+)^\ast K_{\mathcal Z_t}$, as $\psi_t$ is an isomorphism in codimension 1. 
In particular, $\psi_t$ is crepant birational with respect to $K_{\mathcal X_t}$ and $K_{\mathcal X^+_t}$ is nef.
Hence, given an ample divisor $D^+_t$ on $\mathcal X^+_t$, which exists by (i) in Step 2, $K_{\mathcal X^+_t}+\epsilon D^+_t$ is ample for all positive real numbers $\epsilon$.
We define $D_t\coloneqq (\psi^{-1}_t)_\ast D_t^+$.
Choosing $0 < \epsilon \ll 1$ and $D^+_t$ general in its $\mathbb Q$-linear series, then $(\mathcal X_t, \epsilon D_t)$ is terminal and $\psi_t$ is the outcome of the run of a $(K_{\mathcal X_t}+D_t)$-MMP.
As $\psi_t$ is crepant birational with respect to $K_{\mathcal X_t}$ and $K_{\mathcal X_t} \sim_{\mathbb Q, \mathcal S_t} 0$, then $\psi_t$ can also be obtained as a run of the relative $D_t$-MMP over $\mathcal S_t$.

Since $\nsr(\mathcal X_t)=\mathcal{GN}^1(\mathcal{X}/T)\vert_{\mathcal X_t}$, as $t$ is very general, there exists a divisor $\mathcal D$ on $\mathcal X$ such that the numerical class of $\mathcal D$ restricts to $D_t$ on $\mathcal X_t$.
Furthermore, as the identification $\nsr(\mathcal X_t)=\mathcal{GN}^1(\mathcal{X}/T)\vert_{\mathcal X_t}$ relies on the relative Hilbert scheme of $\mathcal X$ over $T$, cf.~\cite{KM92}*{proof of Proposition 12.2.5}, we may assume that $\mathcal{D}$ is itself effective, flat over $T$, and $\mathcal{D}|_{\mathcal{X}_t}=D_t$ as divisors.
In particular, we can assume that $\mathcal{D}$ does not contain any fiber.
We set $\mathcal{D}_{\Delta^k}\coloneqq \mathcal{D}|_{\mathcal{X}_{\Delta^k}}$. 
Let $\mathcal{D}^+_{\Delta^k}$ be the strict transform of $\mathcal D_{\Delta^k}$ on $\mathcal{X}^+_{\Delta_k}$;
to simplify the notation, for $v \in \Delta^k$, we set $\mathcal{D}^+_{v}\coloneqq \mathcal{D}^+_{\Delta^k}|_{\mathcal{X}_v^+}$.
For our choice of $t \in \Delta^k$, $\mathcal{D}^+_{ t}=
D_t^+$.
Hence, $\mathcal{D}^+_{ t}$ is ample on $\mathcal X^+_t$.
As ampleness is an open property, cf. \cite{KM92}*{proof of Theorem 12.2.10}, then, for $v \in \Delta^k$ general (in the analytic Zariski topology) $\mathcal{D}^+_v$ is ample on $\mathcal X^+_v$.
\\

{\bf Step 4.}
{\it In this step, we show that 
\begin{enumerate}
    \item[(A)] 
for all $v \in T$,
$\mathcal{D}_v$ (resp. $\K \mathcal X_v. + \mathcal D_v$) is big;
    \item[(B)]
there exist a positive real number $\epsilon_0$ 
such that for all $0 \leq \epsilon\leq \epsilon_0$, $(\mathcal X_v, \epsilon \mathcal D_v)$ is terminal for all $v \in T$;
    \item[(C)]
for all $v \in T$, $\mathcal{D}_v$ is movable over $\mathcal{S}_v$.
\end{enumerate}
}
\begin{enumerate}
    \item[(A)] 
By Step 3, $\mathcal{D}_v$ is big for $v \in \Delta^k$ general in the analytic Zariski topology.
By the semi-continuity theorem \cite{Har77}*{Theorem III.12.8}, for $v \in T$ very general in the Zariski topology, $h^0(\mathcal X_v, \mathcal O_{\mathcal X_v}(m\mathcal{D}_v))$ is constant for any fixed choice of $m \in \mathbb N$.
Hence, for very general $v \in T$, $\mathcal{D}_v$ is big;
finally, applying the semi-continuity theorem again, we can conclude that  $\mathcal{D}_v$ is big for all $v \in T$.
The exact same argument applies to prove the bigness of $\K \mathcal X_v. + \mathcal D_v$ for all $v \in T$.
\item[(B)] 
As for all $t \in T$, $X_t$ is terminal and $\mathbb Q$-factorial, and $\mathcal D_t$ is effective, then the conclusion simply follows by Noetherian induction on $T$, thanks to~\cite{Kol97}*{Theorem~4.8} and the fact that being terminal is an open condition in a family, see \cite{dFH11a}*{Proposition 3.5}.
\item[(C)]
By part (B) of this step and since $\K \mathcal{X}.$ is $\mathbb Q$-linearly equivalent to the pull-back of an ample divisor on $\mathcal{S}$, cf. Step 0, for all $0 < \epsilon \ll 1$, all $(\K \mathcal{X}. + \epsilon\mathcal{D})$-negative curves are contained in the fibers of $\mathcal{X} \rar \mathcal{S}$.
Hence, we may assume that for all $0<\epsilon \ll 1$, any run of the relative $(\K \mathcal{X}. + \epsilon \mathcal{D})$-MMP over $T$ with scaling of an ample divisor is also a run of the relative $(\K \mathcal{X}. + \epsilon \mathcal{D})$-MMP over $\mathcal{S}$.
Furthermore, with the same choice of $\epsilon$, $(\K \mathcal{X}. + \epsilon \mathcal{D}) \sim_{\mathbb Q, \mathcal S} \epsilon \mathcal{D}$; 
hence, the way a relative $(\K \mathcal{X}. + \epsilon \mathcal{D})$-MMP is run will be independent of $\epsilon$ for $0 < \epsilon \ll 1$.
As $\K \mathcal{X}. + \epsilon \mathcal{D}$ is big over $T$ for all positive values of $\epsilon$, then any run of the $(\K \mathcal X.+\epsilon\mathcal D)$-MMP must terminate with a good minimal model, see~\cite{BCHM}.
Thus, for $0 < \epsilon \ll 1$, let 
\begin{align}
    \label{diag.flop.ext4}
\xymatrix{
\mathcal{X} =: \mathcal X_0
\ar@{-->}[r]_{\Phi_0} \ar@/_/[drrr]
\ar@/^1.9pc/[rrrrrr]_{\Phi=\Phi_{n-1}\circ \Phi_{n-1} \circ \dots \circ \Phi_1 \circ \Phi_0}
&
\mathcal X_1
\ar@{-->}[r]_{\Phi_1}
\ar@/_/[drr]
&
\mathcal X_2
\ar@{-->}[r]_{\Phi_2} \ar[dr]
&
\dots
\ar@{-->}[r]_{\Phi_{n-3}}
&
\mathcal X_{n-2}
\ar@{-->}[r]_{\Phi_{n-2}} \ar[dl]
&
\mathcal X_{n-1}
\ar@{-->}[r]_{\Phi_{n-1}} \ar@/^/[dll]
&
\mathcal X_n=:\mathcal{X}' \ar@/^1pc/[dlll]
\\
& & &\mathcal S \ar[d] & & &
\\
& & & T & & &
}
\end{align}
be one such run of the relative $(\K \mathcal X.+\epsilon \mathcal{D})$-MMP with scaling of an ample divisor -- over $T$ or $\mathcal S$, equivalently.
We define $\mathcal D'\coloneqq \Phi_\ast \mathcal D$ and $\phi_t \coloneqq  \Phi\vert_{\mathcal X_t}$.
As $\mathcal D'$ is nef and big over $\mathcal S$ by construction, then to conclude the proof of (C) it suffices to prove the following claim.
\\

{\it Claim 1}.
For all $v \in T$, $\phi_v \colon \mathcal X_v \dashrightarrow \mathcal X'_v$ is an isomorphism in codimension 1.
\begin{proof}
Let us assume, by contradiction, that for some $v \in T$, $\phi_v \colon \mathcal X_v \dashrightarrow \mathcal X'_v$ is not an isomorphism in codimension 1.
By Steps~2-3, for any $t \in \Delta^k$ very general (in the analytic Zariski topology)
\begin{enumerate}
    \item[(I)] 
$\mathcal X^+_t$ is $\qq$-factorial, cf. Step 1, and for every $\epsilon> 0$, $\K \mathcal X^+_t.+ \epsilon \mathcal D^+_t$ ample on $\mathcal X^+_t$;
    \item[(II)]
for all $0< \epsilon \ll 1$, $(\mathcal X'_t, \epsilon \mathcal D'_t)$ is terminal and $\K \mathcal X'_t.+\epsilon \mathcal D'_t$ is big and semi-ample on $\mathcal X'_t$; and,
    \item[(III)]
$\phi_t \colon \mathcal X_t \dashrightarrow \mathcal X'_t$ is an isomorphism in codimension 1, by~\cite{HMX18}*{Lemma~3.2}:
indeed, with reference to the statement of~\cite{HMX18}*{Lemma~3.2}, it suffices to take 
\begin{itemize}
    \item 
$X \coloneqq \mathcal X$, $U \coloneqq T$, $0 \coloneqq t$;
    \item 
$\Delta \coloneqq \epsilon \mathcal D$ for $0< \epsilon \ll 1$ such that $(\mathcal X, \Delta)$ is terminal, the same holds for $(\mathcal X_t, \Delta_t)$ so that property $(2)$ in the statement of the lemma is satisfied. 
Furthermore, for such choice of $\epsilon$, the $\mathbb Q$-linear system of $\K \mathcal X_t.+ \Delta_t$ is movable by (I), thus property $(3)$ in the statement of the lemma is satisfied; and 
    \item 
$D_1, \dots, D_{\dim T}$, $\dim T$ sufficiently general effective divisors meeting transversely at $0$ .
\end{itemize}
\end{enumerate}
As the indeterminacy locus of $\Phi$ is Zariski closed and its exceptional locus is locally closed in the Zariski topology, it follows that $\phi_u$ must be is an isomorphism in codimension 1 for $u \in T$ general.
On the other hand, as $\phi_v$ is a birational map over $\mathcal{S}_v$, any divisor contracted by $\phi_v$ is very exceptional with respect to $f_v \colon \mathcal{X}_v \rar \mathcal{S}_v$.
By condition (4) in Step 0, there exists a prime divisor $\mathcal E \subset \mathcal X$ horizontal over $T$ such that $\mathcal E_t=kE$ for some $k > 0$.
It suffices to show that $\mathcal E$ must be contained in the exceptional locus of $\Phi$ to obtain the sought contradiction.
But, if that were not the case, there would exist an integer $i \in \{0, 1, \dots, n-1\}$ such that for the extremal contraction $\xi_i \colon \mathcal X_i \to \mathcal Z_i$ in the $i$-th step of the $(\K \mathcal X.+\mathcal D)$-MMP in~\eqref{diag.flop.ext4}, then $\dim \xi_i(\mathcal E_i)_t=2$ for $t \in T$ general, whereas $\dim \xi(\mathcal E_i)_v=1$; here, $\mathcal E_i$ is the strict transform of $\mathcal E$ on $\mathcal X_i$.
This is clearly impossible, by the upper semi-continuity of fiber dimension.
\end{proof}
\end{enumerate}

{\bf Step 5}.
{\it
In this step, we show that there exists a positive real number $\epsilon_1$ such that for all $0 < \epsilon \leq \epsilon_1$, $\K \mathcal X'_v.+ \epsilon \mathcal{D}'_v$ is big and semi-ample on $\mathcal X'_v$ for all $v \in T$, where the model $\mathcal{X}'$ is the one constructed in \eqref{diag.flop.ext4}.
}
\\
The MMP in~\eqref{diag.flop.ext4} terminates with a good minimal model $\mathcal X'$ over $\mathcal S$. 
Thus, for all $u \in T$, $\mathcal{D}'_u$ is big and semi-ample over $\mathcal{S}_u$.
As $\K \mathcal X.$ is the pull-back of a $\mathbb Q$-divisor on $\mathcal S$ ample over $T$, then $\K \mathcal{X}'_u.+ \epsilon \mathcal{D}'_u$ is big and semi-ample on $\mathcal X'_u$.
\newline

{\bf Step 6.}
{\it 
In this step, we show that $\mathcal{D}_{0}^+$ is ample over $\mathcal{S}_0$.}
\\
We first show that $\mathcal{D}_{0}^+$ is nef over $\mathcal{S}_0$.
Let us assume by contradiction that this is not the case.
Thus, there must exist a $\mathcal{D}^+_0$-negative curve $\Gamma_0^+$ vertical over $\mathcal{S}_0$ spanning an extremal ray of the effective cone of curves on $\mathcal X_0^+$.
As $\mathcal K_{\mathcal X_0^+}\sim_{\mathbb Q, \mathcal S_0}0$, then $\Gamma_0^+$ is also a $(\K \mathcal{X}_0^+. + \mathcal{D}_0^+)$-negative curve.
Since in Step 4 we showed that $\mathcal D_0$ is movable over $\mathcal{S}_0$, then the same conclusion must hold for $\mathcal D^+_{0}$ over $\mathcal S_0$, as $\psi_0$ is an isomorphism in codimension 1 over $\mathcal S_0$, by construction.
Hence, the contraction of $\Gamma_0^+$ gives rise to a relative $(\K \mathcal{X}_0^+. + \mathcal{D}_0^+)$-flipping contraction $\mu_0 \colon \mathcal X_0^+ \rar \mathcal Y_0$ over $\mathcal S_0$ which is also a relative $\K \mathcal{X}_0^+.$-flop.
Then, \cite{KM92}*{Corollary 11.11} implies that up to possibly passing to a positive multiple, for very general $t \in \Delta^k$ (in the analytic Zariski topology) there exists a curve $C_t \subset \mathcal X^+_t$ specializing to a curve supported on the exceptional locus of $\mu_0$.
But then, for some irreducible component $C^i_t$ of $C_t$, $\mathcal D^+_t \cdot C_t^i<0$:
this yields the desired contradiction as $\mathcal{D}_t^+$ is ample for $t \in T$ general by construction, cf. Step~3.

The same reasoning, utilizing~\cite{KM92}*{Corollary~11.10}, shows that if $\mathcal D_{0}^+$ is nef, then the effective cone of curves of $\mathcal X^+_0$ does not admit a $\mathcal D_{0}^+$-trivial extremal ray over $\mathcal S_0$.
Hence, $\mathcal D_{0}^+$ is ample over $\mathcal S_0$.\\

{\bf Step 7}.
{\it In this step, we show that there exists an analytic Zariski open neighborhood $U' \subset \Delta^k$ of $0$ and an isomorphism over $U'$ between the restrictions of $\mathcal X'\vert_{U'}$ and $\mathcal X^+_{\Delta^k}\vert_{U'}$}.
\\
As $\mathcal D^+_{0}$ is ample over $\mathcal S_0$ and $\K \mathcal X^+_0.=f_0^\ast H_{\mathcal S_0}$, with $H_{\mathcal S_0}$ ample, then for all $0< \epsilon \ll 1$, $\K \mathcal X^+_0. + \epsilon \mathcal D^+_{0}$ is ample on $\mathcal X^+_0$.
By the openness of ampleness, then $\K \mathcal X^+_v. + \epsilon \mathcal D^+_{v}$ is ample on $\mathcal X^+_v$ for $v \in \Delta^k$ general (for the analytic Zariski topology) and $\mathcal X^+_{\Delta^k} \rar \Delta^k$ is projective over an analytic Zariski open neighborhood $U \subset \Delta^k$ of $0$.
Hence, over $U$ we have the following commutative diagram
\[
\xymatrix{
& \mathcal X \vert_U \ar@{-->}[dr]^{\Psi\vert_U} & 
\\
\mathcal X'\vert_U \ar@{-->}[ur]^{\Phi^{-1}\vert_{U}}
\ar@{-->}[rr]^{\Xi_U}
\ar[dr]
& & 
\mathcal X^+_{\Delta^k}\vert_U 
\ar[dl]
& & 
\Xi_U\coloneqq  (\Psi\vert_U) \circ (\Phi^{-1}\vert_{U}).
\\
& U &
}
\]
By Claim 1 and the construction of $\Psi$, cf. Step~1, $\Xi_U\vert_{\mathcal X'_0}$ is an isomorphism in codimension 1.
As 
$(\Xi_U\vert_{\mathcal X'_0}^{-1})_\ast( \K \mathcal X^+_0.+ \mathcal D^+_0)=
\K \mathcal X'_0.+ \mathcal D'_0$, $\K \mathcal X'_0.+ \mathcal D'_0$ 
is big and semi-ample on $\mathcal X'_0$, and $\mathcal X^+_0$ is $\mathbb Q$-factorial, Lemma~\ref{lem.kaw.type} implies that
$\K \mathcal X'_0.+ \mathcal D'_0$ 
is ample and that $\Xi_U\vert_{\mathcal X'_0}$ is an isomorphism.
Thus, $\K \mathcal X'.+ \mathcal D'$ and $\K \mathcal X^+_{\Delta^k}.+ \mathcal D^+_{\Delta^k}$ are both ample over a common analytic Zariski open $U' \subset U$. 
Hence, $(\Xi_U)\vert_{U'}$ is an isomorphism over $U'$ since $(\Xi_U)_\ast (\K \mathcal X'.+ \mathcal D')\vert_U = 
(\K \mathcal X^+_{\Delta^k}.+ \mathcal D^+_{\Delta^k})\vert_U$.
\\

{\bf Step 8}.
{\it Conclusion of the proof}.
\\
To conclude the proof we just need to show that $\Phi$ is just given by a $\K \mathcal X.$-flop over $T$.
\\
As mentioned at the start of~\cite{KM92}*{proof of Theorem~12.2.10}, the restriction to $\mathcal X_0$ induces a natural injection, cf.~\cite{KM92}*{Proposition~12.2.6},
\[
\mathfrak r_0 \colon 
\neone \mathcal X/ T. 
\hookrightarrow
\neone \mathcal X_0. . 
\]
Recall that $\psi_0 \colon \mathcal X_0 \dashrightarrow \mathcal X^+_0$ is the flop of an extremal ray $R_0 \subset \overline{{\rm NE}(\mathcal X_0)}$.
Moreover, \cite{KM92}*{Corollary 12.3.3} implies that $R_0 \subset Im(\mathfrak r_0)$.
We set $\widetilde{R}_0:=\mathfrak r_0^{-1}(R_0)$.
Then, $\widetilde{R}_0$ is an extremal ray in $\overline{{\rm NE}(\mathcal X/ T)}$. 
Moreover, by specialization, 
$\K \mathcal X. \cdot \widetilde{R}_0 =0$ and 
$\mathcal D \cdot \widetilde{R}_0 < 0$.
Since $\K \mathcal X.+\epsilon \mathcal D$ is big, by Step 4, then there exists the flop of $\widetilde R_0$ over $T$,
\begin{align}
     \label{diag.flop.ext5}
\xymatrix{
\mathcal X \ar@{-->}[rr]^\Upsilon \ar[dr]_{s} 
& &
\mathcal X'' \ar[dl]^{s''} 
\\
& \mathcal Z'.
& 
}
\end{align}
Since $\mathfrak r_0(\widetilde R_0)=R_0$, then the curves contracted by $\mathcal X_0 \rar \mathcal Z'_0$ are all the curves in $R_0$.
By construction then, denoting $s_0:=s\vert_{\mathcal X_0}$, $s''_0:=s\vert_{\mathcal X''_0}$ $s_0=r_0$ and, by the uniqueness of flops, it follows then also that $s''_0=r_0^+$. 
Hence, the flop $\Upsilon$ in~\eqref{diag.flop.ext5} yields a flop of $\mathcal X$ lifting the flop $\psi_0$ in~\eqref{diag.flop.ext} on $\mathcal X_0$ as claimed in the statement of the theorem.

To conclude, we argue that $\mathcal{X}''=\mathcal{X}'$ and $\Upsilon=\Phi$.
Let $\mathcal{D}''$ denote the strict transform of $\mathcal{D}$ to $\mathcal{X}''$, and let $\mathcal{D}''_0$ denote its restriction to $\mathcal{X}''_0$.
By construction, we have $\upsilon_0=\psi_0=\phi_0$, where $\upsilon_0$ denotes the restriction of $\Upsilon$ to the special fiber.
Then, as $\mathcal{X}''$ is $\mathbb Q$-factorial and ampleness is an open condition, $\mathcal{D}''$ is ample over an open neighborhood of $0$.
Thus, $\mathcal{X}''=\mathcal{X}'$ and $\Upsilon=\Phi$ hold true over a non-empty open set of $T$ containing $0$.

By construction, $\mathcal{X}'$ and $\mathcal{X}''$ are $\mathbb Q$-factorial and isomorphic in codimension 1 to $\mathcal{X}$.
Thus, $\mathcal{X}'$ and $\mathcal{X}''$ are connected by a sequence of flops, which is also a sequence of flips for a suitable pair structure.
As $\mathcal{X}''=\mathcal{X}'$ holds true generically over the base, these flips are concentrated over a proper closed subset $G$ of $T$.
Yet, the results of \cite{KM92}*{\S~11 and \S~12} imply that extremal curves deform over the base under our assumptions.
Thus, it follows that $G= \emptyset$ and $\mathcal{X}''=\mathcal{X}'$ and $\Upsilon=\Phi$.
\end{proof}

The following is an immediate corollary of Theorem~\ref{T-extflop}.

\begin{corollary}
\label{C-extflop}
Let $\mathcal X  \rightarrow T$ be a flat projective family of minimal terminal $\mathbb{Q}$-factorial threefolds with Kodaira dimension 2.
Then, up to stratifying $T$ into a finite union of locally closed Zariski subsets and taking finite covers, the following holds:
\\
Let $0 \in T$ be any closed point, and let $\psi_0 \colon\mathcal X_0 \dashrightarrow \mathcal X_0^+$ be a finite sequence of $K_{\mathcal X_0}$-flops.
Then, there exists a finite sequence of $K_{\mathcal X}$-flops $\mathcal X \dashrightarrow \mathcal X^+$  over $T$ extending $\mathcal X_0 \dashrightarrow \mathcal X_0^+$.
\end{corollary}

In the proof of Theorem~\ref{T-extflop} we used the following easy consequence of~\cite{Kaw97}*{Lemma~1.5}.

\begin{lemma}
\label{lem.kaw.type}
Let $(Y_1, D_1)$, $(Y_2, D_2)$ be projective klt pairs.
Assume that 
\begin{enumerate}
    \item 
$\K Y_1.+D_1$ is ample and $\K Y_2.+D_2$ is nef and big;
    \item
$Y_1$ is $\mathbb Q$-factorial;
    \item 
there exists a birational map 
$\lambda \colon Y_1 \dashrightarrow Y_2$ which is an 
isomorphism in codimension 1; and
    \item 
$\lambda_\ast(\K Y_1.+D_1)=\K Y_2.+D_2$.
\end{enumerate}
Then $\lambda$ is an isomorphism.
\end{lemma}

\begin{proof}
Let 
$\tau \colon Y_3 \rar Y_2$
be a $\mathbb Q$-factorialization of $Y_2$.
We set 
$D_3\coloneqq \tau^{-1}_\ast D_2$;
thus, 
$\K Y_3.+D_3=\tau^\ast (\K Y_2.+D_2)$ so that $\K Y_3.+D_3$ is big and nef.
Then,
$\eta\coloneqq 
\tau^{-1} \circ \lambda \colon 
Y_1 \dashrightarrow Y_3$
is an isomorphism in codimension 1 of $\mathbb Q$-factorial varieties such that
$\eta_\ast (\K Y_1.+D_1) = \K Y_3.+D_3$.
But then,~\cite{Kaw97}*{Lemma~1.5} implies that $\eta$ is an isomorphism since the interior of ${\rm Nef}(Y_3)$ and of $\eta_\ast {\rm Nef}(Y_1)$ have non-empty intersection.
Hence, since 
$\lambda = \tau \circ \eta$, then $\lambda$ is a morphism and 
$\K Y_1.+D_1=\lambda^\ast(\K Y_2.+D_2)$.
As $\K Y_1.+D_1$ is ample on $Y_1$, then $\lambda$ can only be an isomorphism.
\end{proof}

\subsection{Towards progress in higher dimension}

The following result is the main technical result in the proof of the boundedness of elliptic Calabi--Yau threefolds.
It shows how the results of \S~\ref{section cone conj} can be used to prove the boundedness of certain elliptically fibered varieties once we know that they are bounded up to flops over the base.

\begin{theorem} 
\label{metatheorem}
Fix a positive natural number $d$.
Let $\mathfrak{F}$ be a set of triples $((X, 0), (Y, 0), f)$.
Let
$
\mathfrak B \coloneqq
\left\{
Y \
\middle \vert \
\exists (X, Y, f)\in \mathfrak F
\right\}
$.
Assume that all triples
$(X, Y, f) \in \mathfrak F$ satisfy the following properties:
\begin{itemize}
    \item 
$X$ is a  projective terminal $\qq$-factorial variety of dimension $d$;
    \item 
$h^1(X,\O X.)=h^2(X, \O X.)=0$; and
    \item 
$f \colon X \rar Y$ is a relatively minimal elliptic fibration.
\end{itemize}
If $\mathfrak F$ is bounded in codimension 1 and $\mathfrak B$ is bounded, then $\mathfrak{F}$ is bounded.
\end{theorem}

The assumption $h^1(X,\O X.)=0=h^2(X, \O X.)$ in Theorem~\ref{metatheorem} is needed to apply the results in \S~\ref{section deform}, which allow extending flops from a special fiber to the whole family.

\begin{proof}
Let $(X, Y, f) \in \mathfrak F$.
As $Y$ belongs to the bounded set $\mathfrak B$, there exist $v \in \nn$ and a very ample divisor $H_Y$ on $Y$ such that $H_Y^{\dim Y} \leq v$ and $(Y,\frac{1}{2}H_Y)$ is klt.
Such choice of $v$ is independent of $Y \in \mathfrak B$.
Up to replacing $H_Y$ with a multiple only depending on $\mathfrak F$, by the boundedness of the extremal rays, we can assume that $\K X. + \frac{1}{2}f^*H$ is semi-ample with $\kappa (X,\frac{1}{2}f^*H)=n-1$.
Then, arguing as in Step 0 of the proof of Theorem~\ref{T-extflop}, up to a stratification of a family bounding $\mathfrak F$ in codimension 1, we may partition $\mathfrak F$ into a finite number of classes such that $h^0(\mathcal O(m(\K X.+\frac{1}{2}f^*H)))$ only depends on $m$ sufficiently divisible for all $X$ in one of the given classes partitioning $\mathcal F$.
In particular, as the stratification is finite, $\vol(X,\K X. + \frac{1}{2}f^*H)$ can only attain finitely many values.
Moreover, since, by assumption, $\mathfrak F$ 
is birationally bounded, then there exists a positive integer 
$C$, 
independent of the triple 
$(X, Y, f)$, 
such that $f$ admits a rational $l$-section, for some $d \leq C$. 
Thus, we can apply~\cite{Fil20}*{Theorem 1.1} to deduce that the set of pairs
\begin{align*}
\left\{
(X,\frac{1}{2}f^\ast H_Y) \
\middle \vert 
(X, Y, f) \in \mathfrak F \
\text{and $H_Y$ is the very ample divisor on $Y$ constructed above}
\right \}
\end{align*}
is log bounded.
Even better,~\cite{Fil20}*{Theorem 1.1} implies that
there exist quasi-projective varieties 
$\mathcal X, \mathcal Y, T$ 
and a commutative diagram
\begin{align}
\label{diag.meta.thm}
\xymatrix{
\mathcal X \ar[rrrr]^f \ar[drr]_\pi& & & & \mathcal Y \ar[dll]^g \\
& & T & &
}
\end{align}
of projective morphisms such that for any triple $(X, Y, f)$ in $\mathfrak{F}$ there exists a closed point $t \in T$ such that 
\begin{enumerate}
    \item 
$Y \simeq \mathcal{Y}_t$;     \item 
$X$ and $\mathcal{X}_t$ are connected by a sequence of flops over $Y=\mathcal{Y}_t$.
\end{enumerate}

Up to passing to a stratification and an \'{e}tale base change of the original parameter space $T$, we may assume that Theorem \ref{thm KM conj general} and Theorem \ref{thm cones} apply to the pull-back of the morphisms in~\eqref{diag.meta.thm} to each of the finitely many irreducible components of $T$.
Furthermore, we may assume that each irreducible component of $T$ is affine.
As there are finitely many of these components, in the following we focus on a single one, with the understanding that the same argument has to be repeated on each one of them individually.
By abusing notation, we will denote this irreducible component by $T$.

By Theorem \ref{thm KM conj general}, $\mathcal{X} \rar \mathcal{Y}$ admits finitely many minimal models $\mathcal{X}_1, \ldots , \mathcal{X}_k$ over $\mathcal{Y}$, up to isomorphism over $\mathcal Y$.
For any $(X, Y, f) \in \mathfrak F$,
there exist $t \in T$ and an isomorphism in codimension 1 $\psi \colon X \dashrightarrow \mathcal{X}_t$ which can be factored into a sequence of flops over $Y=\mathcal{Y}_t$.
The cones $\overline{M}(\mathcal{X}_t)$ and $\overline{M}(X)$ are naturally identified by $\psi_\ast$ and the same holds also for $\overline{M}(\mathcal{X}_t/\mathcal{Y}_t)$ and $\overline{M}(X/Y)$.
Then, there exists a class $D_t \in \overline{M}(\mathcal{X}_t)$ such that the rational map $\psi^{-1} \colon \mathcal{X}_t \drar X$ is a $D_t$-MMP over $\mathcal{Y}_t$.
Furthermore, we may assume that $D_t$ lies in the interior of both $\overline{M}(\mathcal{X}_t)$ and of $\psi^{-1}_\ast(\overline{A}(X/Y))$ in the decomposition of $\overline{M}(\mathcal{X}_t/\mathcal{Y}_t)$.
By Theorem \ref{thm cones}, there exists $\mathcal D \in \overline{M}(\mathcal{X}/T)$ such that $\mathcal D\vert \subs \mathcal{X}_t.=D_t$.
Let $\Phi \colon \mathcal{X} \drar \mathcal{X}'$ be a $\mathcal D$-MMP over $\mathcal{Y}$.
By Theorem \ref{thm KM conj general}, there is $1 \leq i \leq k$ such that $\mathcal{X}'$ and $\mathcal{X}_i$ are isomorphic over $\mathcal{Y}$.

Then, by \cite{HMX18}*{\S~3}, the $\mathcal D$-MMP $\Phi$ above restricts to a $D_t$-negative birational map $\Phi\vert_{\mathcal X_t} \colon \mathcal X_t \dashrightarrow \mathcal X'_t$ on the fiber $\mathcal{X}_t$ which can be factored into a sequence of small $K_{\mathcal X_t}$-trivial birational maps.
By Remark \ref{rmk sections} and the fact that $D_t$ is in the interior of $\overline{M}(\mathcal{X}_t)$, then, up to rescaling by a positive rational number, there exists $\mathcal D$ such that $(\mathcal{X}, \mathcal D)$ and $(\mathcal{X}_t, D_t)$ are both terminal.
Thus, $\mathcal{X}'_t$ is itself terminal and it is connected to $\mathcal{X}_t$ by a sequence of flops.
Since 
$D_t \in \psi_\ast(\overline{A}(X/{Y})) \cap (\Phi\vert_{\mathcal X_t})_{ \ast}(\overline{A}(\mathcal{X}'_t/\mathcal{Y}_t))$, and it is in the interior of $\psi_\ast(\overline{A}(X/{Y}))$, then by \cite{Kaw97}*{Lemma 1.5} $\mathcal{X}'_t$ and $X$ are isomorphic.
In turn, $\mathcal{X}_{i,t}$ and $X$ are also isomorphic and since there are only finitely many models $\mathcal{X}_1, \ldots , \mathcal{X}_k$, the claim follows.
\end{proof}

\begin{remark}
While Theorem \ref{metatheorem} is stated in the setting of elliptically fibered varieties, its underlying philosophy is quite general.
In fact, as soon as we have boundedness modulo flops for a set of $K$-trivial varieties (resp. a set of Calabi--Yau fiber spaces), one could try to prove the Kawamata--Morrison cone conjecture for that particular situation.
If that is successful and one can argue that flops can be extended from a fiber to the total space, then the corresponding analog of Theorem \ref{metatheorem} would follow.
\end{remark}

As an immediate corollary of Theorem~\ref{metatheorem} we are able to bound a large class of elliptic Calabi--Yau fibrations in higher dimension.

\begin{corollary}
\label{cor.bound.cy.high.dim}
Fix positive integers $d, C$.
Let $\familyellfano$ be the set of triples
\begin{align*}
\familyellfano \coloneqq
\left\{
(X, Y, f) \
\middle \vert \ 
\begin{array}{l} 
\text{$X$ is a terminal projective Calabi--Yau variety of dimension $d$,}  
\\
\text{$f \colon X \to Y$ is an elliptic fibration admitting a rational $l$-section}\\
\text{of degree $l\leq C$, and $Y$ is a log Fano variety}.
\end{array}
\right\}
\end{align*}
Then $\familyellfano$ is bounded.
\end{corollary}

\begin{proof}
Let $\logfanoeps$ be the set of varieties
\begin{align*}
\logfanoeps \coloneqq
\left\{
Y \
\middle \vert \ 
\begin{array}{l}
\dim Y= d-1, \ \text{and there exists an effective divisor $\Delta$ on $Y$}
\\
\text{such that $(Y, \Delta)$ is $\epsilon$-klt}, \
K_Y+\Delta_Y \sim_\mathbb{Q} 0,\
\text{$\Delta$ is big}
\end{array}
\right\}.
\end{align*}
By~\cite{Bir21}*{Theorem~1.4}, for any fixed real number $\epsilon > 0$, $\logfanoeps$ is log bounded.

Let $(X, Y, f) \in \familyellfano$.
By Proposition~\ref{prop models} and Remark~\ref{rmk_surfaces}, given an elliptic Calabi--Yau variety $f \colon X \rar Y$, there exists a boundary $\Delta_Y$ on $Y$ with coefficients in $\celliptic$ such that $(Y, \Delta_Y)$ is klt and $K_Y+\Delta_Y\sim_{\mathbb Q}0$.
By~\cite{HMX14}*{Theorem~1.5}, there exists a positive real number $\epsilon_0 $ such that $(Y, \Delta_Y)$ is $\epsilon_0$-klt.
As $Y$ is log Fano, then, $\Delta_Y$ is big and, hence, $Y \in \logfanoepszero$.
Hence, the set $\setbaseslf$ of varieties
\begin{align*}
\setbaseslf\coloneqq
\left\{
Y \
\middle \vert \
\exists (X, Y, f) \in \familyellfano
\right \}
\end{align*}
is bounded.

On the other hand, $\familyellfano$ is bounded in codimension 1 by~\cite{Fil20}*{Theorem~7.2}.
Hence, we can apply Theorem~\ref{metatheorem} with $\mathfrak F=\familyellfano$ and $\mathfrak B=\setbaseslf$.
\end{proof}

\section{Proof of the main results}

\begin{proof}[Proof of Theorem \ref{thm intro}]
This follows immediately from Proposition \ref{prop_bdd_flops}, Proposition \ref{prop_bdd_flops enriques base}, and Theorem \ref{metatheorem}. 
\end{proof}

\begin{remark}
Theorem \ref{metatheorem} can be used to deduce analogs of Theorem \ref{thm intro} in higher dimension.
So far, there have been several results addressing the boundedness in codimension 1 of elliptic Calabi--Yau varieties admitting a rational section in any dimension, see \cites{BDCS,FS, dCS17}.
Unfortunately, for $n \geq 4$, the current state of the art regarding elliptic Calabi--Yau $n$-folds $f \colon X \rar Y$ can only guarantee that $Y$ is bounded in codimension 1.
Once we are able to address the actual boundedness of the set of bases $Y$, the statements in \cites{BDCS,dCS17} could be enhanced to full boundedness using the tools here discussed.
\end{remark}

\begin{proof}[Proof of Corollary~\ref{cor.ell.top}]
This follows immediately from Theorem~\ref{thm intro} and Verdier’s generalization of Ehresmann’s theorem \cite{Ver76}*{Corollaire~5.1}.
 \end{proof}

\begin{proof}[Proof of Theorem \ref{thm intro2}]
By \cite{Fil20}*{Theorem 1}, $\familykodtwo$ is bounded up to flops.
Moreover, by~\cite{Fil18}*{Theorem~1.14}
the set of varieties
\[
\left\{
S \ 
\middle \vert \ 
\exists (X, S, f) \in \familykodtwo
\right\}
\]
is bounded, cf. Proposition~\ref{prop models}.
Thus, there exist $N \in \nn$ and a very ample divisor effective $H_S$ on $S$ such that $H_S^{2} = Nv$ and $(Y,\frac{1}{2}H_Y)$ is klt.
Such choice of $N$ is independent of the triple $(X, S, f) \in \familykodtwo$.
As we are assuming the existence of a degree $C$ rational section of $f$, we can apply~\cite{Fil20}*{Theorem 1.1} to deduce that the set of pairs
\[
\left\{
(X,\frac{1}{2}f^\ast H_S) \
\middle \vert 
(X, S, f) \in \mathfrak F \
\text{and $H_Y$ is the very ample divisor on $Y$ constructed above}
\right \}
\]
is log bounded in codimension 1.
Even better,~\cite{Fil20}*{Theorem 1.1} implies that there exist quasi-projective varieties $\mathcal X, \mathcal S, T$ and a commutative diagram of projective morphisms
\[
\xymatrix{
\mathcal X 
\ar[drr]^{\pi} \ar[rrrr]^{\Phi}
& & & & 
\mathcal S \ar[dll]^g
\\
& & T & &
}
\] 
such that for any $(X, S, f) \in \familykodtwo$ there exists $t \in T$ such that $S\cong \mathcal S_t$ and $f$ is birational to $\Phi _t\colon\mathcal X _t\to \mathcal S _t$ over $S$ for some $t\in T$. Note that $X\dasharrow \mathcal X _t$ is given by a sequence of flops over $S=\mathcal S_t$.
Moreover, by \cite{HX15}*{Proposition 2.4}, \cite{dFH11a}*{Proposition 3.5}, and \cite{KM92}*{Theorem 12.1.10},
we may assume that $\mathcal X$ is $\qq$-factorial and all fibers $\mathcal X_t$ are terminal $\qq$-factorial projective varieties;
furthermore, up to an additional stratification, we may assume that all fibers are varieties of Kodaira dimension 2 and $\mathcal S_t \cong {\rm Proj}(R(K_{\mathcal X_t}))$ for every $t\in T$, see \cite{Fil20}*{cf. proof of Theorem 6.1}.

By Theorem~\ref{thm KM conj general}, there exist $k \in \mathbb Z_{>0}$ and finitely many marked minimal models $\phi _i\colon \mathcal X\dasharrow \mathcal X_i$ over $\mathcal S$, for $i \in \{ 1, 2, \dots,  k\}$, such that for any $\mathbb Q$-divisor $\mathcal D$ pseudo-effective over $\mathcal S$ and any minimal model $\phi\colon\mathcal X\dasharrow \mathcal X'$ over $\mathcal S$, there exists $1\leq i\leq k$ for which $\phi _i$ is birationally equivalent  over $\mathcal S$ to $\phi$. In particular, if $\phi\colon\mathcal X\dasharrow \mathcal X'$ is a sequence of flops, then up to isomorphism over $\mathcal S$, $\phi=\phi_i$ for some $1\leq i\leq k$.
Let $X\in \familykodtwo$ and $\mathcal X_t\dasharrow X$ the sequence of flops mentioned above. By Theorem \ref{T-extflop}, we may assume that this extends to a sequence of flops $\mathcal X \dasharrow \mathcal X'$ and there exists a birational isomorphism $\mathcal X'\cong \mathcal X _i$ over $\mathcal S$ for some $1\leq i\leq k$, hence $X\cong \mathcal X _{i,t}$.
 \end{proof}

\end{document}